\documentclass[11pt,reqno]{amsart}
\usepackage{graphicx}
\usepackage{indentfirst,csquotes}

\usepackage{amssymb,amsthm,amsmath,mathrsfs}
\usepackage{xcolor,paralist,hyperref,fancyhdr,etoolbox}
\newtheorem{theorem}{Theorem}[section]
\newtheorem{definition}[theorem]{Definition}

\newtheorem{lemma}[theorem]{Lemma}
\newtheorem{proposition}[theorem]{Proposition}
\newtheorem{corollary}[theorem]{Corollary}

\newtheorem{problem}[theorem]{Problem}
\newcommand{\Rn}{\mathbb{R}^n}
\newcommand{\R}{\mathbb{R}}
\newcommand{\sphere}{S^{n-1}}
\newcommand{\ball}{B^n_2} 
\newcommand{\dom}{\mathrm{dom~}}
\newcommand{\gen}{\Psi_{\alpha}}
\newcommand{\CA}{\mathcal{C}^+_{\alpha}}

\newcommand{\Ja}{J}
\newcommand{\JD}{\delta J_{\alpha} }
\newcommand{\conv}{\mathrm{Conv}}
\newcommand{\dpsi}{D_{\psi}}
\newcommand{\dwpsi}{D_{\widetilde\psi}}
\newcommand\numberthis{\addtocounter{equation}{1}\tag{\theequation}}
\newcommand{\mres}{\mathbin{\vrule height 1.6ex depth 0pt width 0.13ex\vrule height 0.13ex depth 0pt width 1.3ex}}

\numberwithin{equation}{section}

\hypersetup{ colorlinks=true, linkcolor=blue, filecolor=black, urlcolor=black }

\begin{document}
\title[A Minkowski problem for  $\alpha$-concave functions via optimal transport]{A Minkowski problem for  $\alpha$-concave functions via optimal transport} 
\author[Li]{Xiao Li} \author[Nguyen]{Nguyen Dac Khoi Nguyen}   \author[Ye]{Deping Ye}
\date{\today}

\address{Xiao Li, School of Mathematical Science, Chongqing Normal University, Chongqing, 401331, China}
\email{lxlixiaolx@163.com}
\address{Nguyen Dac Khoi Nguyen, Department of Mathematics and Statistics, Memorial University of Newfoundland, St.
John’s, Newfoundland, A1C 5S7, Canada}
\email{ndknguyen@mun.ca}
\address{Deping Ye, Department of Mathematics and Statistics, Memorial University of Newfoundland, St.
John’s, Newfoundland, A1C 5S7, Canada}
\email{deping.ye@mun.ca}

\makeatletter
\newcommand*\bigcdot{\mathpalette\bigcdot@{0.7}}
\newcommand*\bigcdot@[2]{\mathbin{\vcenter{\hbox{\scalebox{#2}{$\m@th#1\bullet$}}}}}
\makeatother

 \let\thefootnote\relax

\begin{abstract} The notions of the Euclidean surface area measure and the spherical surface area measure of $\alpha$-concave functions in $\Rn$,  with $-\frac{1}{n}<\alpha<0$, are introduced via a first variation of the total mass functional with respect to the $\alpha$-sum operation. Subsequently, these notions are extended to  those for $\alpha$-concave measures. We then study the Minkowski problem associated with the Euclidean surface area measures of $\alpha$-concave measures via optimal transport.

\end{abstract}
\maketitle

\section{Introduction}
To begin our discussion, let $-\infty< \alpha< +\infty$ be a fixed constant. A function $f:\Rn\to [0,+\infty)$ is said to be \textit{$\alpha$-concave} if $\mathrm{Supp}(f)$, the support  of $f$, is convex, and for any $x,y\in \mathrm{Supp}(f)$ and $\lambda\in [0,1]$,
\begin{align}
   f(\lambda x+(1-\lambda)y)&\geq \left( \lambda f(x)^{\alpha}+(1-\lambda)f(y)^{\alpha}\right)^{\frac{1}{\alpha}}, \ \ \mathrm{for} \ \ \alpha\neq 0; \label{a-concave-function-1} \\ f\left(\lambda x+(1-\lambda)y\right) &\geq f(x)^\lambda f(y)^{1-\lambda}, \ \ \mathrm{for} \ \ \alpha =0. \label{log-concave-function-1} 
\end{align}
The notion of $\alpha$-concave functions may be referred to, e.g., Avriel \cite{A72}, Borell \cite{B74,B75}, and Brascamp and Lieb \cite{BL76}. Functions that satisfy \eqref {log-concave-function-1} are called \textit{log-concave functions} because their logarithms are concave.     Note that if $f$  is $\alpha_1$-concave, then it is also $\alpha_2$-concave for all $\alpha_2<\alpha_{1}$, see e.g., \cite{BL76}. In particular, any $\alpha_{1}$-concave function for any $\alpha_{1}\geq0$ must be $\alpha_{2}$-concave for any $\alpha_{2}\leq 0$. It can be checked that the characteristic function $\mathbf{1}_K$ of any convex body (i.e., a compact convex set in $\Rn$ with nonempty interior) $K$ is $\alpha$-concave for any $\alpha\leq0$, where
\begin{eqnarray*}
\mathbf{1}_K(x)=\left\{\begin{array}{ccc}
1  & \text{if}&  x\in K,  \\
0  & \text{if}&  x\not\in  K.
\end{array} \right.
\end{eqnarray*}
In this way, $\alpha$-concave functions can be viewed as functional generalizations of convex bodies.

In the literature, a great deal of work has been done to lift the geometric theory on convex bodies to an analogous theory for log-concave functions (see \cite{AAGCJV19,AMJV16,AMJV17,AKM04, AKSW12,CFGLSW16,CW14,CW15,CF13,FYZ24,FM06,FM08,HLXZ24,IN22,KM05,Leh09,LiSW19,LiSW19-2,Rot23,Uli24} among others). However, only a few works have been conducted for $\alpha$-concave functions. In 2013, Rotem \cite{Rot13} defined the support function of $\alpha$-concave functions  and also introduced the concept of mean width for this class. Milman and Rotem \cite{MR13} proposed a natural addition operation $\oplus_{\alpha}$ on the class of $\alpha$-concave functions and defined the mixed integral, a generalization of the mixed volumes of convex bodies. Other contributions for $\alpha$-concave functions can be found in e.g., \cite{AKM04, AKSW12,CW14,CFGLSW16,RX23}.

Our focus in this paper is to develop the notion of surface area measures of $\alpha$-concave functions in $\Rn$ for $-\frac{1}{n}<\alpha <0$ and to solve their related Minkowski problems. The concept of surface area measures of convex bodies is central in convex geometry due to its connections with other important topics (e.g., the mixed volume, Minkowski inequality, Minkowski's problem, etc.), see more information in the book by Schneider \cite{Sch13} and the references in it. The surface area measure $S_K(\cdot)$ of a convex body $K$ can be obtained via the following variational formula:
$$\lim _{t \rightarrow 0^{+}} \frac{|K+t L|-|K|}{t}=\int_{\sphere} h_L d S_K,$$
where $\sphere$ is the unit sphere,  $|K|$ is the volume of $K$, $h_{L}$ is the support function of a convex body $L$, and $K+tL$, for $t>0$, is the Minkowski addition defined by
$$K+tL:=\{x+ty:x\in K,y\in L\}.$$

There has been growing interest in recent years in developing analogous variational formulas for log-concave functions. For two log-concave functions $f=e^{-\varphi}$ and $g=e^{-\psi}$, the first variation of $f$ along $g$ in terms of the Asplund sum is defined by $$\delta J(f, g)=\lim_{t\rightarrow 0^+} \int_{\Rn} \frac{(f\oplus t\bigcdot g)(x)-f(x)}{t}\,dx $$ where the  Asplund sum    $f\oplus t\bigcdot g$ is given by $e^{-(\varphi^*+t\psi^*)^*}$, with $\varphi^*$ denoting the Legendre transform of $\varphi$ (see \eqref{legendre-tran}). Colesanti and Fragal\`a initiated the study in \cite{CF13}, where they derived the variational formula for log-concave functions under certain regularity and growth assumptions. In \cite{Rot22}, Rotem weakened the assumptions made in \cite{CF13} to essential continuity, and later, in \cite{Rot23} he successfully removed these assumptions entirely and derived a general variational formula for the total mass of log-concave functions. Rotem's approach involves calculating an explicit formula for the variation of the total mass of log-concave functions in the direction of a scaled characteristic function and using it to control variation along a bounded function. 
This approach was used again in \cite{HLXZ24} by Huang, Liu, Xi, and Zhao to derive the variational formula for the so-called $q$-th moment of a log-concave function $f$. However, they required an assumption on the rate of convergence of $f$ at the origin $o$ to control the position of $o$, which was successfully removed by Ulivelli in \cite{Uli24}. Similar approach to those in \cite{ HLXZ24, Rot23} 
was applied to the Orlicz-moment of log-concave functions in \cite{FYZZ25} to derive the so-called dual Orlicz curvature measures of log-concave functions.

 In a recent paper, Fang, Ye, and Zhang \cite{FYZ24} obtained a variational formula for the Riesz $\alpha$-energy, which is the functional counterpart of the chord integral, a key object in integral geometry. They effectively eliminated the rate of convergence assumption and obtained a formula that applies to log-concave functions without regularity assumptions. Their conditions can be described  in a simple form  as follows: given two proper, convex and lower semi-continuous
 convex functions  $\varphi, \psi:\Rn\rightarrow \R\cup\{+\infty\}$, there exist constants $\beta_{1}> 0$ and $\beta_{2}\in\R$,  such that, the origin $o$ belongs to $\mathrm{int}\big(\overline {\dom \varphi}\big)$ and  
  \begin{equation}\label{1.1-introduction}
 -\infty<\inf \psi^*\leq   \psi^*\leq\beta_1\varphi^*+\beta_2,\end{equation} where $\dom \varphi=\{x\in \Rn: \varphi(x)<\infty\}$. This relation is natural from the point of view of variational formulas for convex bodies. For example, in \cite{HLYZ16}, Huang, Lutwak, Yang, and Zhang computed the variational formula for the $q$-th dual volume of $K$ in the direction of some convex body $L$. In that case, $K$ has a positive support function, which is a special case of (\ref{1.1-introduction}), if $\varphi^{*} = h_{K}$ and $\psi^{*}=h_{L}$ on $\sphere$. If translation invariance is expected,   \eqref{1.1-introduction} can be weakened to  \begin{equation}\label{1.1-introduction-weak}
 (\dom \psi) \cap \mathrm{int}\big(\beta_1  \overline {\dom \varphi}\big) \neq\emptyset  \ \ \mathrm{and}\ \ \psi^*\leq
\beta_1\varphi^*+\beta_2.  \end{equation} The condition \eqref{1.1-introduction-weak} will be used in our main result, Theorem \ref{thm10}.

In this paper, we will focus on $\alpha$-concave functions with  $\alpha\in (-\frac{1}{n}, 0)$. In this case, from \eqref{a-concave-function-1},  $f^\alpha$ is convex if   $f:\Rn\to[0,+\infty)$ is $\alpha$-concave. Hence, there exists a unique convex function $\varphi: \Rn\to\R\cup\{+\infty\}$ such that the $\alpha$-concave function $f$ can be written in the form: 
$$f(x) = (1-\alpha \varphi(x))^\frac{1}{\alpha}.$$
The convex function $\varphi$ is called the \textit{base} of $f$.   To have $f$ well-defined, it is necessary to have  $\varphi>\frac{1}{\alpha}$.  Let $\conv(\Rn)$ be the set of proper, convex and lower semi-continuous functions $\varphi:\mathbb{R}^n\rightarrow \mathbb{R}\cup\{+\infty\}$. Define $\CA(\Rn)$ to be the set of $\alpha$-concave functions $f = (1-\alpha\varphi)^{\frac{1}{\alpha}},$ where $\varphi \in \conv(\Rn)$ is nonnegative and  coercive, i.e.,
$$\liminf_{\left |x\right|\to +\infty}\frac{\varphi(x)}{\left |x\right|} > 0.$$
Given two $\alpha$-concave functions $f = (1-\alpha\varphi)^\frac{1}{\alpha}$ and  $g = (1-\alpha\psi)^\frac{1}{\alpha}$,  the $\alpha$-combination of $f$ and $g$, with parameter $t$, is defined in \cite{MR13} by
$$f\oplus_{\alpha} t\bigcdot_\alpha g = \big(1-\alpha(\varphi^*+t\psi^*)^*\big)^\frac{1}{\alpha}.$$  
To motivate the notion of surface area measure, we will compute the first variation of  the total mass at $f$  along $g$ in terms of the $\alpha$-sum;
$$ \JD  (f,g):=\lim_{t\to 0^+}\frac{\Ja(f\oplus_{\alpha} t\bigcdot_\alpha g)-\Ja(f)}{t},$$where $J(f) = \int_{\mathbb{R}^{n}}f\,dx$. We shall find the explicit integral expression for $\JD  (f,g)$
following the approach used in   \cite{FYZ24} under the conditions given in \eqref{1.1-introduction} (and hence \eqref{1.1-introduction-weak} because the translation invariance is expected). As one can see from Lemma \ref{Lem2.2}, such an integral expression for $\JD(f, g)$ is arguably meaningful when $-\frac{1}{n}<\alpha<0$. Unlike log-concave functions, the base function of a general $\alpha$-concave function should be strictly larger than $\frac{1}{\alpha}$, and this creates technical difficulties while computing the variational formula. Consequently, in order to avoid singularities and other technical difficulties, our calculation for $\JD(f, g)$ will be done under the assumption that $f=(1-\alpha\varphi)^\frac{1}{\alpha}$ such that  $\varphi\geq 0$.

 Let $\partial E$ be the boundary of $E,$ $\mathrm{int}(E)$ denote the interior of $E$, $K_f$ be the support of $f$ and $\mathcal{H}^{n-1}$ denote the $(n-1)$-dimensional Hausdorff measure. Our first main result is stated in Theorem \ref{MainTheorem1} as follows.

 \vskip 2mm\noindent {\bf Theorem \ref{MainTheorem1}.} {\em 
   Let $-\frac{1}{n}<\alpha < 0$, and $f=(1-\alpha\varphi)^{\frac{1}{\alpha}}\in \CA(\Rn) $. Assume that $g = (1-\alpha\psi)^{\frac{1}{\alpha}}\in \mathcal{C}_{\alpha}^{+}(\mathbb{R}^{n})$, where $\psi\in \conv(\Rn)$, and there exist constants $\beta_{1}> 0$ and $\beta_{2}\in\R$  satisfying \eqref{1.1-introduction-weak}. Then,  
\begin{equation}\label{Var-Formula}
\JD(f,g)=\int_{\Rn}\psi^*\left(\nabla \varphi(x)\right) f(x)^{1-\alpha}\,dx + \int_{\partial {K_f}}{h_{K_g}(\nu_{K_f}(x))}f(x)\,d\mathcal{H}^{n-1}(x),
\end{equation}
where $h_{K_g}$ is the support function of $K_g$, $\nu_{K_f}$  is the  Gauss map of $K_f$, and $\nabla \varphi$ denotes the gradient of $\varphi$.}

Note that our variational formula for $\alpha$-concave functions is not the first in the literature. In special cases where $\varphi$ has compact domain and $\psi$ satisfies certain conditions, our result overlaps with Theorem 3.15 in \cite{Uli24}. However, our condition \eqref{1.1-introduction-weak} is more natural when dealing with convex perturbations. The technique used in \cite[Theorem 3.15]{Uli24} involves approximating convex functions by convex bodies through the Gaussian-symmetric distance, which is entirely different from our method.

Formula (\ref{Var-Formula}) naturally defines the following pair of Borel measures: \begin{align}\Big((\nabla\varphi)_{\sharp}(f^{1-\alpha}\,dx),\ \ (\nu_{K_{f}})_{\sharp}\big(f\,d\mathcal{H}^{n-1}|_{\partial K_{f}}\big)\Big). \label{SAM}\end{align} Here, the former one is the push-forward Borel measure (on $\Rn$) of $f^{1-\alpha}\,dx$ under the map $\nabla \varphi,$ while the  latter one is  the push-forward Borel measure (on the unit sphere $S^{n-1}$)  of $f\,d\mathcal{H}^{n-1}|_{\partial K_{f}}$ under the map $\nu_{K_f}$. Under certain regularity conditions for $f$, say the essential continuity (see Section \ref{necessary-1} or  \cite{EK15} for its definition),  one has $f|_{\partial K_{f}}(x)=0$ for $\mathcal{H}^{n-1}$-almost every $x$ and hence the second measure in \eqref{SAM} is a zero measure. The first measure in \eqref{SAM} is called the \textit{Euclidean surface area measure} of $f$, and is more important for our purpose in this paper. To prove  (\ref{Var-Formula}), we first explicitly calculate $\JD(f, f)$, then condition (\ref{1.1-introduction}) will be used to control the variation of $f$ in the direction of $g$, and finally, the translation invariance will be used to replace \eqref{1.1-introduction} by \eqref{1.1-introduction-weak}. Our proof relies heavily on the generalized dominated convergence theorem and the variational formula \cite[Lemma 5.3]{FYZ24} for the radial functions of unbounded closed convex sets in terms of Minkowski addition, extending from the work for convex bodies by Huang, Lutwak, Yang and Zhang \cite{HLYZ16}.

Having defined the surface area measures in (\ref{SAM}), it is natural to study the associated Minkowski problems, which can be stated as follows.
\vskip 2mm \noindent {\bf Problem \ref{Mink-double-1}.} {\em
  Let $ -\frac{1}{n}<\alpha< 0$, and let $\mu$ and $\nu$ be two Borel measures defined on $\Rn$ and $\sphere$, respectively. Find the necessary and/or sufficient conditions on $\mu$ and $\nu$, such that there exists an $\alpha$-concave function $f$ satisfying 
$$(\mu, \nu)=\Big((\nabla\varphi)_{\sharp}(f^{1-\alpha}\,dx),\ \ (\nu_{K_{f}})_{\sharp}\big(f\,d\mathcal{H}^{n-1}|_{\partial K_{f}}\big)\Big) .$$   }

In this paper, we will concentrate on the case when the second measure in (\ref{SAM}) is the zero measure, where Problem \ref{Mink-double-1} becomes Problem \ref{problem-E-FMP}. 
\vskip 2mm \noindent {\bf Problem \ref{problem-E-FMP}.} {\em Let $ -\frac{1}{n}<\alpha< 0$ and $\mu$  be a Borel measure  defined on $\Rn$. Find the necessary and/or sufficient conditions on $\mu$, such that there exists an $\alpha$-concave function $f$ satisfying 
$$\mu =  (\nabla\varphi)_{\sharp}(f^{1-\alpha}\,dx)\quad \text{and}\quad 0=(\nu_{K_{f}})_{\sharp}\big(f\,d\mathcal{H}^{n-1}|_{\partial K_{f}}\big).$$}

Assuming sufficient regularity on the unknown convex function $\varphi$, to find solutions to Problem \ref{problem-E-FMP} is reduced to solve the following Monge-Amp\'{e}re equation: 
 \begin{align*}
h(\nabla\varphi(y)){\rm det}(\nabla^2\varphi(y))= \big(1-\alpha \varphi(y)\big)^{\frac{1-\alpha}{\alpha}},  
\end{align*}where $h$ is some smooth function in $\Rn$ and ${\rm det}(\nabla^2\varphi(y))$ denotes the determinant of the Hessian matrix of $\varphi$ at $y$.

Note that the classical Minkowski problem (see e.g. \cite{Min1897, Min1903}) is a question about characterizing the surface area measures of convex bodies defined on the unit sphere. The Minkowski type problems take this idea and extend it by considering different kinds of measures associated with convex bodies or functional generalizations of them. The variational method is the most common tool to find solutions to various Minkowski type problems. It was initially used by Alexandrov \cite{Ale38} to solve the classical problem, and has since been extended to Minkowski type problems for convex bodies (see \cite{BBCY19,BLYZ13,BLYZZ20, CLZ19,GHWXY19, HLYZ16,HLYZ18, HXZ21, Lut93, LXYZ23, Zha17, Zhu15} among others). The variational approach was first used to solve the Minkowski problem for log-concave functions by  Cordero-Erausquin and Klartag \cite{EK15}, and has also been applied to solve other Minkowski type problems for log-concave functions in e.g. \cite{FR25, FXY20+, FYZ24, FYZZ25, HLXZ24, Rot22}. 

Optimal transport, on the other hand, provides a machinery to attack problems arising in many areas. It is impossible for us to list all the extraordinary results, so we refer readers to the excellent books by Figalli and Glaudo \cite{FG21}, Santamborgio \cite{San15}, and Villani \cite{Vil03, Vil08}. Of the most importance for our purpose is the wonderful work by Santamborgio \cite{ San16}, where the author  used optimal transport to solve Minkowski type problems for log-concave functions. In another paper \cite{HK21},  Huynh and Santamborgio used the same approach to deal with the $q$-moment measure for convex functions. Underlying the works in  \cite{HK21, San16} is the   celebrated Brenier’s theorem, see e.g. \cite{Bre91, GM96, Mc01}. We shall use optimal transport to deal with the related Minkowski problem for $\alpha$-concave measures in Section \ref{section-solution}. 

Comparing log-concave functions to $\alpha$-conave functions, one can notice the essential difference between them: the convex function  $\varphi_1$  for a log-concave function $f_1=e^{-\varphi_1}$ can  be arbitrarily far below zero; however, this is not the case for $\alpha$-concave functions. As mentioned above, the base function $\varphi$ for an $\alpha$-concave function $f=(1-\alpha \varphi)^{\frac{1}{\alpha}}$ must be  greater than $\frac{1}{\alpha}$. This difference brings significant challenges to solving the associated Minkowski problem for $\alpha$-concave functions using the variational approach. To address this issue and to use the feasible approach by optimal transport, we need to extend the notion of surface area measures from $\alpha$-concave functions to $\alpha$-concave measures.   In brief, an $\alpha$-concave measure is of the form 
\begin{align}\label{a-con-measure-1}\varrho(x) = \rho(x)\,dx + \varrho^{s},
\end{align}
where $\rho=  \left( 1-\alpha\varphi\right)^{\frac{1}{\alpha}}$ for some convex function $\varphi\geq \frac{1}{\alpha}$, and the singular part $\varrho^{s}$ is mutually singular with respect to the Lebesgue measure and is concentrated on $\left\{\varphi =  \frac{1}{\alpha}\right\}$.  Hence, the singular part $\varrho^{s}$ is supported in the set $\{\rho=+\infty\}$.

The notion of the surface area measures of an $\alpha$-concave function $f=(1-\alpha\varphi)^{\frac{1}{\alpha}}$ will be extended to that of $\alpha$-concave measures as follows. By $\partial \varphi$, we mean  the subgradient of $\varphi$ (see \eqref{subg-def-1}). By $\mathrm{Supp}(\pi)$ we mean the support of $\pi$, i.e., the largest closed subset of $\Rn$ for which every open neighbourhood of any point of the set has positive measure with respect to $\pi$.
 
 \vskip 2mm \noindent {\bf Definition \ref{def-measure-measure}.} {\em 
    Let $\varrho$ given in \eqref{a-con-measure-1} be a finite  $\alpha$-concave measure. One says that $\mu$ is a    \textit{Euclidean surface area measure}    of $\varrho$,  if there is a measure $\pi$ on $\mathbb{R}^{n}\times \mathbb{R}^{n}$, whose marginals are $\rho^{1-\alpha}\,dx+\varrho^s$ and $\mu$:$$\pi(A\times \mathbb{R}^{n}) = \big(\rho^{1-\alpha}\,dx+\varrho^s\big)(A)\ \ \mathrm{and} \ \ \pi(\mathbb{R}^{n}\times B) = \mu(B)$$  for any $\big(\rho^{1-\alpha}\,dx+\varrho^s\big)$-measurable set $A$ and  $\mu$-measurable set $B$; and $$\mathrm{Supp}(\pi)\subset \mathrm{Graph}(\partial \varphi):=\{(x,y)\in \mathbb{R}^{n}\times \mathbb{R}^{n}:y\in \partial \varphi(x)\},$$ i.e.,  $ 
y\in \partial \varphi(x) 
$  for $\pi$-almost every $(x,y)$. The spherical surface area of $\varrho$ is defined by $$(\nu_{K_{\rho}})_{\sharp}\big(\rho\,d\mathcal{H}^{n-1}|_{\partial K_{\rho}}\big).
$$}

\vspace{1mm} The extended Minkowski problem for $\alpha$-concave measures is stated as follows.

\vskip 2mm \noindent {\bf Problem \ref{problem-FMP-measure}.} {\em 
 Let $-\frac{1}{n}<\alpha<0$ and let $\mu$  be a Borel measure on $\Rn$ and $\nu$ be a Borel measure on $S^{n-1}$. Does  there exist an $\alpha$-concave measure $\varrho$, given by \eqref{a-con-measure}, such that $\mu$ and $\nu$ are the Euclidean and the spherical surface area measures of $\varrho$, respectively?   } 
  
\vspace{1mm} The extended version of Problem \ref{problem-E-FMP} can be stated as follows.
\vskip 2mm \noindent {\bf Problem \ref{problem-E-FMP-measure}.} {\em     Let $-\frac{1}{n}<\alpha<0$ and $\mu$ be a Borel measure on $\mathbb{R}^{n}$. Does there exists an $\alpha$-concave measure $\varrho$ such that $\mu$ is the Euclidean surface area measure of $\varrho$ and the spherical surface area measure of $\varrho$ is the zero measure?}

 \vspace{1mm}
 As our measures now may no longer be absolutely continuous with respect to the Lebesgue measure, they do not meet a necessary condition of Brenier’s theorem. Instead, we employ an optimality criterion due to Knott and Smith \cite{KS84}, which, although weaker than Brenier's theorem, remains standard in optimal transport theory. Our main result in Section \ref{section-solution} is summaried in Theorem \ref{sol-measure-measure}, which is stated here for readers' convenience. 

 \vskip 2mm \noindent {\bf Theorem \ref{sol-measure-measure}.} {\em 
Let $-\frac{1}{n}<\alpha<0$ and $\mu$ be a probability measure on $\Rn$ with finite first moment such that  the barycenter of $\mu$ is the origin $o$ and $\mu$ is not supported in any hyperplane. Then, there exists an $\alpha$-concave measure $$\bar\varrho = (1-\alpha\varphi_{0})^{\frac{1}{\alpha}}\,dx + \varrho^{s}_{0},$$ such that $\mu$ is the Euclidean surface area measure of $\bar{\varrho}$ and the spherical surface area measure of $\bar\varrho$ is the zero measure.}
 
 \vspace{1mm} In general, however, $\mu$ may not be uniquely determined.  Notably, for almost every $x\in  \left\{\frac{1}{\alpha}<\varphi <+\infty \right\}$, we have $y = \nabla \varphi(x)$. When $\inf \varphi=\frac{1}{\alpha}$  and $\varphi$ is differentiable at 
  $x\in \left\{\varphi=  \frac{1}{\alpha}\right\}$, we have $\partial \varphi(x) = \{o\}$,  and hence  $
\mu = (\nabla\varphi)_{\sharp}(\rho^{1-\alpha}\,dx+\varrho^s).
$  In particular, if $\varphi>\frac{1}{\alpha}$ on $\Rn,$ then $
\mu = (\nabla\varphi)_{\sharp} (\rho^{1-\alpha}\,dx)
$ is exactly the Euclidean surface area measure of the $\alpha$-concave function $\rho=(1-\alpha \varphi)^{\frac{1}{\alpha}}$ defined by the first measure in \eqref{SAM}, and the extended Minkowski problem (i.e., Problem \ref{problem-E-FMP-measure}) reduces to the Euclidean functional Minkowski problem for $\alpha$-concave functions (i.e.,  Problem \ref{problem-E-FMP}). Thus, if $\inf\varphi>\frac{1}{\alpha},$ Theorem \ref{sol-measure-measure} provides a solution to Problem \ref{problem-E-FMP}. However, it is not clear when the solution $\varphi$ given in Theorem \ref{sol-measure-measure} does  satisfy $\inf \varphi >\frac{1}{\alpha}$, and we leave this investigation for future studies. 

\section{Preliminaries and notations} In this section, we will introduce some basic background needed for later context.  We refer readers to \cite{Roc70,Sch13} for more background in convex analysis, and to \cite{FG21,Vil03,Vil08} for topics related to optimal transport. 

\subsection{\texorpdfstring{Convex sets, convex functions and $\alpha$-concave functions}{Convex sets, convex functions and alpha-concave functions}}
Let $\mathbb{N}$, $\mathbb{Z}$, and $\mathbb{Q}$ denote the sets of positive integers, integers, and rational numbers, respectively. Let $\Rn$ be the Euclidean space with dimension $n\in \mathbb{N}$.  Let $\mathcal{L}^n$ and $\mathcal{H}^{n-1}$ be the $n$-dimensional Lebesgue measure and $(n-1)$-dimensional Hausdorff measure, respectively. By $o$, $|x|$  and $\langle x,  y\rangle$, we mean the origin in $\Rn$, the Euclidean norm of $x\in \Rn$  and, respectively, the inner product of $x, y\in \Rn$. Let $\sphere:=\{x\in\Rn:|x|=1\}$ and $\ball:=\{x\in\Rn:|x|\leq1\}$ be the unit sphere and unit ball in $\Rn$.  
Let $\omega_n$ denote the volume of $\ball$.

Let $E\subset \Rn$ be a nonempty set. By $\partial E,\overline{E}$ and $\mathrm{int}(E)$, we mean the boundary, closure and the interior of $E$, respectively. Convex sets play a central role in our study. The set $E$ is said to be convex if, for any $\lambda\in [0,1]$ and $x,y\in E$, we also have $\lambda x+(1-\lambda)y\in E$. If in addition $E$ is compact and has nonempty interior, we call $E$ a convex body.  

When $E\subset \Rn$ is a closed convex set, denote by  $\nu_E: \partial E\rightarrow \sphere$ the  Gauss map  of $E$, mapping the subsets of $\partial E$ to subsets of $\sphere.$   
Note that $\nu_E(x)$ is a singleton set for $\mathcal{H}^{n-1}$-almost all $x\in \partial E$. Let $aE=\{ax: x\in E\}$ for $a>0$. Define the  Minkowski sum of two closed convex sets $aE$ and $bF$ by
$$aE+bF:= \Big\{a x+b y: x\in E,y\in F\Big\}.$$ 
The support function $h_E$ of a closed convex set $E$ is defined by $$h_E(x) := \sup\Big\{\langle x, y\rangle : y \in E\Big\}.$$

A function $\varphi:\Rn\rightarrow \R\cup\{+\infty\}$ is said to be convex if $$\varphi(\lambda x+(1-\lambda)y)\leq \lambda \varphi(x)+(1-\lambda)\varphi(y)$$ for any $\lambda\in [0, 1]$ and $x, y\in \Rn.$ By $\dom \varphi$ we mean the effective domain of $\varphi$, given by $$\dom  \varphi:=\Big\{x\in\Rn:\varphi(x)<+\infty\Big\}.$$ Clearly, if $\varphi$ is a convex function, then $\dom \varphi$ of is a convex set. To enhance readability, we occasionally use $D_{\varphi}$ to mean $\overline{\mathrm{dom}\ \varphi}$. A convex function $\varphi$ is said to be proper if  $\dom \varphi\neq \emptyset$. The subgradient of $\varphi$ at a point $x_0\in\dom \varphi$ is the set
\begin{align}\label{subg-def-1} \partial\varphi(x_0)=\{y\in\Rn:\varphi(x)\geq\varphi(x_0)+\langle y, x-x_0\rangle \ \ \mathrm{for \ any}\ x\in\Rn\}.\end{align}
A well-known fact is that every convex function is differentiable almost everywhere in $\rm int (\dom \varphi)$. When $\varphi$ is differentiable at $x_0\in \rm int (\dom \varphi)$, we use $\nabla\varphi(x_0)$ for the gradient of $\varphi$ at $x_0$.
A convex function $\varphi: \Rn\rightarrow \R\cup\{+\infty\}$ is  coercive if
$$\liminf_{\left |x\right|\to +\infty}\frac{\varphi(x)}{\left |x\right|} > 0.$$ 
According to \cite[Lemma 2.5]{CF13}, if  $\varphi$ is a coercive convex function, then  there are constants $a>0$ and $b\in \R$ such that \begin{align}\label{coercive} \varphi(x)\geq a|x|+b\ \ \mathrm{for}\ \ x\in \Rn. \end{align}  
Let $\conv(\Rn)$ be the set of proper, convex, and lower semi-continuous functions $\varphi:\mathbb{R}^n\rightarrow \mathbb{R}\cup\{+\infty\}$. A natural duality of a function (not necessarily convex) $\varphi$ is the  Legendre transform, which takes the following formulation:  
\begin{align}\label{legendre-tran}\varphi^*(y):=\sup_{x\in\Rn} \Big\{\langle x, y\rangle -\varphi(x)\Big\} \ \ \mathrm{for}\ \  y\in\Rn.\end{align}
The following properties for the Legendre transform hold: for $\lambda>0$ and $\beta\in\R$,
\begin{eqnarray}\label{property}
(\lambda\varphi)^*(y)=\lambda\varphi^*\left(\frac{y}{\lambda}\right)\quad\mathrm{and}\quad  (\varphi-\beta)^*(y)=\varphi^*(y)+\beta.
\end{eqnarray} Clearly, if $\varphi$ is a proper convex function, then $\varphi^*(y)>-\infty$ for any $y\in\Rn$. It can be checked that, for $\varphi: \Rn\rightarrow \R\cup\{+\infty\}$ (not necessarily convex),  $\varphi^*$ is always convex and lower semi-continuous.  Moreover, $\varphi^{*}\leq \psi^*$ if $\varphi\geq \psi$. Let $\varphi^{**}=(\varphi^*)^*$, and then $\varphi^{**}\leq \varphi$ with equality if and only if $\varphi$ is convex and lower semi-continuous. In particular, if $\varphi$ is differentiable at $x\in \rm int (\dom \varphi)$, then
\begin{eqnarray}\label{basicequ}
\varphi^*(\nabla \varphi(x))+\varphi(x)= \langle x,\nabla \varphi(x)\rangle.
\end{eqnarray}
We need the following lemma about the first variation of Legendre transform.
\begin{lemma}\label{Berman's formula} \cite[Proposition 2.1]{Rot22}
Let $\varphi,g:\Rn\to\R\cup\{+\infty\}$ be lower semi-continuous functions. Assume that $g$ is bounded from below and  $g(o),\varphi(o)<+\infty$. Then
$$
\frac{d}{dt}\Big |_{t=0^+}(\varphi+tg)^*(y)=-g(\nabla\varphi^*(y))
$$
at any point $y\in \Rn$ where $\varphi^*$ is differentiable.
\end{lemma}

Fix $\alpha\in (-\infty,0)$. Let $\gen: (\frac{1}{\alpha}, +\infty]\rightarrow [0, +\infty)$ be a function given by \begin{align}\label{generate-function-alpha} \gen(t)=(1-\alpha t)^{\frac{1}{\alpha}},\qquad t\in \Big(\frac{1}{\alpha},+\infty\Big)\end{align} and $\gen(+\infty)=\lim_{t\rightarrow +\infty}\gen(t)=0.$  We call $f:\Rn\to[0,+\infty)$ an \textit{$\alpha$-concave function} if $f=\gen(\varphi)$ for some convex function $\varphi: \Rn\rightarrow (\frac{1}{\alpha},+\infty]$. We will call $\varphi$ and $\gen$ the \textit{base function} and the \textit{generator function} of $f$, respectively. That is, if $f$ is an $\alpha$-concave function, then for any $x\in \Rn$,
\begin{align} \label{a-concave-base} f(x)=\Psi_{\alpha}(\varphi(x)) = (1-\alpha \varphi(x))^\frac{1}{\alpha}.\end{align}  Denote by  $\CA(\Rn)$ the set of all $\alpha$-concave functions whose base functions are  nonnegative, coercive, and in $\mathrm{Conv}(\Rn)$. It is clear that  $f^\alpha$ is a convex function and$$\varphi=\frac{1-f^{\alpha}}{\alpha}.$$  Let  $K_f:=\mathrm{Supp}(f) = \overline{\dom\varphi}$ represent the support of $f$, which is closed and convex. Therefore $\nu_{K_f}$ is well-defined $\mathcal{H}^{n-1}$-almost everywhere on $\partial K_f$. 

Let $\alpha\in (-\infty, 0)$, and  $f, g\in \CA(\Rn)$ be two $\alpha$-concave functions with nonnegative base functions $\varphi$ and $\psi$. Denote by $f\oplus_{\alpha} g$ the  \textit{$\alpha$-sum} of $f$ and $g$, formulated by
\begin{equation}\label{alpha-addition}
f\oplus_{\alpha} g := \Big(1-\alpha(\varphi^*+\psi^*)^*\Big)^{\frac{1}{\alpha}}=\gen\Big((\varphi^*+\psi^*)^*\Big).
\end{equation} For$t>0$, let $t\bigcdot_\alpha f$ be the \textit{$\alpha$-scalar multiplication} of $f$ by $t$, given by 
\begin{equation}\label{alpha-muti}
t\bigcdot_\alpha f:=\Big(1-\alpha (t\varphi^*)^*\Big)^{\frac{1}{\alpha}}=\gen\Big((t\varphi^*)^*\Big).
\end{equation}
Putting them together, we can define the following combination
\begin{align}
f\oplus_{\alpha}t\bigcdot_{\alpha}g=\Psi_{\alpha}((\varphi^{*}+t\psi^{*})^{*}).\label{a-combination}
\end{align}
When $\alpha=0$, the $\alpha$-concave functions become log-concave functions, and hence the $0$-sum of log-concave functions is the well-known Asplund sum of log-concave functions. Moreover, it can be checked that the $\alpha$-sum and the $\alpha$-scalar multiplication also generalize the Minkowski combinations for convex sets through their characteristic functions. Let \begin{align*}
    \Ja(f)=\int_{\Rn} f(x)\,dx
\end{align*}   be the \textit{total mass} of  $f\in \CA(\Rn)$. 

The following lemma shows that the variational formulas in Theorems \ref{thm8} and \ref{thm10} are meaningful when $-\frac{1}{n}<\alpha<0$. By default, we always let $f$ and $\varphi$ satisfy \eqref{a-concave-base}.  
\begin{lemma}\label{Lem2.2}
    Let $\alpha < 0$, $f=(1-\alpha \varphi)^{\frac{1}{\alpha}} \in \CA(\Rn)$,  $0\leq l<  n$ and $p\geq 0$. If $-\frac{1}{n-l+p}<\alpha$, then
    \begin{align} \int_{\Rn}|x|^p\big(1-\alpha \varphi(x)\big)^{\frac{1}{\alpha}-l}\,dx < +\infty. \label{finite -p}\end{align}
In particular, when $-\frac{1}{n}<\alpha<0$,
\begin{align} \label{finite -p=0}\int_{\Rn}f(x)\,dx < +\infty\quad \mathrm{and} \quad\int_{\Rn}f(x)^{1-\alpha}\,dx < +\infty.\end{align}
\end{lemma}

\begin{proof}
    From \eqref{coercive}, there exist constants $a$ and $b$, with $a>0$ and $b\in\R$,  such that $\varphi(x)\geq a\left |x\right|+b$ holds  for any $x\in \Rn$. 
Let $r_0=\max\big\{1,\frac{2(1-\alpha b)}{\alpha a}\big\}$. Clearly, if $|x|>r_0$,  
\begin{equation}\label{radial}
1-\alpha\varphi(x)\geq1-\alpha a|x|-\alpha b\geq -\frac{\alpha a}{2}|x|.
\end{equation}    Note that 
    \begin{align*} 
    \int_{\Rn}|x|^p\big(1-\alpha \varphi(x)\big)^{\frac{1}{\alpha}-l}\,dx
     &=\int_{\{x\in\Rn:|x|\leq r_0\}}|x|^{p}(1-\alpha \varphi(x))^{\frac{1}{\alpha}-l}\,dx\\
    &\quad+\int_{\{x\in\Rn:|x|> r_0\}}|x|^{p}(1-\alpha \varphi(x))^{\frac{1}{\alpha}-l}\,dx\\
    &=: I_1+I_2.
     \end{align*}
  From   $\varphi\geq 0$ and the polar coordinates, we have \begin{align*}
I_1 \leq \int_{\{x\in\Rn:|x|\leq r_0\}}|x|^{p}\,dx 
=\int_{\sphere}\int_{0}^{r_0}r^{n-1+p}\,dr\,du 
=\frac{n}{n+p}\omega_n {r_0}^{n+p}.
\end{align*}
It follows from  \eqref{radial} and the polar coordinates that \begin{align*}
I_2&\leq \Big(\!\!-\!\frac{\alpha a}{2}\Big)^{\frac{1}{\alpha}-l}\int_{\{x\in\Rn:|x|> r_0\}}|x|^{p+\frac{1}{\alpha}-l}\,dx\\
&=\Big(\!\!-\!\frac{\alpha a}{2}\Big)^{\frac{1}{\alpha}-l} \int_{\sphere}\int_{r_0}^{+\infty}r^{n-1+p+\frac{1}{\alpha}-l}\,dr\,du\\
&=-\Big(\!\!-\!\frac{\alpha a}{2}\Big)^{\frac{1}{\alpha}-l} \Big(\frac{n\omega_n }{n+p+\frac{1}{\alpha}-l}\Big) {r_0}^{n+p+\frac{1}{\alpha}-l},
\end{align*}
where we use $n+p+\frac{1}{\alpha}-l<0$ if 
 $-\frac{1}{n-l+p}<\alpha$.
Therefore
$$\int_{\Rn}|x|^p\big(1-\alpha \varphi(x)\big)^{\frac{1}{\alpha}-l}\,dx < +\infty.$$ This proves inequality \eqref{finite -p}. 

For \eqref{finite -p=0}: let  $p=0$. If $-\frac{1}{n}<\alpha<0$, then $-\frac{1}{n-l}\leq -\frac{1}{n}<\alpha<0 $ and 
     $$\int_{\Rn}(1-\alpha \varphi(x))^{\frac{1}{\alpha}-l}\,dx < +\infty.$$ When $l=0$, we get the first inequality in \eqref{finite -p=0}; while if $l=1,$ the second inquality in \eqref{finite -p=0} holds. \end{proof}

\subsection{Optimal Transport} Let $\varrho$ and $\mu$  be two probability measures on $\Rn$. We call $\pi$ a  \textit{transference plan} of $\varrho$ and $\mu$, if $\pi$ is a measure on $\Rn\times \Rn$ such that, for all measurable sets $A, B\subset \Rn$,  $\pi(A\times \Rn) = \varrho(A)$ and $\pi(\Rn\times B) = \mu(B)$. Denote by $\Pi(\varrho,\mu)$ the set of transference plans of $\varrho$ and $\mu$.  Let $c: \Rn\times \Rn\to[0,+\infty]$ be a cost function. A well-known problem in optimal transport is the Kantorovich problem, which aims to find minimizers for
\begin{align}
    \min\left\{\int_{\Rn\times \Rn}c(x,y)\,d\pi(x,y):\pi\in\Pi(\varrho,\mu)\right\}. \label{KP-problem}
\end{align}
      This problem admits solutions if the cost function $c(\cdot, \cdot)$ is assumed to be lower semi-continuous.   The dual problem for the Kantorovich problem is \begin{align}
\max\left\{\int_{\Rn}\varphi \,d\varrho+\int_{\Rn}\psi \,d\mu: \ \varphi(x)+\psi(y)\leq c(x,y)\right\}. \label{dual-KP} \end{align} 
Note that the extremal values of the Kantorovich problem (i.e., \eqref{KP-problem}) and its dual (i.e., \eqref{dual-KP}) are the same.


A measure $\varrho$ on $\Rn$ has finite $p$-th moment if $\int_{\Rn}|x|^pd\varrho(x)<+\infty$. Denote by $\mathcal{P}_{p}(\Rn)$ the set of probability measures on $\Rn$ with finite $p$-th moment. Suppose that $T$ is a map between two measure spaces $(\mathbb{R}^{n},\varrho)$ and $(\mathbb{R}^{n},\mu)$. We say that $T$ pushes $\varrho$ forward to $\mu$, denoted by $T_{\sharp}\varrho = \mu$, if $\mu(B) = \varrho(T^{-1}(B))$ for any $\mu$-measurable set $B$.  If $\varrho, \mu \in \mathcal{P}_{2}(\Rn)$, a distance between $\varrho$ and $\mu$ can be defined as follows:  
\begin{align*}
W_{2}(\varrho,\mu)= \inf_{T}\left\{\left(\int_{\Rn}\left |x-T(x)\right|^{2}\,d\varrho(x)  \right)^{1/2}:T_{\#}\varrho = \mu\right\}, 
\end{align*} which will be called the  $2$-Wasserstein distance between $\mu$ and $\varrho.$ Define the  \textit{maximal correlation functional} by $$
\mathcal{T}(\varrho,\mu):= \frac{1}{2}\int_{\Rn}\left |x\right|^{2}d\varrho(x)+\frac{1}{2}\int_{\Rn}\left |y\right|^{2}d\mu(y) - \frac{1}{2}W^{2}_{2}(\varrho,\mu). $$
Note that the following integral is finite: $$\int_{\Rn}\left |y\right|^{2}d\mu(y)<\infty. $$ Consequently, $2\mathcal{T}(\varrho,\mu)$ differs from  $\int\left |x\right|^{2}d\varrho - W^{2}_{2}(\varrho,\mu)$ by a constant if we fix $\mu$. Hereafter, we shall use $\mathcal{T}(\varrho,\mu)$ rather than $W_{2}(\varrho,\mu)$ because  $\mathcal{T}(\varrho,\mu)$  is well-defined in $\mathcal{P}_{1}(\Rn)$ by the formula
\begin{align}\label{max-Big-Tau}
\mathcal{T}(\varrho,\mu) :=\sup\left\{\int_{\Rn\times \Rn}\langle x,  y\rangle\mathrm{~d}\pi(x, y)~:~\pi\in\Pi(\varrho,\mu)\right\}. 
\end{align} 
Finding the maximizer for the optimization problem \eqref{max-Big-Tau} is equivalent to solving the Kantorovich problem (i.e., \eqref{KP-problem}) with $c(x,y)=-\langle x,  y\rangle$; it is also  equivalent to the problem \eqref{KP-problem} with quadratic cost (i.e., $c(x,y) = \frac{\left |x-y\right|^{2}}{2}$) if $\varrho,\mu\in \mathcal{P}_{2}(\Rn)$. A dual formulation for $\mathcal{T}$ can be formulated by (see \cite[Theorem 2.5.6]{FG21}):
\begin{align} \label{def-T-2}
\mathcal{T}(\varrho,\mu):=\inf\left\{\int_{\mathbb{R}^{n}}\varphi \,d\varrho+\int_{\mathbb{R}^{n}}\varphi^{*} \,d\mu:\varphi\text{ is convex and lower semi-continuous}\right\}.
\end{align} The following fact is standard in optimal transport theory, see e.g. \cite[Theorem 4.1 and Theorem 5.10]{Vil08}.
\begin{theorem}[\bf{Knott-Smith optimality criterion}]\label{[Vil08, Theorem 5.10]}
    Let $\varrho$ and $\mu$ be two probability measures on $\Rn$ such that $\mathcal{T}(\varrho,\mu) < +\infty$. Then the solution $\pi$ to \eqref{KP-problem} exists. Furthurmore,
    $$\mathrm{Supp}(\pi)\subset \mathrm{Graph}(\partial \varphi),$$
 where $\varphi$ is the convex, lower semicontinous function such that
 $$\mathcal{T}(\varrho,\mu)=\int_{\mathbb{R}^{n}} \varphi \,d \varrho+\int_{\mathbb{R}^{n}} \varphi^* d \mu.$$
Here $\mathrm{Supp(\pi)}$ is the largest closed subset of $\Rn$ for which every open neighbourhood of any point of the set has positive measure with respect to $\pi$.
 \end{theorem}
We will need the following facts about the maximal correlation functional, which can be found in \cite[Proposition 3.1]{San16}.
\begin{proposition}\label{[San16,Prop3.1]} Let $\mu\in \mathcal{P}_1(\Rn)$ be a fixed measure whose barycenter is $o$. Then:

\vskip 1mm \noindent (i) For every $\varrho\in \mathcal{P}_1(\Rn)$, one has $\mathcal{T}(\varrho,\mu)\geq 0$.
       \vskip 1mm \noindent (ii)  If $\tilde{\varrho}(x) = \varrho(x+z)$ for some $z\in\Rn$, then $\mathcal{T}(\varrho,\mu) = \mathcal{T}(\tilde\varrho,\mu)$.
      \vskip 1mm \noindent (iii)  If $\varrho_n$ has barycenter at $o$ for all $n$ and converges weakly to $\varrho$, then $$\liminf_{n\to\infty} \mathcal{T}(\varrho_n,\mu) \geq\mathcal{T}(\varrho,\mu).$$
     
\end{proposition}

The following lemma compares the first moment of $\varrho$ and $\mathcal{T}(\varrho, \mu)$. We follow the approach in \cite[Proposition 3.2]{San16}, where the case $\varrho$ is absolutely continuous with respect to $\mathcal{L}^{n}$, denoted by $\varrho \ll \mathcal{L}^{n}$, is treated.
\begin{proposition}\label{[San16,Prop3.2]} 
 Let $ \mu\in\mathcal{P}_1(\Rn)$ be such that the support of $\mu$ is not concentrated on a hyperplane. Suppose that the barycenter of $\varrho \in\mathcal{P}_1(\Rn)$ is $o$. Then, there exists a constant $c=c_{\mu}>0$, such that   $$\mathcal{T}(\varrho,\mu)\geq c\int_{\mathbb{R}^{n}}\left|x\right|d\varrho(x).$$
\end{proposition}
\begin{proof}
    Let $\{e_{i}\}$ be the standard basis of $\Rn$. It can be checked that 
    \begin{align*}\left |x\right| &= \bigg( \sum_{i=1}^{n}\left\langle x,e_{i}\right\rangle^{2}\bigg)^{\frac{1}{2}}\leq \sum_{i=1}^{n}\left |\left\langle x,e_{i}\right\rangle\right|= \sum_{i=1}^{n}[\left\langle x,e_{i}\right\rangle_{+}+\left\langle x,-e_{i}\right\rangle_{+}],\end{align*}
where $a_{+} = \max\{0,a\}$. Hence, the first moment of $\varrho$ is bounded above by 
    \begin{equation}\int_{\Rn} \left |x\right|d\varrho \leq \int _{\Rn}  \sum_{i=1}^n\left[\left\langle x,e_{i}\right\rangle_{+}+\left\langle x,-e_{i}\right\rangle_{+}\right] \,d\varrho\leq 2n \sup _{\xi \in S^{n-1}} \int_{\mathbb{R}^{n}}\left\langle x, \xi\right\rangle_{+}\mathrm{d} \varrho(x).\label{BoundOfFirstMoment}\end{equation}

    Now we prove the existence of $\xi_{0}\in \sphere$ depending on $\varrho$,  which maximizes the function \begin{align}
         \Phi: \xi \mapsto \int_{\Rn} \left\langle x, \xi \right\rangle_{+} \mathrm{d} \varrho(x) \label{definition-Phi}
    \end{align} on  $\sphere$. Note that $\xi \mapsto \left\langle x, \xi \right\rangle_{+}$ is a continuous function. Moreover, $\left\langle x, \xi \right\rangle_{+} \leq \left |x\right|$ and $\int_{\Rn} |x|  \,d\varrho<\infty$. By the dominated convergence theorem, for any sequence $\xi_{n}\in S^{n-1}$ converging to $\xi_{0}$,
\begin{align*}\lim_{n\to\infty}\int_{\mathbb{R}^{n}} \!\! \left\langle x,\xi _{n}\right\rangle _{+}\,d\varrho(x) = \!\! \int_{\mathbb{R}^{n}}  \lim_{x\to\infty}\left\langle x,\xi_{n}\right\rangle _{+}\,d\varrho(x) = \!\! \int_{\mathbb{R}^{n}} \!\! \left\langle x, \lim_{n\to \infty}  \xi_{n}\right\rangle _{+}\,d\varrho(x) = \!\! \int_{\mathbb{R}^{n}}\!\! \left\langle x,\xi_{0}\right\rangle_{+} \,d\varrho(x).\end{align*} Thus,  $\Phi$ is continuous on $S^{n-1}$, and there exists $\xi_{0}\in \sphere$ which  maximizes $\Phi$ over $\sphere$, due to the compactness of  $S^{n-1}$.

Let $A_{0}^{+} := \{x: \! \langle x, \xi_{0}\rangle \! >0\}$, $A_{0}:=\{x: \! \langle x, \xi_{0}\rangle \! =0\}$, and  $A_{0}^{-}:=\{x: \! \langle x, \xi_{0} \rangle \!<0\}$.   
We define $A_i$, $A_i^+$ and $A_i^-$ iteratively for $i=1, 2, \cdots, n-1$, as follows:  choose any $\xi_i\in A_{i-1} \cap \sphere$, and  let  $A_{i}^{+}:=\!\{x\!\in\! A_{i-1}:\!\langle x, \xi_{i}\rangle\! >0\}$, $A_{i} :=\! \{x\!\in\! A_{i-1}:\!\langle x, \xi_{i}\rangle =0\}$, and $ A_{i}^{-}:=\! \{x\!\in\! A_{i-1}:\!\langle x, \xi_{i}\rangle \! <0\}. $   Let 
$$A^{+}:=\bigcup_{i=0}^{n-1}A_{i}^{+} \ \ \mathrm{and} \ \  A^{-}:= \bigcup_{i=0}^{n-1}A_{i}^{-}.$$ Note that $A^+\cap A^{-}=\emptyset$, $A^+\cup A^{-}=\Rn\setminus\{o\}$, and $A^{-}=-A^{+}.$

Let $m^{+}:= \varrho(A^{+})+\frac{1}{2}\varrho(\{o\})$ and $m^{-}:=\varrho(A^{-})+\frac{1}{2}\varrho(\{0\})$, which gives   $m^{+}+m^{-}=1$. Without loss of generality, we assume that $\varrho$ is not the Dirac delta measure $\delta_{o}$ supported at $\{o\}$.  Define $\xi^{+}:= \frac{v^{+}}{|v^{+}|},$ where $$v^{+}= \int_{A^{+}} x\, \,d \varrho \neq o $$ due to the assumptions about $\varrho$. Clearly, \begin{align} 
v^{-}=\int _{A^{-}} x\,d\varrho=-v^{+}.\label{def-v--}\end{align}

We now claim the existence of   $\ell_{0}\in \R$ such that  $$\mu(\{x: \left\langle x, \xi^+  \right\rangle >\ell_{0}\}) \leq  m^{+} \ \ \mathrm{and} \ \ \mu(\{x: \left\langle x, \xi^+ \right\rangle <\ell_{0}\}) \leq  m^{-}.$$ To this end,  by \cite[Theorem 1.2.1]{Dur19}, the cumulative distribution function $$F(\ell) := \mu\{x:\left\langle x, \xi^{+} \right\rangle\leq \ell\} $$ is right-continuous and non-decreasing. Let $\ell_0=\sup \{t: \ F(t)\leq m^-\}$. Then, $$\mu\{x:\left\langle x, \xi^+ \right\rangle< \ell_0 \} =\lim_{t\rightarrow \ell_0^-}  F(t)\leq m^-. $$ Moreover, for any $t>\ell_0$, we get $F(t)>m^-.$ Due to the right-continuity of $F(\cdot)$, we get $$F(\ell_0)=\lim_{t\rightarrow \ell_0^+} F(t)\geq m^-.$$ This further gives $$\mu\{x:\left\langle x, \xi^+ \right\rangle >\ell_0\}=1-\mu\{x:\left\langle x, \xi^+ \right\rangle\leq \ell_{0}\}=1-F(\ell_0)\leq 1-m^-=m^+. $$ 
 Let $\mu=\mu^{+}+\mu^{-}$, where $\mu^{\pm}$ are nonnegative measures such that $ \mu^{\pm}(\mathbb{R}^{n}) = m^{\pm}$, and
$$
\mathrm{Supp}(\mu^{+})\subset \{x:\left\langle x, \xi^{+}\right\rangle\geq \ell_{0}\}\ \ \mathrm{and} \ \  \mathrm{Supp}(\mu^{-})\subset \{x:\left\langle x, \xi^{+}\right\rangle\leq \ell_{0}\}.
$$
Consider the following coupling $\gamma\in \Pi(\varrho,\mu)$ of $\varrho$ and $\mu$:
$$
\gamma:= \frac{1}{m^{+}}\Big( \varrho\mres A^{+}+\frac{1}{2} \varrho \mres \{o\}\Big)\otimes \mu^{+}+\frac{1}{m^{-}}\Big( \varrho\mres A^{-}+\frac{1}{2}\varrho\mres \{o\}\Big)\otimes \mu^{-},
$$ where $\varrho \mres E$ is the restriction of $\varrho $ to $E,$ and $\nu\otimes \mu$ is the product measure of $\nu$ and $\mu.$ Thus 
\begin{align*}\int_{\Rn\times \Rn} \left\langle x,y\right\rangle d\gamma 
&= \frac{1}{m^{+}}\bigg\langle  \int_{\Rn}  x \,d \Big( \varrho\mres A^{+}+\frac{1}{2} \varrho \mres \{o\}\Big),\int_{\Rn}  y \,d\mu^{+}(y)\bigg\rangle\\&\ \ \    +\frac{1}{m^{-}}\bigg\langle \int_{\Rn}  xd\Big( \varrho\mres A^{-}+\frac{1}{2}\varrho\mres \{o\}\Big),\int_{\Rn}  y \,d\mu^{-}(y)\bigg\rangle
\\&= \frac{1}{m^{+}}\bigg\langle \int_{A^{+}} x \,d\varrho,\int_{\Rn}  y \,d\mu^{+} \bigg\rangle+\frac{1}{m^{-}}\bigg\langle \int_{A^{-}}x\,d\varrho,\int_{\Rn} y\,d\mu^{-}\bigg\rangle\\&=\frac{1}{m^{+}}\bigg\langle v^+,\int_{\Rn} y \,d\mu^{+} \bigg\rangle+\frac{1}{m^{-}}\bigg\langle v^{-},\int_{\Rn} y\,d\mu^{-}\bigg\rangle.\end{align*}
Recall that $v^+=-v^{-}\neq o$ by \eqref{def-v--}.  Thus, with $\xi^+=\frac{v^+}{|v^+|}$, we get,  
\begin{align} \int_{\Rn\times \Rn} \left\langle x,y\right\rangle d\gamma  &=\frac{|v^{+}|}{m^{+}}\int _{\Rn} \left\langle y,\xi^{+}\right\rangle \,d\mu^{+}(y)-\frac{|v^{+}|}{m^{-}}\int_{\Rn}  \left\langle y, \xi^{+}\right\rangle \,d\mu^{-}(y)\nonumber \\&=\frac{|v^{+}|}{m^{+}}  \int_{\Rn} \left( \left\langle y, \xi^{+}\right\rangle-\ell_0\right) d \mu^{+}(y)-\frac{|v^{+}|}{m^{-}}  \int_{\Rn} \left( \left\langle y, \xi^{+}\right\rangle-\ell_0\right) d \mu^{-}(y) \nonumber \\&=|v^{+}|\int_{\Rn} \left |\left\langle y,\xi^{+}\right\rangle-\ell_0\right|d\bigg( \frac{\mu^{+}}{m^{+}}+\frac{\mu^{-}}{m^{-}}\bigg)\nonumber  \\&\geq|v^{+}|\int_{\Rn}  \big|\left\langle y,\xi^{+}\right\rangle-\ell_0\big| \,d\mu, \label{estimate-Phi}\end{align} where we have used  $m^{\pm}\leq 1$. 
As the support of  $\mu$ is not concentrated on any hyperplane, we have $\int_{\Rn} \left|\left\langle y, \xi^{+}\right\rangle-\ell_0\right| d \mu>0$.  Moreover, 
\begin{align}
\int_{\Rn} \left\langle x,\xi_{0}\right\rangle_{+} \,d \varrho(x)&=\int_{A_0^{+}}\left\langle x,\xi_{0}\right\rangle \,d \varrho(x)\nonumber\\&=\bigg\langle \int_{A_{0}^{+}}x \,d\varrho(x),\xi_{0}\bigg\rangle\leq \left|\int_{A_0^{+}} x\,d\varrho(x)\right| \nonumber \\ & \leq \left |\int_{A_{0}^{+}}x\,d\varrho(x)+\int_{A^{+}\setminus A_{0}^{+}}x\,d\varrho(x)\right| \nonumber \\ & = \left |\int_{A ^{+}}x\,d\varrho(x) \right|  = |v^{+}|<\infty.    \label{upper bound-T}
\end{align} The second inequality is by the fact that $\int_{A_0^{+}} x\,d\varrho$ is oriented as $\xi_{0}$, as otherwise let $\xi_{0}'$ be the orientation of $\int_{A_0^{+}} x\,d\varrho$, then 
\begin{align*}\int_{\mathbb{R}^{n}}\left\langle x,\xi_{0}\right\rangle_{+}\,d\varrho(x) &= \bigg\langle \int_{A_{0}^{+}}x \,d\varrho(x),\xi_{0}\bigg\rangle<\bigg\langle \int_{A_{0}^{+}}x\,d\varrho(x),\xi_{0}'\bigg\rangle \\&= \int_{A_{0}^{+}}\left\langle x,\xi_{0}'\right\rangle \,d\varrho(x) \leq \int_{\{x:\left\langle x,\xi_{0}'\right\rangle>0\}}\left\langle x,\xi_{0}'\right\rangle \,d\varrho(x),\end{align*} a contraction to the optimality of $\xi_0$ for $\Phi(\cdot)$ defined in \eqref{definition-Phi}. Additionally, another reason for the second inequality in \eqref{upper bound-T} is $$\int_{A^{+} \backslash A_0^{+}} x\,d\varrho \in A^{+}\setminus A_{0}^{+}\in \xi_{0}^{\perp}.$$ 

 The desired result is an immediate consequence of the combination of \eqref{max-Big-Tau}, \eqref{BoundOfFirstMoment},  \eqref{estimate-Phi} and \eqref{upper bound-T}:
 \begin{align*}\mathcal{T}(\varrho,\mu)&\geq \int_{\Rn\times \Rn}\langle x,y\rangle\,d\gamma(x,y) \\&\geq |v^{+}|\int_{\mathbb{R}^{n}}\big|\langle y,\xi^{+}\rangle-\ell_{0}|\,d\mu(y)
 \\&\geq \left( \int_{\mathbb{R}^{n}}\big|\langle y,\xi^{+}\rangle-\ell_{0}|\,d\mu(y)\right)\int_{\mathbb{R}^{n}}\langle x,\xi_{0}\rangle_{+}\,d\varrho(x) \\&\geq\frac{1}{2n}\left( \int_{\mathbb{R}^{n}}\big|\langle y,\xi^{+}\rangle-\ell_{0}|\,d\mu(y)\right)\int_{\Rn}|x|\,d\varrho(x)
 \\
 &\geq c\int_{\mathbb{R}^{n}}|x|\,d\varrho(x)
 \end{align*}
 where $c=c_{\mu}$ is given by the positive number $\frac{1}{2n}\inf\{\int_{\mathbb{R}^{n}}|\left\langle y,e\right\rangle-l|\,d\mu(y):e\in S^{n-1},l\in\mathbb{R}\}$.
 \end{proof}

\section{The variational formula}\label{Section:3} 
In this section, we will compute the integral expression for $\JD(f,g)$, which is defined as follows.

\begin{definition}\label{Definition-first-variation} Let  $-\frac{1}{n}<\alpha < 0$, $f\in \CA(\Rn)$, and $g$ be an $\alpha$-concave function. Define $\JD(f,g)$ as the first variation of the total mass $\Ja(f)$ along $g$, in terms of the combination \eqref{a-combination}, by: 
    \begin{align} 
        \JD  (f,g):=\lim_{t\to 0^+}\frac{\Ja(f\oplus_{\alpha} t\bigcdot_\alpha g)-\Ja(f)}{t}, \label{def-1st-variation}
    \end{align} provided that the above limit exists and is finite. 
\end{definition}
 It can be  easily checked from \eqref{a-combination} and \eqref{def-1st-variation} that, for any $\beta>0$, 

 \begin{align} 
        \JD  (f, \beta\bigcdot_{\alpha} g):&=\lim_{t\to 0^+}\frac{\Ja(f\oplus_{\alpha} t\bigcdot_\alpha (\beta\bigcdot_{\alpha} g))-\Ja(f)}{t} \nonumber \\&=\beta \lim_{t\to 0^+}\frac{\Ja(f\oplus_{\alpha} (\beta t)\bigcdot_\alpha  g)-\Ja(f)}{\beta t} \nonumber \\ &=\beta \cdot  \JD(f, g). \label{def-1st-variation-beta}
    \end{align}
In view of \eqref{def-1st-variation-beta}, we get, for any $\beta>0$,   \begin{align}\label{3.1}
        \JD(f,\beta\bigcdot_\alpha f)   = \beta \JD(f,f).
    \end{align}

We begin by calculating the variation of $f\in \CA(\Rn)$ with itself. Recall that $f\in \CA(\Rn)$ if $$f(x) = (1-\alpha \varphi(x))^\frac{1}{\alpha} =\gen(\varphi(x))$$ with $\gen(\cdot)$ the generator function given in \eqref{generate-function-alpha}, and the base function $\varphi$ is proper, nonnegative, lower semi-continuous, and  coercive.

\begin{proposition}\label{VarFormulaWithItself}
Let $-\frac{1}{n}<\alpha < 0$ and $f  \in \CA(\Rn)$. Then
 \begin{align}\label{3.1-no-beta}
        \JD(f,f)=  n\Ja(f) - \int_{\Rn}\varphi(x)f(x)^{1-\alpha}\,dx.
    \end{align} Moreover, $\JD(f,f)\in (-\infty,+\infty)$.
\end{proposition}

\begin{proof}
Let $\overline{f}_{t}:= f\oplus_{\alpha} t\bigcdot_\alpha  f$ for $t>0$. From \eqref{a-combination}, we have
\begin{align*}\overline{f}_{t} &= \gen\Big((\varphi^*+ t\varphi^*)^*\Big)=\gen\Big(\big((1+t)\varphi^*\big)^*\Big).  \end{align*} It follows from \eqref{property} that 
    \begin{align*}\overline{f}_{t}(x) &=  \gen\bigg(\!(t+1)\varphi\Big(\frac{x}{t+1}\Big)\!\bigg).\end{align*}
By \eqref{generate-function-alpha} and the substitution $x=(t+1)\widetilde{x}$, we get 
    \begin{equation}\label{t-beta}
        \Ja(\overline{f}_{t})  =\int_{\Rn}\Big( 1-\alpha (t+1)\varphi\Big( \frac{x}{t+1}\Big)\Big)^{\frac{1}{\alpha}}\,dx = (t+1)^{n}\int_{\Rn}\big(1-\alpha(t+1)\varphi(\widetilde{x})\big)^{\frac{1}{\alpha}}d\widetilde{x}.
    \end{equation} 
As $\varphi\geq 0$, the function $(1-\alpha(t+1)\varphi)^{\frac{1}{\alpha}}$ is increasing as $t\downarrow 0$. Employing the monotone convergence theorem, we have  \begin{equation}\label{mono} \lim_{t\to 0^{+}}\int_{\Rn}(1-\alpha(t+1)\varphi(x))^{\frac{1}{\alpha}}\,dx =\Ja(f). \end{equation} Consequently, we can have   \begin{equation}\label{3.2}
        \begin{aligned}[b]
            \JD(f, f) &= \lim_{t\to 0^{+}}\int_{\Rn}\frac{\overline{f}_{t}(x)-f(x)}{t}\,dx\\
            &= \lim_{t\to 0^{+}}\frac{(t+1)^{n}-1}{t} \lim_{t\to 0^{+}}\int_{\Rn}(1-\alpha(t+1)\varphi(x))^{\frac{1}{\alpha}}\,dx \\
            &\quad + \lim_{t\to 0^{+}}\int_{\Rn}\frac{(1-\alpha(t+1)\varphi(x))^{\frac{1}{\alpha}}-(1-\alpha\varphi(x))^{\frac{1}{\alpha}}}{t}\,dx\\
            &= n \Ja(f)+\lim_{t\to 0^{+}}\int_{\Rn}\frac{(1-\alpha(t+1)\varphi(x))^{\frac{1}{\alpha}}-(1-\alpha\varphi(x))^{\frac{1}{\alpha}}}{t}\,dx.
        \end{aligned}
    \end{equation} 
Let $0\leq t\leq 1$. By the mean value theorem,    for $x\in\dom\varphi$ there exists   $s_x\in(0,t)$, such that 
\begin{align}\label{mean}
\frac{(1-\alpha\varphi(x))^{\frac{1}{\alpha}}-(1-\alpha(t+1)\varphi(x))^{\frac{1}{\alpha}}}{t}
&= \varphi(x)(1\!-\!\alpha(s_x+ 1)\varphi(x))^{\frac{1-\alpha}{\alpha}}\nonumber \\ & \leq \varphi(x)(1-\alpha\varphi(x))^{\frac{1-\alpha}{\alpha}}.  
\end{align}   On the other hand,
since $\frac{\tau}{1-\alpha \tau}$ increases to $-\frac{1}{\alpha}$ as $\tau\in[0,\infty)$, for any $x\in \dom\varphi$, we have
\begin{align} \label{comare-a} \frac{\varphi(x)}{1-\alpha\varphi(x)}\leq -\frac{1}{\alpha}.\end{align} Both \eqref{mean} and \eqref{comare-a} also hold for $x\notin \dom(\varphi)$, in which the left hand side of \eqref{mean} is $0$, and the left hand side of \eqref{comare-a} is viewed as $\lim_{\tau\rightarrow +\infty}\frac{\tau}{1-\alpha\tau}=-\frac{1}{\alpha}$, due to $\varphi(x)=+\infty$.  Therefore, for any $x\in \Rn,$
\begin{equation}\label{minusalpha}
\varphi(x)(1-\alpha\varphi(x))^{\frac{1}{\alpha}-1}=\frac{\varphi(x)}{1-\alpha\varphi(x)}\big(1-\alpha\varphi(x)\big)^{\frac{1}{\alpha}}\leq -\frac{1}{\alpha}\big(1-\alpha\varphi(x)\big)^{\frac{1}{\alpha}}.
\end{equation} Combining \eqref{mean} and \eqref{minusalpha}, we have
\begin{align} \label{upper-b-int} 0\leq\frac{(1-\alpha\varphi(x))^{\frac{1}{\alpha}}-(1-\alpha(t+1)\varphi(x))^{\frac{1}{\alpha}}}{t}
\leq-\frac{f(x)}{\alpha}.\end{align} 
It follows from  Lemma \ref{Lem2.2} and  the dominated convergence theorem that
 \begin{align}
\int_{\Rn}\varphi(x)f(x)^{1-\alpha}\,dx&=-\int_{\Rn}\lim_{t\to 0^{+}}\frac{(1-\alpha(t+1)\varphi(x))^{\frac{1}{\alpha}}-(1-\alpha\varphi(x))^{\frac{1}{\alpha}}}{t}\,dx\nonumber\\ &=-\lim_{t\to 0^{+}}\int_{\Rn}\frac{(1-\alpha(t+1)\varphi(x))^{\frac{1}{\alpha}}-(1-\alpha\varphi(x))^{\frac{1}{\alpha}}}{t}\,dx.  \label{3.3}
    \end{align} Hence, formula \eqref{3.1-no-beta} follows from \eqref{3.2} and \eqref{3.3}. The finiteness of $\JD(f,f)$ is a direct consequence of Lemma \ref{Lem2.2} (in particular, $\Ja(f)\in [0, \infty)$ and \eqref{minusalpha}.
\end{proof}
 
\vskip 2mm 
\begin{theorem}\label{prop5}
    Let $-\frac{1}{n}<\alpha < 0$, $\beta_{1}\geq  0$ and $\beta_{2}\in\R$ be constants.  Let $f \in \CA(\Rn)$ with $o\in \mathrm{int}(K_f)$, and \begin{align} 
            \widehat{f_{t}}   = \big( 1-\alpha [(1+\beta_{1}t)\varphi^*]^*+ \alpha\beta_{2}t\big)^{\frac{1}{\alpha}}, \label{define-f-hat-1}
        \end{align} for $t>0$ small enough.       
        Then,  the following formula holds: 
    \begin{equation}\label{3.4}
       \mathcal{I}_{\alpha}(f)=  \lim_{t\to 0^{+}}\int_{\Rn}\frac{\widehat{{f}_{t}}(x)-f(x)}{t}\,dx  =\beta_{2}\int_{\Rn}f(x)^{1-\alpha}\,dx+\beta_{1}\JD(f,f).
    \end{equation}
    Moreover, $ \mathcal{I}_{\alpha}(f) \in (-\infty,+\infty)$.
\end{theorem}
\begin{proof}  Note that $\varphi\geq 0$. Then, $\inf \varphi$ exists and is nonnegative. It can be checked that, due to   $-\alpha \inf \varphi\geq 0$ and $\beta_1\geq 0$,
    \begin{align*}
    -\alpha [(1+\beta_{1}t)\varphi^*]^*(x)+\alpha\beta_{2}t
    &=\bigg(\!\alpha\beta_{2}-\alpha\beta_{1}\varphi\Big(\frac{x}{1+\beta_{1}t}\Big)\!\bigg)t - \alpha\varphi\Big(\frac{x}{1+\beta_{1}t}\Big)\\
    &\geq \Big(\alpha\beta_{2}-\alpha\beta_{1}\inf\varphi\Big)t - \alpha\inf\varphi. 
    \end{align*}  Thus,   there exists $t_0$ (independent of $x\in \Rn$), such that for all $0<t<t_0$ and any $\beta_2$, we have  
    \begin{align} \label{control-beta-2}
      -\alpha [(1+\beta_{1}t)\varphi^*]^*+ \alpha\beta_{2}t\geq -\frac{1}{2}.  
    \end{align} 

  For $x\in \Rn$, let  \begin{align} s&=s(x,t):=-\alpha[(1+\beta_{1}t)\varphi^*]^*(x), \label{definition-s-x}  \\  
\tau &=\tau(x,t): =\alpha \beta_{2}t-\alpha[(1+\beta_{1}t)\varphi^*]^*(x). \label{definition-t-x}\end{align} It can be easily checked from \eqref{control-beta-2} that $s\geq -\frac{1}{2}$ and $\tau\geq-\frac{1}{2}$ for $0<t<t_0$.    Let $g(\lambda):=(1+\lambda)^\frac{1}{\alpha}$. Then,  for $0<t<t_0$,  
\begin{align}\label{relation-g-tilde} \widetilde{f_t}(x)&= \left( 1-\alpha [(\beta_{1}t+1)\varphi^*]^*(x)\right)^{\frac{1}{\alpha}}=g(s),  \\ \widehat{f_{t}}(x)& = \big( 1-\alpha [(1+\beta_{1}t)\varphi^*]^*(x)+ \alpha\beta_{2}t\big)^{\frac{1}{\alpha}}=g(\tau) \label{define-f-hat} \end{align} are both well-defined functions.  

Without loss of generality, we always let $0<t<t_0.$    
 The convexity of $g$ in $(-1,+\infty)$ implies
\begin{align} \label{basic-convex-pro} g'(s)(\tau-s)\leq g(\tau)-g(s)\leq g'(\tau)(\tau-s).\end{align}  Thus, for any $x\in \Rn$ and for $t>0$,   by \eqref{relation-g-tilde},  \eqref{define-f-hat} and  \eqref{basic-convex-pro}, we have, \begin{align*}
      \frac{g'(s)(\tau-s)}{t}\leq\frac{\widehat{f_t}(x)-\widetilde{f_t}(x)}{t}  & =\frac{g(\tau)-g(s) }{t} \leq  \frac{g'(\tau)(\tau-s)}{t}. 
\end{align*}   
It follows from  \eqref{definition-s-x} and \eqref{definition-t-x} that 
\begin{align}\label{bound-left-1}
    \beta_{2}\Big(1-\alpha [(1+\beta_{1}t)\varphi^*]^*(x)\Big)^{\frac{1}{\alpha}-1} &\leq \frac{\widehat{f_t}(x)-\widetilde{f_t}(x)}{t} \nonumber \\ &\leq\!\beta_{2}\Big(1-\alpha [(1+\beta_{1}t)\varphi^*]^*(x)+ \alpha\beta_{2}t\Big)^{\frac{1}{\alpha}-1}.
\end{align} 
According to the lower semi-continuity of $\varphi$ with  $o\in \mathrm{int}(\dom\varphi)$ and \cite[Lemma 1.6.11]{Sch13}, we have, for $x\in \overline{\dom \varphi}$, 
    \begin{align}\label{continuity-star}
    \lim_{t\to0^{+}}((1+t\beta_1)\varphi^*)^*(x)
    =\lim_{t\to0^{+}}\Big((1+\beta_1t)\varphi\Big(\frac{x}{1+\beta_1t}\Big)\Big)=\varphi(x).
    \end{align}  Taking the limit  in \eqref{bound-left-1} as $t\rightarrow 0^+$, we get \begin{align}\label{bound-left-limit-1}
     \lim_{t\rightarrow 0^+} \frac{\widehat{f_t}(x)-\widetilde{f_t}(x)}{t} =\beta_{2}\big(1-\alpha \varphi(x)\big)^{\frac{1}{\alpha}-1}=\beta_2 f^{1-\alpha}.
\end{align}

Following the proofs of \eqref{t-beta} and \eqref{mono}, we can get the following by using the monotone convergence theorem: 
\begin{align}\label{leftdominated} 
       \lim_{t\to 0^+}\!\int_{\Rn}\!\!\Big(\!1\!-\!\alpha [(1\!+\!\beta_{1}t)\varphi^*]^*(x)\!\Big)^{\frac{1}{\alpha}-1}\!\,dx 
       &= \!\! \int_{\Rn}\lim_{t\to 0^+}\!\Big(\!1\!-\!\alpha [(1+\beta_{1}t)\varphi^*]^*(x)\!\Big)^{\frac{1}{\alpha}-1}\!\,dx\nonumber\\
       &=\int_{\Rn}\Big(1-\alpha \varphi(x)\Big)^{\frac{1}{\alpha}-1}\,dx.
       \end{align}
Let  $0<t_1<t_0$ be a constant small enough such that  $2|\alpha\beta_2|t_1\leq 1$ (in the case when $\beta_2=0$, we just let $t_1=t_0$).  For all $0<t<t_1$, we have  
\begin{align*}
0&\leq \Big(1-\alpha [(1+\beta_{1}t)\varphi^*]^*(x)+ \alpha\beta_{2}t\Big)^{\frac{1}{\alpha}-1}\\
&\leq
\Big(1-\alpha [(1+\beta_{1}t)\varphi^*]^*(x)-
\frac{1}{2}\Big)^{\frac{1}{\alpha}-1}\\ &=\Big(
\frac{1}{2}-\alpha [(1+\beta_{1}t)\varphi^*]^*(x)\Big)^{\frac{1}{\alpha}-1}:=h_t(x).
\end{align*} Again, following the proofs of \eqref{t-beta} and \eqref{mono},  by using the monotone convergence theorem, we can get the following:
       \begin{align*}\label{3.6} 
       \lim_{t\to 0^+}\int_{\Rn}h_t(x)\,dx = \int_{\Rn}\lim_{t\to 0^+}h_t(x)\,dx<+\infty.
       \end{align*} It follows from  the generalized  dominated convergence theorem that \begin{align*}
           \lim_{t\to 0^+}\!\int_{\Rn}\!\!\! \Big(\!1\!-\!\alpha [(1\!+\!\beta_{1}t)\varphi^*]^*\!+\! \alpha\beta_{2}t\!\Big)^{\frac{1}{\alpha}-1}\!\! dx \!=\! \!\int_{\Rn}\lim_{t\to 0^+}\!\Big(\!1\!-\!\alpha [(1\!+\!\beta_{1}t)\varphi^*]^*\!+\! \alpha\beta_{2}t\!\Big)^{\frac{1}{\alpha}-1}\!\!dx.
       \end{align*}   Together with  \eqref{bound-left-1},   \eqref{bound-left-limit-1}, \eqref{leftdominated} and  the generalized  dominated convergence theorem, we get   \begin{align*}
     \lim_{t\to 0^{+}}\int_{\Rn}\frac{\widehat{{f}_{t}}(x)-\widetilde {{f}_{t}}(x)}{t}\,dx  =\int_{\Rn}\lim_{t\to 0^{+}}\frac{\widehat{{f}_{t}}(x)-\widetilde {{f}_{t}}(x)}{t}\,dx
=\beta_{2}\int_{\Rn}f(x)^{1-\alpha}\,dx.
    \end{align*}   
 Together with (\ref{3.1}), we get 
    \begin{align*}
        \lim_{t\to 0^{+}}\int_{\Rn}\frac{\widehat{{f}_{t}}(x)-f(x)}{t}\,dx  
        &=\lim_{t\to 0^{+}}\int_{\Rn}\frac{\widehat{{f}_{t}}(x)-\widetilde {{f}_{t}}(x)}{t}\,dx+\lim_{t\to 0^{+}}\int_{\Rn}\frac{\widetilde {{f}_{t}}(x)-f(x)}{t}\,dx\\ &=\beta_{2}\int_{\Rn}f(x)^{1-\alpha}\,dx+\beta_{1}\JD(f,f).
    \end{align*} This concludes the proof of \eqref{3.4}. The fact that  $\mathcal{I}_{\alpha}(f)\in (-\infty,+\infty)$ follows immediately from   Lemma \ref{Lem2.2} and Proposition \ref{VarFormulaWithItself}.  \end{proof}
The following corollary is a direct consequence of  Definition \ref{Definition-first-variation} and Theorem \ref{prop5}.  

    \begin{corollary} 
    Let $-\frac{1}{n}<\alpha < 0$ and $f \in \CA(\Rn)$ with $o\in \mathrm{int}(K_f)$. Let $\beta_{1}>  0$ and $\beta_{2}\in\R$ be constants  such that the function  \begin{align*} \widehat{f} := \big(1-\alpha (\beta_{1}\varphi^*+\beta_{2} )^* \big)^{\frac{1}{\alpha}} 
\end{align*} 
    is well-defined.   Then,  the following formula holds: 
    \begin{equation*} 
        \JD(f,\widehat{f}) = \beta_{1}\JD(f,f) + \beta_{2}\int_{\Rn}f(x)^{1-\alpha}\,dx.
    \end{equation*}
    Moreover, $\JD(f,\widehat{f})\in (-\infty,+\infty)$.
\end{corollary}
The following lemma is needed.
\begin{lemma}\label{lem6}
Let $-\frac{1}{n}<\alpha < 0$, $\beta_{1}\geq  0$ and $\beta_{2}\in\R$  be given constants. Assume that  $f \in \CA(\Rn)$ with $o\in \mathrm{int}(K_f)$. For $t>0$ small enough, let   $\widehat{f_t}$ be defined as in \eqref{define-f-hat-1}.   Then, the following formula holds: 
    \begin{equation}\label{3.7}
        \lim_{t\to 0^{+}}\int_{S^{n-1}}E_{t}(u)\,du = \int_{S^{n-1}}\lim_{t\to 0^{+}} E_{t}(u)\,du,
    \end{equation}
where, for $u\in \sphere$, \begin{align}\label{etu}
E_{t}(u) = \int_{0}^{+\infty}\frac{\widehat{f_{t}}(ru)-f(ru)}{t}r^{n-1}\,dr.
\end{align}
\end{lemma}
\begin{proof}  By repeating the proof of Theorem \ref{prop5} verbatim, we get  
    \begin{align} \label{etulim}         
        \lim_{t\to 0^{+}}\!E_{t}(u)  &\!=\! \lim_{t\to 0^{+}}\int_{0}^{+\infty}\frac{\widehat{f_{t}}(ru)-f(ru)}{t}r^{n-1}\,dr\nonumber\\
          &\!=\! \beta_{1}\!\!  \int_{0}^{+\infty}\!\!\! \Big(nf(ru)\! -\!\varphi(ru)f(ru)^{1-\alpha}\Big)r^{n-1}dr\!   +\! \beta_{2}\!\! \int_{0}^{+\infty}\!\!\! f(ru)^{1-\alpha}r^{n-1}dr. 
        \end{align} Consequently, the  polar coordinates formula yields 
        \begin{align*}
         \int_{S^{n-1}}\lim_{t\to 0^{+}}E_{t}(u)\,du &= \beta_{1}\int_{S^{n-1}}\int_{0}^{+\infty}\!\!\! \Big(nf(ru)\! -\!\varphi(ru)f(ru)^{1-\alpha}\Big)r^{n-1}\,dr\,du\nonumber\\
         &\quad\ \ + \beta_{2}\int_{S^{n-1}}\!\int_{0}^{+\infty}\!\!\! f(ru)^{1-\alpha}r^{n-1}\,dr\,du\nonumber\\
        &=\beta_{1}\left(n\int_{\Rn}f(x)\,dx\! -\! \int_{\Rn}\varphi(x)f(x)^{1-\alpha}\,dx\right) + \beta_{2}\int_{\Rn}f(x)^{1-\alpha}\,dx \nonumber \\ 
     &=\beta_{1}\JD(f,f) + \beta_{2}\int_{\Rn}f(x)^{1-\alpha}\,dx, \label{3.81}
        \end{align*} where the last equality follows from \eqref{3.1-no-beta}.  
    Due to the polar coordinates formula and Theorem \ref{prop5}, we have
        \begin{align*}
            \lim_{t\to 0^{+}}\int_{S^{n-1}}E_{t}(u)\,du &=\lim_{t\to 0^{+}}\int_{S^{n-1}}\int_{0}^{+\infty}\frac{\widehat{f_{t}}(ru)-f(ru)}{t}r^{n-1}\,dr\,du\nonumber\\
            &=\lim_{t\to 0^{+}}\int_{\Rn}\frac{\widehat{f_{t}}(x)-f(x)}{t}\,dx\nonumber\\
            &=\beta_{1}\JD(f,f) + \beta_{2}\int_{\Rn}f(x)^{1-\alpha}\,dx\nonumber\\ &=\int_{S^{n-1}}\lim_{t\to 0^{+}}E_{t}(u)\,du. 
        \end{align*} This completes the proof of formula \eqref{3.7}. 
\end{proof}

Assume  $o\in \mathrm{int}(K_f)$. For $u\in \sphere$, let  \begin{align*}
    \rho_{K_f}(u) &= \sup\{t >0 : tu \in K_f\} \ \ \mathrm{and}\ \  \Omega_f =\{u\in S^{n-1}:\rho_{K_f}(u)<+\infty\}. 
\end{align*}

\begin{lemma}\label{lem7.1} Let $-\frac{1}{n}<\alpha < 0$, and $f=(1-\alpha\varphi)^{\frac{1}{\alpha}}\in \CA(\Rn) $ with $o\in \mathrm{int}(K_f)$. Assume that $\psi\in \conv(\Rn)$ and there exist constants $\beta_{1}> 0$ and $\beta_{2}\in\R$ such that
   \begin{equation}\label{3.11}
    -\infty<\inf\psi^*\leq\psi^*\leq
    \beta_1\varphi^*+\beta_2\quad\mathrm{on}\quad\Rn.
    \end{equation}  For $t>0$ small enough, let  \begin{align} f_{t}:=\big( 1-\alpha (\varphi^*+t\psi^*)^*\big)^{\frac{1}{\alpha}}. \label{define-f-1-1}  \end{align} 
 Recall that $\dpsi = \overline{\dom \varphi}$. Then, for almost every $u\in \Omega_f$, one has
    \begin{align*}
        \lim_{t\to0^+}\frac{\rho_{K_{f_t}}(u)-\rho_{K_f}(u)}t=\frac{h_{\dpsi}(\nu_{K_f}(\rho_{K_f}(u)u))}{\langle u,\nu_{K_f}(\rho_{K_f}(u)u)\rangle}.
    \end{align*}
\end{lemma}
\begin{proof}
    The proof follows from \cite[Lemma 5.3]{FYZ24} verbatim and is omitted.
\end{proof}
Indeed, condition \eqref{3.11} implies that $o\in \dom \psi$, due to fact that $\psi(o)=-\inf\psi^*<\infty.$  
Note that $E_{t}(u)$ in \eqref{etu} can be written as follows:
\begin{align} 
E_{t}(u)&= \int_{0}^{+\infty}\frac{\widehat{f_{t}}(ru)-f(ru)}{t}r^{n-1}\,dr   \ \ \mathrm{for} \ \ u\notin \Omega_f, \label{def-etu-notin} \\E_{t}(u)&= \int_{0}^{\rho_{K_{f}}(u)}\frac{\widehat{f_{t}}(ru)-f(ru)}{t}r^{n-1}\,dr \!+ \!\int_{\rho_{K_{f}}(u)}^{\rho_{K_{\widehat{f_{t}}}}(u)}\frac{\widehat{f_{t}}(ru)}{t}r^{n-1}\,dr \ \ \mathrm{for} \ \ u\in \Omega_f. \label{def-etu-in} 
\end{align} For convenience, we often only use the expression of \eqref{def-etu-in} for $E_t(\cdot)$; and when $u\notin \Omega_f$, it is understood as in \eqref{def-etu-notin}.

\begin{lemma}\label{lem7} 
Let $-\frac{1}{n}<\alpha < 0$, $\beta_{1} \geq  0$ and $\beta_{2}\in\R$  be given constants. Assume that  $f \in \CA(\Rn)$ with $o\in \mathrm{int}(K_f)$. For $t>0$ small enough, let   $\widehat{f_t}$ be given in \eqref{define-f-hat-1}.    Then, for almost every $u\in \sphere$, one has
   \begin{equation}\label{3.18}
       \lim_{t\to0^{+}}\int_{0}^{\rho_{K_{f}}(u)}\!\! \frac{\widehat{f_{t}}(ru)\!-\!f(ru)}{t}r^{n-1}\,dr=\!\!\int_{0}^{\rho_{K_{f}}(u)}\!\!\! \!\! \lim_{t\to 0^{+}}\frac{\widehat{f_{t}}(ru)\!-\! f(ru)}{t}r^{n-1}\,dr<+\infty.
    \end{equation} 
\end{lemma}
\begin{proof} Recall that $\widehat{f_{t}}$ is well defined if $0<t<t_0$ for some small fixed $t_{0}$. For $u\in \sphere$, $0<t<t_0$ and $r>0,$ let 
    $$\mathcal{E}_{t,u}(r):=\frac{\widehat{f_{t}}(ru)-f(ru)}{t}r^{n-1}.$$  
It can be checked from \eqref{continuity-star} that, for almost all $u\in \sphere$ and all $0<r<\rho_{K_f}(u)$, we have 
\begin{align}\label{widehat}
    \lim_{t\to0^{+}}\widehat{f_{t}}(ru)=f(ru).
    \end{align}
Note that $\widehat{f_{t}}$ in \eqref{define-f-hat-1} can be rewritten as follows:  \begin{align} \label{lemmadiff}
            \widehat{f_{t}}   = \big( 1-\alpha [(1+\beta_{1}t)\varphi^*]^*+ \alpha\beta_{2}t\big)^{\frac{1}{\alpha}} =\Big( 1-\alpha \big[\varphi^* +t (\beta_{1}\varphi^*+\beta_{2})\big]^* \Big)^{\frac{1}{\alpha}}.
        \end{align} 
When $\beta_1=0$, we have $\widehat{f_{t}} =\big( 1-\alpha \varphi +\alpha\beta_{2} t  \big)^{\frac{1}{\alpha}}$, and from \eqref{etulim} and \eqref{widehat}, we have
\begin{align*}
\int_{0}^{\rho_{K_{f}}(u)}\lim_{t\to0^{+}}\mathcal{E}_{t,u}(r)\,dr&=\int_{0}^{\rho_{K_{f}}(u)}\beta_{2}f(ru)^{1-\alpha}r^{n-1}\,dr\nonumber\\
&= \lim_{t\to 0^{+}} E_{t}(u)=\lim_{t\to0^{+}}\int_{0}^{\rho_{K_{f}}(u)}\mathcal{E}_{t,u}(r)\,dr.
\end{align*}

Next, assume that $\beta_1$ is positive. Applying Lemma \ref{Berman's formula} to $\varphi^*$ and $\beta_{1}\varphi^*+\beta_{2}$, it follows from  \eqref{widehat} and \eqref{lemmadiff} that
    \begin{align}\label{3.13}
            \lim_{t\to 0^+}\mathcal{E}_{t,u}(r)
            &=(\beta_{1}\varphi^*+\beta_{2})(\nabla \varphi(ru))(1-\alpha \varphi(ru))^{\frac{1}{\alpha}-1}r^{n-1}\nonumber\\
            &=\Big(\beta_{1}\varphi^*(\nabla \varphi(ru))+\beta_{2}\Big)f(ru)^{1-\alpha}r^{n-1}
        \end{align} holds for almost all $u\in \sphere$ and all $0<r<\rho_{K_f}(u).$ This further implies that   \begin{align}\label{omega-f}
    \int_{0}^{\rho_{K_{f}}(u)}\lim_{t\to0^{+}}\mathcal{E}_{t,u}(r)\,dr=\int_{0}^{\rho_{K_{f}}(u)}\Big(\beta_{1}\varphi^*(\nabla \varphi(ru))+\beta_{2}\Big)f(ru)^{1-\alpha}r^{n-1}\,dr.
   \end{align} 
    Let us now compute
$$\lim_{t\to0^{+}}\int_{0}^{\rho_{K_{f}}(u)}\mathcal{E}_{t,u}(r)\,dr.$$  It can be checked that, for almost all $u\in \sphere$ and all $0<r<\rho_{K_f}(u),$
    \begin{align*}
        \frac{d}{dr}\left( f(ru)r^{n}\right)&=\frac{d}{dr}\Big( (1-\alpha\varphi(ru))^{\frac{1}{\alpha}}r^{n}\Big) \nonumber  \\& =  \Big(nf(ru)\!-\!\left\langle ru,\nabla\varphi(ru)\right\rangle f(ru)^{1-\alpha}\Big)r^{n-1}. \end{align*} 
        Together with \eqref{basicequ}, for almost every $u\in \sphere$,  formula \eqref{etulim} can be written as 
    \begin{align}  \label{detulim}
    \lim_{t\to 0^{+}}\!\! E_{t}(u) 
     &\!=  \beta_{1}n\int_{0}^{\rho_{K_{f}}(u)}f(ru)r^{n-1}\,dr-\beta_{1}\int_{0}^{\rho_{K_{f}}(u)}\left\langle ru,\nabla\varphi(ru)\right\rangle f(ru)^{1-\alpha}r^{n-1}\,dr\nonumber\\
     &\quad +\int_{0}^{\rho_{K_{f}}(u)}\Big(\beta_{1}\varphi^*(\nabla \varphi(ru))+\beta_{2}\Big)f(ru)^{1-\alpha}r^{n-1}\,dr\nonumber\\
     &\!=  \beta_{1}\!\!\int_{0}^{\rho_{K_{f}}\!(u)}\!\!\frac{d}{dr}\left( f(ru)r^{n}\right)\!+\!\int_{0}^{\rho_{K_{f}}(u)}\!\!\Big(\beta_{1}\varphi^*(\nabla \varphi(ru))\!+\!\beta_{2}\Big)f(ru)^{1\!-\!\alpha}r^{n\!-\!1}\,dr.\end{align}
     
         First of all, let $u\in \Omega_f$. From \eqref{detulim}, we have\begin{align} \label{finite-u}     
        \lim_{t\to 0^{+}}\!E_{t}(u)  =  \beta_{1}\rho_{K_{f}}\!(u)^{n}\!f(\rho_{K_{f}}\!(u)u)
     \!+\!\!\int_{0}^{\rho_{K_{f}}\!(u)}\!\!\! \!\!\! \big(\beta_{1}\varphi^*(\nabla \varphi(ru))\!+\!\beta_{2}\big)f(ru)^{1\!-\!\alpha}r^{n\!-\!1}\,dr. \end{align}  
It is easy to check that ${\overline {\dom (\beta_1\varphi^*)^*}}= \beta_{1}{\overline {\dom \varphi}}=\beta_1K_f$.
    Applying the mean value theorem, \eqref{widehat} and Lemma \ref{lem7.1}, for almost every $u\in\Omega_f$, we have
    \begin{align}\label{radialdiff}
        \lim_{t\to 0^{+}}\int_{\rho_{K_{f}}(u)}^{\rho_{K_{\widehat{f}_{t}}}(u)}\frac{\widehat{f_{t}}(ru)}{t}r^{n-1}\,dr
        &= \lim_{t\to 0^{+}}\widehat{f_{t}}(r(t,u)u)r(t,u)^{n-1}\frac{h_{\beta_{1}K_{f}}(\nu_{K_{f}}(\rho_{K_{f}}(u)u))}{\left\langle u,\nu_{K_{f}}(\rho_{K_{f}}(u)u))\right\rangle}\nonumber\\
        &=\beta_{1}\rho_{K_{f}}(u)^{n}f(\rho_{K_{f}}(u)u),
    \end{align}
     where $r(t,u)\in[\rho_{K_f}(u),(1+t\beta_{1})\rho_{K_f}(u)]$ satisfies  $\lim_{t\to 0^{+}}r(t,u)=\rho_{K_f}(u).$   
    Hence, for almost every $u\in \Omega_{f}$, it follows from \eqref{def-etu-in} and \eqref{finite-u} that
    
         \begin{align} \label{3.15}
         \lim_{t\to0^{+}}\int_{0}^{\rho_{K_{f}}(u)}\mathcal{E}_{t,u}(r)\,dr &= \lim_{t\to 0^{+}}E_{t}(u)-\lim_{t\to 0^{+}}\int_{\rho_{K_{f}}(u)}^{\rho_{K_{\widehat{f}_{t}}}(u)}\frac{\widehat{f_{t}}(ru)}{t}r^{n-1}\,dr\nonumber\\
        &=\int_{0}^{\rho_{K_{f}}(u)}\Big(\beta_{1}\varphi^*(\nabla \varphi(ru))+\beta_{2}\Big)f(ru)^{1-\alpha}r^{n-1}\,dr.
    \end{align}

       Now let $u\notin\Omega_f$ which implies $\rho_{K_f}(u)=+\infty$. 
   From \eqref{radial}, we have, for $r>r_0$ ($r_0$ is a constant), $$f(ru)\leq \Big(\frac{-\alpha ar}{2}\Big)^{\frac{1}{\alpha}}.$$
    Due to  $-\frac{1}{n}<\alpha<0$, we have
    \begin{equation}\label{0}
        \lim_{r\to+\infty}f(ru)r^n \leq \lim_{r\to+\infty}\Big(\frac{-\alpha ar}{2}\Big)^{\frac{1}{\alpha}}r^n=0.
    \end{equation}
   It follows from \eqref{detulim} and (\ref{0}) that, for almost all $u\notin \Omega_{f}$, we have
 \begin{align*} 
    \lim_{t\to 0^{+}}E_{t}(u) &=\beta_{1}\int_{0}^{+\infty}\frac{d}{dr}\left( f(ru)r^{n}\right)+\int_{0}^{+\infty}\Big(\beta_{1}\varphi^*(\nabla \varphi(ru))+\beta_{2}\Big)f(ru)^{1-\alpha}r^{n-1}\,dr\nonumber\\
     &=\int_{0}^{+\infty}\Big(\beta_{1}\varphi^*(\nabla \varphi(ru))+\beta_{2}\Big)f(ru)^{1-\alpha}r^{n-1}\,dr.\end{align*}   Thanks to \eqref{def-etu-notin},  we obtain, for almost all $u\notin \Omega_{f}$,
   \begin{align}\label{3.19}
            \lim_{t\to 0^+} \int_{0}^{\rho_{K_{f}}(u)}\!\!\!\mathcal{E}_{t,u}(r)\,dr &=\! \lim_{t\to 0^+} \int_{0}^{+\infty}\mathcal{E}_{t,u}(r)\,dr=\lim_{t\to 0^+}E_t(u)\nonumber\\
            &= \int_{0}^{+\infty}\Big(\beta_{1}\varphi^*(\nabla \varphi(ru))+\beta_{2}\Big)f(ru)^{1-\alpha}r^{n-1}\,dr. 
        \end{align}   
        
        Combining \eqref{omega-f}, (\ref{3.15}), and (\ref{3.19}), we get, for $\beta_1>0$ and almost all $u\in \sphere$, 
    $$\lim_{t\to0^{+}}\int_{0}^{\rho_{K_{f}}(u)}\mathcal{E}_{t,u}(r)\,dr=\int_{0}^{\rho_{K_{f}}(u)}\lim_{t\to 0^{+}}\mathcal{E}_{t,u}(r)\,dr<+\infty.$$ This completes the proof of formula \eqref{3.18}. \end{proof}

We are now in the position to prove the following result. 
\begin{theorem}\label{thm8}
  Let $-\frac{1}{n}<\alpha < 0$, and $f=(1-\alpha\varphi)^{\frac{1}{\alpha}}\in \CA(\Rn) $ with $o\in \mathrm{int}(K_f)$. Assume that $\psi\in \conv(\Rn)$ and there exist constants $\beta_{1}> 0$ and $\beta_{2}\in\R$ satisfying \eqref{3.11}. Let $f_t$ be give in \eqref{define-f-1-1} for $t>0$ small enough. Then,
    \begin{equation}\label{main-formular}
        \lim_{t\to 0^{+}}\!\int_{\Rn}\!\! \frac{{{f}_{t}}(x)\!-\!f(x)}{t}\,dx=\! \int_{\Rn}\!\! \psi^*\!\left(\nabla \varphi(x)\right)\! f(x)^{1-\alpha}\,dx +   \int_{\partial {K_f}}\!\!\!  {h_{\dpsi}\!(\nu_{K_f}\!(x))}f(x)\,d\mathcal{H}^{n-1}(x).
    \end{equation}
\end{theorem}
   
\begin{proof} For $t>0$ small enough, we have \begin{align*}
         \lim_{t\to0^{+}}\int_{\Rn}\frac{f_{t}(x)-f(x)}{t}\,dx
         &= \lim_{t\to0^{+}}\int_{S^{n-1}}\!\int_{0}^{+\infty}\frac{f_{t}(ru)-f(ru)}{t}r^{n-1}\,dr\,du=\lim_{t\to0^{+}} \int_{S^{n-1}}F_{t}(u)\,du,
    \end{align*}  where, for $u\in \sphere$,   $$F_{t}(u):= \int_{0}^{+\infty}\frac{f_{t}(ru)-f(ru)}{t}r^{n-1}\,dr.$$ 
 
   For $u\in \sphere$, let  \begin{align*}
       G_t(u): =  \int_{0}^{+\infty}\frac{\big(1-\alpha \varphi+\alpha t \inf\psi^*\big)^{\frac{1}{\alpha}}(ru)-f(ru)}{t}r^{n-1}\,dr,
\end{align*} where $\inf\psi^* >-\infty$ due to \eqref{3.11}. 
It follows from \eqref{3.11} that $$\varphi^*+t\inf\psi^*\leq \varphi^*+t\psi^*\leq(1+t\beta_1)\varphi^*+t\beta_2.$$ This implies that, for $t>0$ small enough,\begin{align} \label{compare-3-functions}
\big(1-\alpha \varphi+\alpha t \inf\psi^*\big)^{\frac{1}{\alpha}}\leq f_t\leq \widehat{f_{t}}.\end{align} Together with \eqref{widehat}, we get 
     \begin{equation}\label{double1continuity}
     \lim_{t\to0^{+}}{f_{t}}(x)=f(x).
     \end{equation} For $t>0$ small enough, the support of  $\big(1-\alpha \varphi+\alpha t \inf\psi^*\big)^{\frac{1}{\alpha}}$ is equal to $K_f$.  Moreover,  
\begin{align}
 \label{compare-F-G-E}
    G_t\leq F_t\leq E_t\ \ \mathrm{and} \ \ K_f\subset K_{f_t}\subset K_{\widehat{f_{t}}}.
\end{align}   It follows from  Lemma \ref{lem6} and the generalized dominated convergence theorem that
    $$\lim_{t\to 0^{+}}\int_{S^{n-1}}F_{t}(u)\,du = \int_{S^{n-1}}\lim_{t\to 0^{+}}F_{t}(u)\,du.$$
We can rewrite $F_t(\cdot)$ as follows: for $u\in \sphere,$
    \begin{align} F_{t}(u)= \int_{0}^{\rho_{K_{f}}(u)}\frac{f_{t}(ru)-f(ru)}{t}r^{n-1}\,dr + \int_{\rho_{K_{f}}(u)}^{\rho_{K_{f_{t}}}(u)}\frac{f_{t}(ru)}{t}r^{n-1}\,dr,\label{rewrite-F-t}\end{align} where the second integral is treated as $0$ if $u\notin \Omega_f$ (i.e., $\rho_{K_f}(u)=+\infty$)  due to \eqref{compare-F-G-E}. That is, if $u\notin \Omega_{f}$,
     \begin{equation}\label{bounddiff1}
    \int_{\rho_{K_{f}}(u)}^{\rho_{K_{f_{t}}}(u)}\frac{f_{t}(ru)}{t}r^{n-1}\,dr=0.
    \end{equation} Similar to \eqref{radialdiff},  it can be checked from \eqref{double1continuity} that, for almost every $u\in\Omega_f$,   
    \begin{align}\label{bounddiff}
        \lim_{t\to 0^{+}}\int_{\rho_{K_{f}}(u)}^{\rho_{K_{f_{t}}}(u)}\frac{f_{t}(ru)}{t}r^{n-1}\,dr&= \lim_{t\to 0^{+}} f_{t}(r(t,u)u)r(t,u)^{n-1}\lim_{t\to 0^{+}}\frac{\rho_{K_{f_{t}}}(u)-\rho_{K_{f}}(u)}{t}\nonumber\\
        &= f(\rho_{K_{f}}(u)u)\rho_{K_{f}}(u)^{n-1}\frac{h_{\dpsi}(\nu_{K_{f}}(\rho_{K_{f}}(u)u))}{\left\langle u,\nu_{K_{f}}(\rho_{K_{f}}(u)u))\right\rangle},
    \end{align}
    where $r(t,u)\in[\rho_{K_f}(u),\rho_{K_{f_{t}}}(u)]$ such that $ \lim_{t\to 0^{+}}r(t,u)=\rho_{K_f}(u)$.
    
    Next, we will compute, for almost all  $u\in \sphere$:  
    $$\lim_{t\to 0^{+}}\int_{0}^{\rho_{K_{f}}(u)}\frac{f_{t}(ru)-f(ru)}{t}r^{n-1}\,dr.$$ 
    For $t>0$ small enough and $u\in \sphere$, let 
   \begin{align*} \mathcal{F}_{t,u}(r)&=\frac{f_{t}(ru)-f(ru)}{t}r^{n-1}, \\ \mathcal{G}_{t,u}(r)&=\frac{\big(1-\alpha \varphi+\alpha t \inf\psi^*\big)^{\frac{1}{\alpha}}(ru)-f(ru)}{t}r^{n-1}.\end{align*} 
    It follows from \eqref{compare-3-functions} that  
    $\mathcal{G}_{t,u}(r)\leq \mathcal{F}_{t,u}(r)\leq \mathcal{E}_{t,u}(r)$. Together with Lemma \ref{lem7} and the generalized dominated convergence theorem, we have  
     $$\lim_{t\to0^{+}}\int_{0}^{\rho_{K_{f}}(u)}\mathcal{F}_{t,u}(r)\,dr=\int_{0}^{\rho_{K_{f}}(u)}\lim_{t\to 0^{+}}\mathcal{F}_{t,u}(r)\,dr.$$
   Similar to the proof for \eqref{3.13}, by \eqref{double1continuity} and Lemma \ref{Berman's formula} (applying to $\varphi^*$ and $\psi^*$), we have
    \begin{align*}
    \lim_{t\to 0^+}\mathcal{F}_{t,u}(r) =\lim_{t\to 0^+}\frac{f_t(ru)-f(ru)}{t}r^{n-1}  =\psi^*(\nabla \varphi(ru))f(ru)^{1-\alpha}r^{n-1}, 
    \end{align*}  holds for almost all $u\in \sphere$ such that $ru\in \dom\varphi$. Therefore,
    \begin{equation}\label{ft}
    \lim_{t\to0^{+}}\int_{0}^{\rho_{K_{f}}(u)}\mathcal{F}_{t,u}(r)\,dr=\int_{0}^{\rho_{K_{f}}(u)}\psi^*(\nabla \varphi(ru))f(ru)^{1-\alpha}r^{n-1}\,dr.
    \end{equation} 
    
    Combining \eqref{rewrite-F-t}, \eqref{bounddiff1},  \eqref{bounddiff}, and \eqref{ft}, we have  
 \begin{align*}
    \lim_{t\to 0^{+}}\int_{S^{n-1}}F_{t}(u)\,du &= \int_{S^{n-1}}\lim_{t\to 0^{+}}F_{t}(u)\,du\\
    &=\int_{S^{n-1}}\int_{0}^{\rho_{K_{f}}(u)}\psi^*\left(\nabla \varphi(ru)\right) f(ru)^{1-\alpha}r^{n-1}\,dr\,du\\
     &\qquad+ \int_{\Omega_f}f(\rho_{K_{f}}(u)u)\rho_{K_{f}}(u)^{n-1}\frac{h_{\dpsi}(\nu_{K_f}(\rho_{K_f}(u)u))}{\langle u,\nu_{K_f}(\rho_{K_f}(u)u)\rangle}\,du\\
    &=\int_{\Rn}\psi^*\left(\nabla \varphi(x)\right) f(x)^{1-\alpha}\,dx + \int_{\partial {K_f}}{h_{\dpsi}(\nu_{K_f}(x))}f(x)\,d\mathcal{H}^{n-1}(x).
\end{align*} This completes the proof of \eqref{main-formular}.  \end{proof}

Our next result is to relax the conditions $-\infty<\inf \psi^*<\psi^*$ in  \eqref{3.11}  (implying $o\in \dom\psi$) and $o\in \mathrm{int}(K_f)$.
Note that the last inequality in \eqref{3.11}  implies $$\overline{\dom \psi} \subset \beta_1 \overline{\dom\varphi}=\beta_1 K_f.$$ The idea is to translate $\dom\psi$ to $(\dom\psi)-x_0$ for some $x_0\in (\dom\psi)\cap\mathrm{int}(\beta_1K_f)$. Let  \begin{align}\label{def-psi-tilde-varphi} 
    \widetilde{\psi}(x)=\psi(x+x_0) \ \ \mathrm{and} \ \ \widetilde{\varphi}(x)=\varphi\Big(x+\frac{x_0}{\beta_1}\Big).
\end{align} Then, $\dom\widetilde{\psi}=\dom\psi-x_0$ and hence $o\in \dom\widetilde{\psi}.$ It can be checked that $\widetilde{\psi}^*\geq -\widetilde{\psi}(o)=-\psi(x_0)$ and hence \begin{align}\label{tildeinf}
-\infty<\inf\widetilde\psi^*\leq\widetilde\psi^*.
\end{align}
On the other hand, $x_0\in \mathrm{int}(\beta_1 K_f )$ and then $o\in \beta_1 \mathrm{int}(K_f-\frac{x_0}{\beta_1})$. That is, $$o\in \mathrm{int}\Big(K_f-\frac{x_0}{\beta_1}\Big)=\mathrm{int}(\dom\widetilde{\varphi}).$$ Moreover, if $\psi$ satisfies $\psi^*\leq \beta_1\varphi^*+\beta_2$, then, for any $y\in \Rn$, \begin{align*}
    \widetilde{\psi}^*(y)=\psi^*(y)-\langle x_0, y\rangle \nonumber \leq \beta_1\varphi^*(y)+\beta_2-\langle x_0, y\rangle \nonumber=\beta_1 \Big(\varphi^*(y)-\big\langle \frac{x_0}{\beta_1}, y\big\rangle\Big)+\beta_2\nonumber=\beta_1  \widetilde{\varphi}^*(y) +\beta_2.
\end{align*} Together with \eqref{tildeinf}, we see that $\widetilde{\varphi}$ and $\widetilde{\psi}$ satisfy the condition \eqref{3.11}, namely, 
\begin{align*}-\infty<\inf\widetilde\psi^*\leq\widetilde\psi^*\leq\beta_1  \widetilde{\varphi}^* +\beta_2. 
\end{align*} This  allows us to use the result in Theorem \ref{thm8}. 

 We shall need the following lemma.  
\begin{lemma}\label{lem9}
    Let $-\frac{1}{n}<\alpha <0$ and  $f = (1-\alpha\varphi)^{\frac{1}{\alpha}}\in \CA(\Rn)$ with $\mathrm{int}(K_f)\neq\emptyset$.
    Then, for any $u\in \sphere,$
    \begin{equation}\label{3.12}
         \int_{\Rn}\langle u,\nabla \varphi(x)\rangle f(x)^{1-\alpha}\,dx + \int_{\partial {K_f}}\langle u,\nu_{K_f}(x)\rangle f(x)\,d\mathcal{H}^{n-1}(x)=0.
    \end{equation} 
\end{lemma}

\begin{proof}
Choose $x_0\in \mathrm{int}(K_f)$. Let $\overline{\varphi}(x)=\varphi(x+x_0)$  and  $$\overline {f}(x)=(1-\alpha\overline\varphi(x))^{\frac{1}{\alpha}}=(1-\alpha\varphi(x+ x_0))^{\frac{1}{\alpha}}.$$  Then, $o\in \mathrm{int}(K_{\overline f})$ and $\overline\varphi^*(x)=\varphi^*(x)-\langle x_0,x\rangle$. From \eqref{3.1-no-beta}, we have
 \begin{align*}
 \JD(\overline f,\overline f)
 =n\Ja(\overline f) - \!\! \int_{\Rn}\overline\varphi(x)\overline f(x)^{1-\alpha}\,dx =n\Ja(f) -\!\!  \int_{\Rn}\varphi(x)f(x)^{1-\alpha}\,dx =\JD( f, f).
  \end{align*}
Applying Theorem \ref{thm8} to $\overline f$, by $K_{\overline f}= K_{ f}-x_0$ and $z=x+x_0$, it follows that 
\begin{align*}
 \JD(\overline f,\overline f)
 &=\int_{\Rn}\overline\varphi^*\left(\nabla \overline\varphi(x)\right) \overline f(x)^{1-\alpha}\,dx +\! \int_{\partial {K_{\overline f}}}{h_{K_{\overline f}}(\nu_{K_{\overline f}}(x)})\overline f(x)\,d\mathcal{H}^{n-1}(x)\\&=\int_{\Rn}\overline\varphi^*\left(\nabla  \varphi(x+x_0)\right)  f(x+x_0)^{1-\alpha}\,dx \\ & \quad +\! \int_{\partial {(K_{f}-x_0)}}{h_{(K_{ f}-x_0)}(\nu_{(K_{ f}-x_0)}(x)})  f(x+x_0)\,d\mathcal{H}^{n-1}(x)\\&=\int_{\Rn}\overline\varphi^*\left(\nabla  \varphi(z)\right)  f(z)^{1-\alpha}\,dz +\! \int_{\partial {K_{f}}}{h_{(K_{ f}-x_0)}(\nu_{K_{ f}}(z)})  f(z)\,d\mathcal{H}^{n-1}(z)\\
 &=\int_{\Rn}\varphi^*\left(\nabla \varphi(z)\right)  f(z)^{1-\alpha}\,dz + \int_{\partial {K_{ f}}}{h_{K_f}(\nu_{K_{ f}}(z)})f(z)\,d\mathcal{H}^{n-1}(z)\\ 
 &\quad -\int_{\Rn}\langle x_0,\nabla \varphi(z)\rangle f(z)^{1-\alpha}\,dz - \int_{\partial {K_f}}\langle x_0,\nu_{K_f}(z)\rangle f(z)\,d\mathcal{H}^{n-1}(z)\\
 &=\JD( f, f),  \end{align*} which is independent of $x_{0}$. This further implies that,  for any $x_0\in\mathrm{int}(K_f)$,\begin{align}\label{u}\int_{\Rn}\langle x_0,\nabla \varphi(z)\rangle f(z)^{1-\alpha}\,dz + \int_{\partial {K_f}}\langle x_0,\nu_{K_f}(z)\rangle f(z)\,d\mathcal{H}^{n-1}(z)=c, \end{align} where $c$ is a constant independent of $x_0$. Because $x_0\in \mathrm{int}(K_f)$, there exits $\epsilon>0$ such that $x_0+\epsilon B_2^n\subset \mathrm{int}(K_f)$. That is, for $u\in\sphere$, $x_0+\epsilon u\in\mathrm{int}(K_f)$ and thus, 
 \begin{align*}\int_{\Rn}\langle x_0+\epsilon u,\nabla \varphi(z)\rangle f(z)^{1-\alpha}\,dz + \int_{\partial {K_f}}\langle x_0+\epsilon u,\nu_{K_f}(z)\rangle f(z)\,d\mathcal{H}^{n-1}(z)=c. \end{align*}
 Together with \eqref{u}, we get, for any $u\in\sphere$,
\begin{align*}
 \int_{\Rn}\langle u,\nabla \varphi(z)\rangle f(z)^{1-\alpha}\,dz + \int_{\partial {K_f}}\langle u,\nu_{K_f}(z)\rangle f(z)\,d\mathcal{H}^{n-1}(z)=0.
 \end{align*} 
This completes the proof of \eqref{3.12}.
\end{proof}   
Our main result in this section is summarized in the following theorem. 
\begin{theorem}\label{thm10}
     Let $-\frac{1}{n}<\alpha < 0$, and $f=(1-\alpha\varphi)^{\frac{1}{\alpha}}\in \CA(\Rn) $. Assume that $\psi\in \conv(\Rn)$, and there exist constants $\beta_{1}> 0$ and $\beta_{2}\in\R$  satisfying
$$   (\dom \psi) \cap\mathrm{int}(\beta_1K_f)\neq\emptyset  \ \ \mathrm{and}\ \ \psi^*\leq\beta_1\varphi^*+\beta_2.$$ Let  $f_t$ be  given in \eqref{define-f-1-1} for $t>0$ small enough.
    Then,
\begin{equation}\label{3.12-1}
        \lim_{t\to 0^{+}}\!\int_{\Rn}\!\! \frac{{{f}_{t}}(x)\!-\!f(x)}{t}\,dx= \int_{\Rn}\! \psi^*\!\left(\nabla \varphi(x)\right) f(x)^{1-\alpha}\,dx + \! \int_{\partial {K_f}}\!  {h_{\dpsi}\!(\nu_{K_f}\!(x))}f(x)\,d\mathcal{H}^{n-1}(x).
    \end{equation}
\end{theorem} 

\begin{proof} Let $x_0\in\dom \psi$ and hence $\frac{x_0}{\beta_1}\in \mathrm{int}(K_f)$. Let $\widetilde{\varphi}$ and  $\widetilde{\psi}$ be given in \eqref{def-psi-tilde-varphi}. As claimed above, $\widetilde{\varphi}$ and  $\widetilde{\psi}$  satisfy the condition \eqref{3.11}. Let \begin{align*} \widetilde f(x)&=(1-\alpha\widetilde\varphi(x))^{\frac{1}{\alpha}}=\Big(1-\alpha\varphi\Big(x+\frac{x_0}{\beta_1}\Big)\Big)^{\frac{1}{\alpha}},\\ \widetilde{f}_{t, \widetilde{\psi}} &=\big( 1-\alpha (\widetilde{\varphi}^*+t\widetilde{\psi}^*)^*\big)^{\frac{1}{\alpha}}. \end{align*}  It can be easily checked that, for all $t>0$ small enough, \begin{align} \int _{\Rn} f(x)\,dx=\int _{\Rn} \widetilde{f}(x)\,dx\ \ \mathrm{and}\ \ \int _{\Rn} f_t(x)\,dx=\int _{\Rn} \widetilde{f}_{t, \widetilde{\psi}}(x)\,dx. \label{equal-two-total mass}\end{align}

 Applying Theorem \ref{thm8} with $\widetilde f$ and $\widetilde \psi$, it follows from $K_{\widetilde f}= K_{ f}-\frac{x_0}{\beta_1}$,  $ \dom\widetilde\psi= \dom\psi-x_0$ and $z=x+\frac{x_0}{\beta_1}$ that 
\begin{align*}
        \lim_{t\to 0^{+}}\! \int_{\Rn}\!\!\!\!\! \frac{\widetilde{f}_{t, \widetilde{\psi}}(x)\!-\!\widetilde f(x)}{t}dx\! &=\!\!  \int_{\Rn}\widetilde\psi^*(\nabla \widetilde\varphi(x)) \widetilde f(x)^{1-\alpha}\,dx +\! \int_{\partial {K_{\widetilde f}}}{h_{\dwpsi}(\nu_{K_{\widetilde f}}(x))}\widetilde f(x)\,d\mathcal{H}^{n-1}(x)\nonumber\\&= \!\!\int_{\Rn} \widetilde\psi^*\bigg(\!\nabla  \varphi\Big(x+\frac{x_0}{\beta_1}\Big)\!\bigg) f\bigg(x+\frac{x_0}{\beta_1}\bigg)^{1-\alpha}\,dx \nonumber \\& \quad +\! \int_{\partial (K_{ f}-\frac{x_0}{\beta_1})}{h_{(\dpsi-x_0)}(\nu_{(K_{ f}-\frac{x_0}{\beta_1})}(x))}f\bigg(x+\frac{x_0}{\beta_1}\bigg)\,d\mathcal{H}^{n-1}(x)\nonumber\\
        &= \!\!\int_{\Rn} \!\!\widetilde\psi^*\big(\nabla  \varphi (z )\big) f(z)^{1-\alpha}dz \!+\! \!\int_{\partial K_{ f}}\!\!\!\!h_{(\dpsi-x_0)}(\nu_{K_{f}}\!(z))f(z)\,d\mathcal{H}^{n-1}(z)\nonumber\\
        &=  \!\!\int_{\Rn}\psi^*\left(\nabla \varphi(z)\right)  f(z)^{1-\alpha}\,dz + \int_{\partial {K_{ f}}}{h_{\dpsi}(\nu_{K_{ f}}(z)})f(z)\,d\mathcal{H}^{n-1}(z)\nonumber\\
        &\quad \!-\!\int_{\Rn}\langle x_0,\nabla \varphi(z)\rangle f(z)^{1-\alpha}\,dz - \!\int_{\partial {K_f}}\!\langle x_0,\nu_{K_f}(z)\rangle f(z)\,d\mathcal{H}^{n-1}(z)\nonumber \\ &=\!\!  \int_{\Rn}\psi^*\left(\nabla \varphi(z)\right)  f(z)^{1-\alpha}\,dz +\! \int_{\partial {K_{ f}}}{h_{\dpsi}(\nu_{K_{ f}}(z)})f(z)\,d\mathcal{H}^{n-1}(z),
    \end{align*} where in the last equality, we have used \eqref{3.12}. Combining with \eqref{equal-two-total mass}, we get 
\begin{align*}
        \lim_{t\to 0^{+}} \int_{\Rn}\!\!  \frac{f_{t}(x)-f(x)}{t}\,dx& =\!\lim_{t\to 0^{+}} \int_{\Rn}\!\! \frac{\widetilde{f}_{t, \widetilde{\psi}}(x)\!-\!\widetilde f(x)}{t}\,dx \\  &= \! \int_{\Rn}\!\!\! \psi^*(\nabla \varphi(x))  f(x)^{1-\alpha}\,dx + \int_{\partial {K_{ f}}}\!\!\! {h_{\dpsi}(\nu_{K_{ f}}(x)})f(x)\,d\mathcal{H}^{n-1}(x).
    \end{align*} This completes the proof of \eqref{3.12-1}.\end{proof}

The following corollary follows directly from Theorem \ref{thm10} and Definition \ref{Definition-first-variation}.

\begin{theorem}\label{MainTheorem1}
   Let $-\frac{1}{n}<\alpha < 0$, and $f=(1-\alpha\varphi)^{\frac{1}{\alpha}}\in \CA(\Rn) $. Assume that $g = (1-\alpha\psi)^{\frac{1}{\alpha}}\in \mathcal{C}_{\alpha}^{+}(\mathbb{R}^{n})$, where $\psi\in \conv(\Rn)$, and there exist constants $\beta_{1}> 0$ and $\beta_{2}\in\R$  satisfying $$ (\dom \psi) \cap\mathrm{int}(\beta_1K_f)\neq\emptyset
 \ \ \mathrm{and}\ \ \psi^*\leq\beta_1\varphi^*+\beta_2.$$ Then, the following equality holds:
\begin{align*}
       \JD(f,g)= \int_{\Rn} \psi^*\!\left(\nabla \varphi(x)\right)\! f(x)^{1-\alpha}\,dx + \! \! \int_{\partial {K_f}} {h_{K_g}(\nu_{K_f}(x))}f(x)\,d\mathcal{H}^{n-1}(x).
    \end{align*}
\end{theorem} 

Theorem \ref{thm10} motivates the following definitions for two measures associated with the $\alpha$-concave function, one defined on $\Rn$ and the other   defined  on  $\sphere$.

\begin{definition}\label{}
Let $-\frac{1}{n}<\alpha < 0$ and $f=(1-\alpha\varphi)^{\frac{1}{\alpha}}\in \CA(\Rn)$. 

\vskip 2mm \noindent (i) Define the Borel measure $\mu_{\alpha}(f,\cdot)$, which is called the Euclidean surface area measure  of the  $\alpha$-concave function $f$,  as follows: for every Borel subset $\eta\subset \mathbb{R}^n$,
\begin{align}\label{euclidean-moment-measure}
\mu_{\alpha}(f,\eta)=\int_{\big\{x\in\mathbb{R}^n:\nabla\varphi(x)\in \eta\big\}} f(x)^{1-\alpha}\,dx.
\end{align}
where $\nabla\varphi$ is the gradient of $\varphi$. That is,  $\mu_{\alpha}(f, \cdot)=(\nabla\varphi)_{\sharp}(f^{1-\alpha}\,dx)$, the push-forward measure of $f^{1-\alpha}\,dx=  (1-\alpha\varphi)^{\frac{1-\alpha}{\alpha}}\,dx $ under the map $\nabla \varphi$,

\vskip 2mm \noindent (ii)  Define the Borel measure $\nu_{\alpha}(f,\cdot)$, which is called the spherical surface area measure of the  $\alpha$-concave function $f$,  as follows:  for every Borel subset $\eta\subset S^{n-1}$,
\begin{align} \label{spherical-moment-measure}
\nu_{\alpha}(f,\eta)=\int_{\big\{x\in{\partial{K_f}}:\ \nu_{K_f}(x)\in \eta\big\}}f(x)\,d\mathcal{H}^{n-1}(x).
\end{align} That is,   $\nu_{\alpha}(f, \cdot)=(\nu_{K_{f}})_{\sharp}\big(f\,d\mathcal{H}^{n-1}|_{\partial K_{f}}\big).$ 
\end{definition}   When $\varphi\geq 0$ satisfies that  $\nabla \varphi: \mathrm{int}(\dom\varphi) \rightarrow \mathrm{int}(\dom\varphi^*)$ is smooth and bijective, the measure $\mu_{\alpha}(f, \cdot)$ is absolutely continuous with respect to the Lebesgure measure and satisfies \begin{align}\label{density-eucldiean measure}
\frac{d\mu_{\alpha}(f, y)}{dy}=\det(\nabla^2\varphi^*(y)) \Big(1-\alpha \big(\langle y, \nabla \varphi^*(y)\rangle -\varphi^*(y)\big)\Big)^{\frac{1-\alpha}{\alpha}} \ \ \mathrm{for} \ \  y\in \Rn. \end{align} In fact, under the above assumptions on $\varphi,$ one has,  \begin{align}
\varphi^*(\nabla \varphi(x))= \langle x , \nabla \varphi(x) \rangle -\varphi(x). \label{def-dual-2}
\end{align} It follows from \eqref{euclidean-moment-measure} that \begin{align}\label{equivalent-f-varphi}
    \int_{\Rn} g(y)\,d\mu_{\alpha}(f, y)=\int_{\Rn}g(\nabla \varphi(x))f(x)^{1-\alpha}\,dx
\end{align} holds for every Borel function $g$ such that $g \in L ^ 1 (\mu_{\alpha} (f,\cdot))$  or $g$ is non-negative. By  letting $y=\nabla \varphi(x)$ (equivalently, $x=\nabla \varphi^*(y)$),  it follows from \eqref{def-dual-2} and \eqref{equivalent-f-varphi} that 
\begin{align*}
\int_{\Rn}g(y)\, d\mu_{\alpha} (f, y)& =\int_{\Rn}g(\nabla \varphi(x))(1-\alpha \varphi(x))^{\frac{1-\alpha}{\alpha}} \,dx\\&
=\int_{\Rn}g(y)  \det(\nabla^2\varphi^*(y)) \Big(1-\alpha \big(\langle y, \nabla \varphi^*(y)\rangle -\varphi^*(y)\big)\Big)^{\frac{1-\alpha}{\alpha}}  \,dy,
\end{align*} which further implies \eqref{density-eucldiean measure}, as desired.

Having defined the Euclidean and spherical surface area measures of an $\alpha$-concave function (for $-\frac{1}{n}< \alpha< 0$), it is natural and interesting to study the following Minkowski problem related to $\alpha$-concave functions. 
\begin{problem}[The functional Minkowski problem for $\alpha$-concave functions]\label{Mink-double-1} Let $ -\frac{1}{n}<\alpha< 0$, and let $\mu$ and $\nu$ be two Borel measures defined on $\Rn$ and respectively, on $\sphere$. Find the necessary and/or sufficient conditions on $\mu$ and $\nu$, such that there exists an $\alpha$-concave function $f$ satisfying 
$$\mu =  \mu_{\alpha}(f,\cdot) \ \ \mathrm{and}\ \ \nu= \nu_{\alpha}(f, \cdot).$$   
\end{problem} When $\nu_{\alpha}(f, \cdot)=0$,   
Problem \ref{Mink-double-1} can be rewritten as follows. 

\begin{problem}[The Euclidean functional Minkowski problem for $\alpha$-concave functions] \label{problem-E-FMP} Let $ -\frac{1}{n}<\alpha< 0$  and $\mu$  be a Borel measure  defined on $\Rn$. Find the necessary and/or sufficient conditions on $\mu$, such that there exists an $\alpha$-concave function $f$ satisfying 
\begin{align}
    \mu =   \mu_{\alpha}(f,\cdot)\ \ \ \text{and}\ \ \ 0=\nu_{\alpha}(f,\cdot).\label{mu=mu-alpha}
\end{align}   
\end{problem}  According to \eqref{density-eucldiean measure}, if $\,d\mu(y)=h(y)\,dy$ and $\varphi\geq 0$ is smooth enough, then equation \eqref{mu=mu-alpha}  reduces to   the following Monge-Amp\'{e}re equation: 
\begin{align*}
h(\nabla\varphi(y)){\rm det}(\nabla^2\varphi(y))= \big(1-\alpha \varphi(y)\big)^{\frac{1-\alpha}{\alpha}},  
\end{align*}  with $\varphi$ being the unknown nonnegative convex function on $\Rn.$

\section{\texorpdfstring{The Minkowski problem for $\alpha$-concave functions and $\alpha$-concave measures}{The Minkowski problem for alpha-concave functions and alpha-measures}} \label{section-solution}
We will show the  necessary conditions to Problem \ref{problem-E-FMP} in Subsection \ref{necessary-1}. In Subsection \ref{section-extended}, we will solve a Minkowski problem extending Problem \ref{problem-E-FMP} to $\alpha$-concave measures instead of $\alpha$-concave functions. The main technique used in Subsection \ref{section-extended} is optimal transport, which has been successfully used in solving the Minkowski type problems for the moment measures of log-concave functions in \cite{San16} by Santambrogi and for the $q$-moment measure in \cite{HK21} by Huynh and Santambrogio.  
\subsection{\texorpdfstring{Properties of the Euclidean surface area measure $\mu_{\alpha}(f, \cdot)$}{}} \label{necessary-1}  A convex function $\phi : \mathbb{R}^n \to \mathbb{R} \cup \{+\infty\}$ is  \textit{essentially continuous} if it is lower semi-continuous and the set of its discontinuous points (on $\partial(\operatorname{dom}\phi)$) has zero   $\mathcal{H}^{n-1}$ measure. According to \cite[Lemma 3]{EK15}, if a convex function $\phi$ is essentially continuous, its restriction to almost every line must be  continuous. 

Let $\varphi$ be the base function of the $\alpha$-concave function $f$, that is, $f=(1-\alpha \varphi)^{\frac{1}{\alpha}}$. Then,  $\varphi$ is essentially continuous if and only if $f\equiv 0$ for $\mathcal{H}^{n-1}$-almost every $x\in \partial K_f$.   This is equivalent to say that the spherical moment  measure $\nu_{\alpha}(f,\cdot)$ defined in \eqref{spherical-moment-measure} is a zero measure. We call $f$ an\textit{ $\alpha$-concave function with essentially continuous base} if its base is essentially continuous.

The following lemma proves that the measure $\mu_{\alpha}(f, \cdot)=(\nabla\varphi)_{\sharp}(f^{1-\alpha}\,dx)$, for an $\alpha$-concave function $f\in \mathcal{C}_{\alpha}^{+}(\Rn)$ with essentially continuous base must have finite first moment and its  barycenter is the origin. Note that, for each $i=1, 2,\cdots, n$, we get  
\begin{align*}
       \int_{\Rn} x_i\,d\mu_{\alpha}(f, x)=\int_{\Rn}\frac{\partial \varphi(x)}{\partial x_{i}}(1-\alpha\varphi(x))^{\frac{1}{\alpha}-1}\,dx.
   \end{align*} If \eqref{barycenter=0} holds, then, for each $i=1, 2, \cdots, n,$  \begin{align*}
       \int_{\Rn} x_i\,d\mu_{\alpha}(f, x)=0.
   \end{align*} Hence for each $y\in \Rn,$  it holds that \begin{align*} 
       \int_{\Rn} \langle x, y\rangle\,d\mu_{\alpha}(f, x)=0.
   \end{align*} That is, the barycenter of $\mu_{\alpha}(f, \cdot)$ is at the origin. On the other hand, if \eqref{first-moment-finite} holds, then 
   \begin{align*}
       \int_{\Rn}|y| \,d\mu_{\alpha}(f, y)=\int_{\Rn}\left |\nabla \varphi(x)\right| f^{1-\alpha}(x)\,dx= \int_{\Rn}\left |\nabla f(x)\right|dx<\infty,
   \end{align*} and hence the first moment of $\mu_{\alpha}(f, \cdot)$ is finite. 

\begin{lemma}\label{Lem7}
    Let $-\frac{1}{n}<\alpha<0$ and $f\in \CA(\Rn)$ be an $\alpha$-concave function with  essentially continuous base $\varphi$. Then, 
    \begin{align}\label{first-moment-finite}
        \int_{\Rn}\left |\nabla f(x)\right|dx <+\infty,
    \end{align} 
    and, for any $i=1,\dots,n$,
   \begin{align}\label{barycenter=0}
       \int_{\Rn}\frac{\partial \varphi(x)}{\partial x_{i}}(1-\alpha\varphi(x))^{\frac{1}{\alpha}-1}\,dx=\int_{\Rn}\frac{\partial f(x)}{\partial x_{i}}\,dx =0.
   \end{align} 
     
\end{lemma}
\begin{proof} The proof follows essentially the approach used in 
 \cite[Lemma 4]{EK15}. For completeness, we will give the proof but with modification concentrated.

    First of all, we verify the following claim:  for   any $i=1,\dots,n$,
    $$
    \int_{\Rn} \left| \frac{\partial f(x)}{\partial x_{i}} \right|dx < +\infty.
    $$ To this end, without  loss of generality, let $i=n$ and write $x=(x',t)$ for any $x\in \Rn$. It is easily checked that the function $t\mapsto\varphi(x',t)$ is convex and  coercive, hence it is  non-increasing in  $(-\infty,c]$ and non-decreasing in $[c, \infty)$ for some $c\in \R.$ Thus,   $t\mapsto f(x',t)$ is non-decreasing in $(-\infty,c]$ and non-increasing in $[c, \infty)$. Moreover,  $f$ is locally Lipschitz since it is the composition of the generator function $\Psi_{\alpha}$, which is convex in $[0,+\infty)$, and the convex function $\varphi$. Therefore, $ \frac{\partial f(x', t)}{\partial x_n }=\frac{\partial f(x', t)}{\partial t}$ exists almost everywhere. Applying  the Lebesgue version of the Fundamental Theorem of Calculus \cite[Theorem 3.35]{Fol99}, we get, for any $x'\in \R^{n-1},$
    \begin{align*}
        \int_{\R}\left |\frac{\partial f(x',t)}{\partial x_{n}}\right|dt &=  \int_{\R}\left |\frac{\partial f(x',t)}{\partial t}\right|dt=\int_{-\infty}^{c}\frac{\partial f(x',t)}{\partial t}\,dt-\int_{c}^{+\infty}\frac{\partial f(x',t)}{\partial t}\,dt\\
        &= 2f(x',c)=2\sup_{t\in\R}f(x',t)=2\left( 1-\alpha\inf_{t\in\R}\varphi(x',t)\right)^{\frac{1}{\alpha}}.
    \end{align*} By the Tonelli theorem, we have \begin{align}
        \label{application-T} \int_{\Rn}\left |\frac{\partial f(x)}{\partial x_{n}}\right|\,dx=\int_{\mathbb{R}^{n-1}}\int_{\R}\left |\frac{\partial f(x',t)}{\partial t}\right|\,dt\,dx'=2\int_{\mathbb{R}^{n-1}}\left( 1-\alpha\inf_{t\in\R}\varphi(x',t)\right)^{\frac{1}{\alpha}}dx'. 
    \end{align} 
Note that $ x'\mapsto \inf_{t\in\R}\varphi(x',t)$ is a function 
defined on $\R^{n-1}$ and satisfies \begin{align*}
 \inf_{t\in\R}\varphi(x',t)\geq a\inf_{t\in\R}|(x', t)|+b =a|x'|+b,  
\end{align*}  where $a>0$ and $b\in \R$ are given by  \eqref{coercive}. Thus, $x'\mapsto \inf_{t\in\R}\varphi(x',t)$ verifies the assumption which is the key ingredient in  the proof of Lemma \ref{Lem2.2} and $-\frac{1}{n-1}<-\frac{1}{n}<\alpha<0$. Therefore, the integral in \eqref{application-T} must be   finite following the proof of Lemma \ref{Lem2.2}. Consequently, \eqref{first-moment-finite} follows immediately due to the fact that $$|\nabla f|\leq \sqrt{n} \max \Big\{  \Big|\frac{\partial{f}}{\partial x_i}\Big|,\ i=1, 2, \cdots, n\Big\}.$$

We can apply  \cite[Lemma 3]{EK15} to get that   $t\mapsto f(x',t)$ is continuous for almost every $x'$. Let $[a_{1},b_{1}]$ be  the support of $t\mapsto f(x',t)$, where $a_{1}$ and $b_{1}$ may be  infinity. Then, 
    $$\int_{\R}\frac{\partial f(x',t)}{\partial x_{n}}\,dt=\int_{[a_{1},b_{1}]}\frac{\partial f(x',t)}{\partial t}\,dt = \lim_{t\to b_{1}}f(x',t) - \lim_{t\to a_{1}}f(x',t) =0.$$
    Using the Fubini’s theorem, we get 
    $$\int_{\Rn}\frac{\partial  f(x)}{\partial x_{n}}\,dx = \int_{\mathbb{R}^{n-1}}\left( \int_{\R}\frac{\partial f(x',t)}{\partial x_{n}}\,dt\right)dx'=0.$$ This completes the proof. 
\end{proof}

 Our next lemma is to prove that $\mu_{\alpha}(f,\cdot)$ is not concentrated in any  proper linear subspace of $\Rn$, and hence not concentrated in any hyperplane (due to the fact that the barycenter of $\mu_{\alpha}(f, \cdot)$ is at the origin $o$). 
\begin{lemma}\label{Lem8}
     Let $-\frac{1}{n}<\alpha<0$ and $f=(1-\alpha \varphi)^{\frac{1}{\alpha}}\in \CA(\Rn)$ be an $\alpha$-concave function with  essentially continuous base $\varphi$ such that  $0<\Ja(f)<+\infty$. Then, $\mu_{\alpha}(f, \cdot)$ is not concentrated in any proper linear subspace of $\Rn$.
\end{lemma} \begin{proof} The proof follows essentially the approach used in  \cite[Lemma 6]{EK15} . For completeness, we will give the proof but with modification concentrated. 

Assume the opposite, that is,  the support of $\mu_{\alpha}(f, \cdot)$ is contained in a proper linear subspace, say  $ e_n^{\perp}=\big \{x\in \Rn: x_n=0 \big\}$, the proper linear subspace with normal vector $e_{n}$. It follows from $\mu_{\alpha}(f, \cdot)=(\nabla\varphi)_{\sharp}(f^{1-\alpha}\,dx)$ that 
   \begin{align*}
   0 &=\int_{ \Rn\setminus e_n^{\perp} }\left |x_{n}\right|d\mu_{\alpha}(f, x)+\int_{   e_n^{\perp}}\left |x_{n}\right|d\mu_{\alpha}(f, x)\\ & =\int_{\Rn}\left |x_{n}\right|d\mu_{\alpha}(f, x) \\ &= \int_{\Rn}\left |\frac{\partial \varphi(x)}{\partial x_{n}}\right|(1-\alpha\varphi(x))^{\frac{1}{\alpha}-1}\,dx\\&=\int_{\Rn}\Big|\frac{\partial f}{\partial x_n}\Big|\,dx.    
   \end{align*} Consequently, for almost all $x=(x', t)\in \Rn,$
   we have 
   $$
    \frac{\partial f(x)}{\partial x_n}=\frac{\partial f(x',t)}{\partial t}=0.
    $$   
    
    By the locally Lipschitz property of $f$ in $\mathrm{int}(K_f)$ and the Lebesgue version of  the Fundamental Theorem of Calculus, for almost $x'\in \R^{n-1}$, if  $(x', t)\in \mathrm{int}(K_f)$, then $f(x', t)$ is a constant on $t$. On the other hand,  by \cite[Lemma 3]{EK15}, for almost all $x'\in \R^{n-1},$ $t\mapsto f(x',t)$  is  continuous on $t$ in $\R$, which implies $f(x' ,t)$ is a constant on $t\in \R.$  Therefore, $\int_{\R}f(x',t)\,dt\in \{0,+\infty\}$ for almost all $x'\in \R^{n-1}$. This is a contradiction to the assumption that $0<\Ja(f)<\infty$, because the  Fubini's theorem gives $$\Ja(f)=\int_{\Rn} f(x)\,dx \in \{0,+\infty\}.$$ This completes the proof. 
\end{proof}

Combining Lemmas \ref{Lem7} and \ref{Lem8}, we can get the following result. 
\begin{theorem}\label{Nec-Conditions}
    Let $-\frac{1}{n}<\alpha<0$ and $f=(1-\alpha\varphi)^\frac{1}{\alpha}\in \CA(\Rn)$ have  an essentially continuous base $\varphi$ such that  $0<\Ja(f)<+\infty.$  Then, $$\mu_{\alpha}(f, \cdot)=(\nabla\varphi)_{\sharp} (f^{1-\alpha}\,dx) $$ has a finite first moment, its barycenter is at the origin and its support is not contained in any hyperplane. 
\end{theorem}

The results in Theorem \ref{Nec-Conditions} provide the necessary conditions to the solution to Problem \ref{problem-E-FMP}, namely, $\mu$ is a finite nonzero Borel measure on $\Rn$ whose first moment is finite, barycenter is at the origin, and support is not contained in any hyperplane. 

\subsection{A  solution to an extended Minkowski problem  via optimal transport} \label{section-extended} In this section, we will show that the conditions in Theorem \ref{Nec-Conditions} are also sufficient to solve a Minkowski problem weaker than Problem \ref{problem-E-FMP}. We will prove the sufficiency using optimal transport as did in  \cite{HK21, San16}.

Recall that  $ \varrho\ll \mathcal{L}^{n}$ if $\varrho$ is absolutely continuous with respect to  $ \mathcal{L}^{n}$. It follows from the Radon-Nikodym theorem that there exists a density function $\rho$ such that $\, \,d\varrho =\rho \,dx$.  In order to state our weak Minkowski problem, we need to consider  the measure  \begin{align}\label{a-con-measure}\varrho = \rho\,dx + \varrho^{s},
\end{align}
where, for some convex function $\varphi\geq \frac{1}{\alpha}$,
$$\rho(x)=  \left( 1-\alpha\varphi(x)\right)^{\frac{1}{\alpha}};$$ and the singular part $\varrho^{s}$ is mutually singular with respect to the Lebesgue measure, denoted by $\varrho^{s}\perp \mathcal{L}^{n}$, and is concentrated on the set $\left\{\varphi =  \frac{1}{\alpha}\right\}$. 
When a finite measure  $\varrho$ is not absolutely continuous with respect to $\mathcal{L}^{n}$,   then both  $\rho\,dx$ and   $\varrho^s\neq 0$ are   finite. If $\varrho$ satisfies \eqref{a-con-measure}, then $\varrho$ is called  an \textit{$\alpha$-concave measure}. For the total mass $|\varrho|=\int_{\Rn}\,d\varrho$  to be finite, the set $\left\{\varphi =  \frac{1}{\alpha}\right\}$ must be a $\mathcal{L}^{n}$-null set. In the particular case when $\left\{\varphi =  \frac{1}{\alpha}\right\}$ is singleton, $\rho^{s}$ must be a multiple of a Dirac delta measure.

We extend the surface area measures from $\alpha$-concave functions to $\alpha$-concave measures. 

\begin{definition}\label{def-measure-measure}
    Let $\varrho$ be a finite  $\alpha$-concave measure as given in \eqref{a-con-measure}. One says that $\mu$ is a    \textit{Euclidean surface area measure}    of $\varrho$,  if there is a measure $\pi$ on $\mathbb{R}^{n}\times \mathbb{R}^{n}$, whose marginals are $\rho^{1-\alpha}\,dx+\varrho^s$ and $\mu$:$$\pi(A\times \mathbb{R}^{n}) = \big(\rho^{1-\alpha}\,dx+\varrho^s\big)(A)\ \ \mathrm{and} \ \ \pi(\mathbb{R}^{n}\times B) = \mu(B)$$  for any $\big(\rho^{1-\alpha}\,dx+\varrho^s\big)$-measurable set $A$ and  $\mu$-measurable set $B$; and $$\mathrm{Supp}(\pi)\subset \mathrm{Graph}(\partial \varphi):=\{(x,y)\in \mathbb{R}^{n}\times \mathbb{R}^{n}:y\in \partial \varphi(x)\},$$i.e.,  $ 
y\in \partial \varphi(x) 
$  for $\pi$-almost every $(x,y)$. The spherical surface area of $\varrho$ is defined by $$(\nu_{K_{\rho}})_{\sharp}\big(\rho\,d\mathcal{H}^{n-1}|_{\partial K_{\rho}}\big).$$\end{definition}

The Minkowski problem for $\alpha$-concave functions can be extended as follows. 

\begin{problem}[The extended Minkowski problem for $\alpha$-concave measures] \label{problem-FMP-measure}
 Let $-\frac{1}{n}<\alpha<0$ and let $\mu$  be a Borel measure  on $\Rn$ and $\nu$ be a Borel measure on $S^{n-1}$. Does  there exist an $\alpha$-concave measure $\varrho$, given by \eqref{a-con-measure}, such that $\mu$ and $\nu$ are the Euclidean and the spherical surface area measures of $\varrho$, respectively?      
\end{problem} 

The extended version of Problem \ref{problem-E-FMP} can be stated as follows.

\begin{problem} \label{problem-E-FMP-measure}
    Let $-\frac{1}{n}<\alpha<0$ and $\mu$ be a Borel measure on $\mathbb{R}^{n}$. Does there exists an $\alpha$-concave measure $\varrho$ such that $\mu$ is the Euclidean surface area measure of $\varrho$ and the spherical surface area measure of $\varrho$ is the zero measure?
\end{problem}

In this paper, we only focus on Problem \ref{problem-E-FMP-measure}. In general, however, $\mu$ may not be uniquely determined. Notably, for almost every $x\in  \left\{\varphi > \frac{1}{\alpha}\right\}$, we have $y = \nabla \varphi(x)$. When $\inf \varphi=\frac{1}{\alpha}$  and $\varphi$ is differentiable at 
  $x\in \left\{\varphi=  \frac{1}{\alpha}\right\}$, we have $\partial \varphi(x) = \{o\}$,  and hence  $
\mu = (\nabla\varphi)_{\sharp}(\rho^{1-\alpha}\,dx+\varrho^s).
$  In particular, if $\varphi>\frac{1}{\alpha}$ on $\Rn,$ then $
\mu = (\nabla\varphi)_{\sharp} (\rho^{1-\alpha}\,dx)
$ is exactly the Euclidean surface area measure of the $\alpha$-concave function $\rho=(1-\alpha \varphi)^{\frac{1}{\alpha}}$ defined in \eqref{euclidean-moment-measure} and the extended Minkowski problem (i.e., Problem \ref{problem-E-FMP-measure}) reduces to the Euclidean functional Minkowski problem for $\alpha$-concave functions (i.e.,  Problem \ref{problem-E-FMP}).


Let $\varrho$ be a probability measure on $\mathbb{R}^{n}$.  In order to solve the extended Minkowski problem for $\alpha$-concave measures, we need to define   \begin{align*}\label{R-D-theorem} \mathcal{F}_{\alpha}(\varrho )=-\int_{\Rn}\rho^{\frac{1}{1-\alpha}}(x)\,dx \end{align*} for any finite  $\alpha$-concave measure  $\varrho=\rho\,dx+\varrho^{s}$. Recall that $\mathcal{T}(\varrho,\mu)$ is defined in \eqref{def-T-2} by:  
\begin{align*}  
\mathcal{T}(\varrho,\mu):&=\sup\left\{\int_{\mathbb{R}^{n}\times \mathbb{R}^{n}}\langle x,  y\rangle\,d\pi(x, y):\pi\in\Pi(\varrho,\mu)\right\}
\\&=\inf\left\{\int_{\Rn}\varphi \,d\varrho+\int_{\Rn}\varphi^{*} \,d\mu:\varphi\text{ is convex and lower semi-continuous}\right\},
\end{align*} where 
   $\varphi^*$ is the Legendre transform of $\varphi.$ If the measure $\mu$ has its barycenter at $o$, we can see that $\mathcal{T}(\varrho, \mu)$ is invariant regarding the translation of the measure $\varrho$, due to Proposition \ref{[San16,Prop3.1]}.   
Note that   $\mathcal{P}_1(\Rn)$ denotes the set of probability measures on $\Rn$ with finite first moment. By $\mathrm{Supp}(\mu)$ we mean the support of $\mu.$
  
\begin{proposition}\label{PropertyOfMinimizer}
Let $-\frac{1}{n}<\alpha<0$ and $\mu\in\mathcal{P}_1(\Rn)$. Suppose that the barycenter of $\mu$ is $o$, and $\mu$ is not supported in any hyperplane. Then, any minimizer $\varrho_0$ (assuming the existence) of following problem
\begin{equation} \label{P}
\inf\Big\{(1-\alpha)\mathcal{F}_{\alpha}(\varrho)-\alpha\mathcal{T}(\varrho,\mu):\varrho\in \mathcal{P}_{1}(\Rn)\Big\}
\end{equation}
satisfies that for some lower semi-continuous convex function $\varphi_0$,
\begin{align*}
    \varrho_0 = (1-\alpha\varphi_0)^{\frac{1}{\alpha}-1}\,dx + \varrho_{0}^{s},
\end{align*}
where $\varrho_{0}^{s}$ and $\mathcal{L}^{n}$ are mutually singular, and if $\varrho_{0}^{s}\neq 0$, then $\mathrm{Supp}(\varrho_{0}^{s})\subset\{\varphi_{0}= \frac{1}{\alpha}\}$. Furthermore, the function $\varphi_{0}$ satisfies
\begin{equation}\mathcal{T}(\varrho_{0},\mu)=\int_{\mathbb{R}^{n}}\varphi_{0}\, d\varrho_{0}+\int_{\mathbb{R}^{n}}\varphi_{0}^*\, \,d\mu. \label{T(Rho0,Mu)}\end{equation}
\end{proposition}

\begin{proof} 
Suppose that $\varrho_{0}=\rho_{0}\,dx+\varrho^{s}_{0}$ is a minimizer of \eqref{P}.  Let $\varrho_{\tau}:=\tau^{-n}\omega_{n}^{-1}\mathbf{1}_{\{|x|\leq \tau\}}\,dx\in\mathcal{P}_{1}(\Rn)$. Then
\begin{align}&(1-\alpha)\mathcal{F}_{\alpha}(\varrho_{\tau})-\alpha\mathcal{T}(\varrho_{\tau},\mu)\nonumber \\&=-(1-\alpha)\!\int_{\mathbb{R}^{n}}\!\left( \tau^{-n}\omega_{n}^{-1}\mathbf{1}_{\{|x|\leq \tau\}}\right)^{\frac{1}{1-\alpha}}\,dx-\alpha\sup\left\{\int_{\mathbb{R}^{n}\times \mathbb{R}^{n}}\langle x,  y\rangle\,d\pi(x, y):\pi\in\Pi(\varrho_{\tau},\mu)\right\}\nonumber\\&\leq -(1-\alpha)\omega_{n}^{\frac{-\alpha}{1-\alpha}}\tau^{\frac{-n\alpha}{1-\alpha}}-\alpha\tau\int_{\mathbb{R}^{n}}|y|\,d\mu(y).\label{NonzeroAbsContPart}\end{align}
Since $-\frac{1}{n}<\alpha<0$, we have $\frac{-\alpha n}{1-\alpha}<1$. Hence, if we choose $\tau$ such that 
$$0<\tau ^{1+\frac{\alpha n}{1-\alpha}}< \frac{(1-\alpha)\omega_{n}^{\frac{-\alpha}{1-\alpha}}}{-\alpha  \int_{\mathbb{R}^{n}}|y|\,d\mu(y)},$$
then, from \eqref{NonzeroAbsContPart} and the minimizing property of $\varrho_{0}$, we obtain
\begin{align}(1-\alpha)\mathcal{F}_{\alpha}(\varrho_{0})-\alpha\mathcal{T}(\varrho_{0},\mu)\leq (1-\alpha)\mathcal{F}_{\alpha}(\varrho_{\tau})-\alpha\mathcal{T}(\varrho_{\tau},\mu) < 0.\label{NegativityOfMinimum}\end{align}

Let $\varphi _1 $ be a convex lower semi-continuous function that is an optimal function of $\mathcal{T}(\varrho_{0},\mu)$. Namely, $\varphi_{1}$ (and hence $\varphi_{1}+c_{0}$ for any constant $c_{0}$) must satisfy
    \begin{equation}\mathcal{T}(\varrho_{0},\mu) = \int_{\Rn}\varphi_1  \,d\varrho_{0}+\int_{\Rn}\varphi_1 ^{*}\,d\mu.\label{T(Rho0,Mu)2}\end{equation}
Without loss of generality, we assume   $\inf\varphi_1 =0$. Because $\mu$ is fixed, $\varrho_{0}=\rho_0\,dx+\varrho_0^s$ is a minimizer of the following functional defined on $\mathcal{P}_{1}(\Rn)$:
$$G:\varrho\mapsto  (1-\alpha)\mathcal{F}_{\alpha}(\varrho) -\alpha\int_{\Rn}\varphi_1  d{\varrho}.$$

We now prove   $\varrho_{0} \neq\varrho_{0}^{s}$ by contradiction.  Assume that $\varrho_{0}=\varrho_{0}^{s}$. Then,  
\begin{align}\label{G-rho>0}G(\varrho_0) = -\alpha\int_{\mathbb{R}^{n}}\varphi_1 \,d\varrho_0^{s} \geq 0 .\end{align}
On the other hand, let $x_{n}$ be a sequence such that $\varphi_1 (x_{n})\to 0$ as $n\to\infty$. Then by \eqref{NegativityOfMinimum}, one has
\begin{align*}G(\varrho_{0}) &= (1-\alpha) \mathcal{F}_\alpha(\varrho_{0})-\alpha \mathcal{T}(\varrho_{0}, \mu) +\alpha\int_{\mathbb{R}^{n}}\varphi_1^{*} \,d\mu\\
&<\alpha\int_{\mathbb{R}^{n}} \varphi_1^{*}\,d\mu\leq \alpha \int (\chi_{\{x_{n}\}} +\varphi_1(x_{n}))^{*}\,d\mu,\end{align*} where 
$\chi_E(x)=0$ if $x\in E$ and $\chi_E(x)=\infty$ if $x\notin E$. It can be checked by \eqref{legendre-tran} that $$(\chi_{\{x_{n}\}} +\varphi_1 (x_{n}))^{*}(x)=\langle x, x_n\rangle-\varphi_1 (x_n).$$ This further gives 
\begin{align*}  G(\varrho_{0})  <\alpha\int_{\mathbb{R}^{n}} \varphi_1^{*}\,d\mu  \leq \lim_{n\rightarrow \infty}   \alpha\int \left\langle x,x_{n}\right\rangle \,d\mu -\lim_{n\rightarrow \infty}  \alpha\varphi_1 (x_{n})
=0,
\end{align*} where we have used the fact that $\mu$ has its barycenter at $o$ and $\varphi_1(x_n)\rightarrow 0. $
This is a contradiction to \eqref{G-rho>0}, and hence   $\varrho_0\neq \varrho_0^s$. That is, $\rho_0\,dx$ is not a zero measure and $\rho_0$ is nonzero on a set of positive $\mathcal{L}^n $ measure. Hence, $\mathrm{Supp}(\rho_0)$ is nonempty and has positive $\mathcal{L}^n $ measure.

Next we prove that  $\mathcal{L}^{n}(\{ \varphi_1  < \infty\}\setminus\{\rho_{0}>0\})=0$. Let $A:= \{\varphi_1 < \infty\}\setminus\{\rho_{0}>0\}$ and suppose, for the sake of contradiction, that $\mathcal{L}^{n}(A)>0$. Then there exists a point $x_{0}\in A$ such that for any ball $B$ containing $x_{0}$, we have $\mathcal{L}^{n}(B\cap A)>0$. Indeed, suppose the contrary, that for every $x\in A$, there exists a ball $B_{x}$ containing $x$ such that $\mathcal{L}^{n}(B_{x}\cap A) = 0$. Since the countable set $\mathcal{B}=\{B(q,r):q\in \mathbb{Q}^{n},r\in \mathbb{Q},r>0\}$, where $B(q,r)=q+rB^n_2$, forms a basis for the Euclidean topology on $\mathbb{R}^{n}$, there exists $B(q_{x},r_{x})\in\mathcal{B}$ such that $x\in B(q_{x},r_{x})\subset B_{x}$ for each $x\in A$. Because the collection $\{B(q_{x},r_{x}):x\in A\} \subset \mathcal{B}$ is countable, there exists a countable subcollection $\{B_{x_{i}}\}_{i=1}^{\infty}$ of $\{B_{x}:x\in A\}$ that covers $A$. Then
$$\mathcal{L}^{n}(A)\leq \mathcal{L}^{n}\Big(\bigcup_{i=1}^{\infty}(B_{x_{i}}\cap A)\Big)\leq \sum_{i=1}^{\infty}\mathcal{L}^{n}(B_{x_{i}}\cap A)=0,$$
which is a contradiction to the assumption on $A$. 

With the help of the above argument, we can now  find a contradiction to the assumption that $\mathcal{L}^{n}(A)>0$ by transporting a little mass of $\rho_{0}\,dx$ to a neighborhood of $x_{0}$. To do so, let $S_{1}\subset \{\rho_{0}>0\}$  be a set of positive $(\rho_{0}\,dx)$-measure, and hence positive $\mathcal{L}^{n}$-measure (because $\rho_{0}\,dx\ll \mathcal{L}^{n}$), such that,  $0<\inf_{S_{1}}\rho_{0} = c$. Let $\varepsilon<c$. By shrinking $S_{1}$ if necessary, we can choose $S_{2}\subset A$ containing $x_{0}$ such that $\mathcal{L}^{n}(S_{2}) = \mathcal{L}^{n}(S_{1})$. Consider $\varrho_{1}\in \mathcal{P}_{1}(\Rn)$ such that $\varrho_{1}^{s} = \varrho_{0}^{s}$ and
$$
\rho_{1}  = \left\{ \begin{array}{cl} \rho _{0}-\varepsilon, &\ \ \text{on }S_{1},\\\varepsilon,&\ \ \text{on } S_{2},\\ \rho_{0},  &\ \ \text{otherwise}.\end{array}\right.
$$
We compare $G(\varrho_{1})$ with $G(\varrho_{0})$ for small $\varepsilon$. Shrinking $S_{1}$ again, we can assume that $S_{2}$ is a subset of $\{x\in A:|\varphi_1(x)-\varphi_1(x_{0})|<\delta\}$ for some small number $\delta$. Then
\begin{equation}\int_{\Rn} \varphi_1 \,d\varrho_{1} < \int_{\Rn} \varphi _1\,d\varrho_{0}+\varepsilon\int_{S_{2}}\varphi_1\,dx \leq \int_{\Rn} \varphi _1 \,d\varrho_{0}+\varepsilon\mathcal{L}^{n}(S_{2})(\varphi_1(x_{0})+\delta).\label{BoundOfIntOfPhiRho1}\end{equation}
Moreover, from the concavity of $t^{k}$ with $0<k<1$, we obtain 
\begin{align}\mathcal{F}_{\alpha}(\varrho_{1})-\mathcal{F}_{\alpha}(\varrho_{0})&= \int_{\Rn}\rho_{0}^{\frac{1}{1-\alpha}}(x)\,dx-\int _{\Rn} \rho_{1}^{\frac{1}{1-\alpha}}(x)\,dx \nonumber
\\&= \int_{S_{1}}\rho_{0}^{\frac{1}{1-\alpha}}(x)\,dx-\int_{S_{1}}\big( \rho_{0}(x)-\varepsilon\big)^{\frac{1}{1-\alpha}}\,dx-\int_{S_{2}}  \varepsilon ^{\frac{1}{1-\alpha}}\,dx \nonumber
\\&\leq\frac{1}{1-\alpha}\int_{S_{1}}(\rho_{0}(x)-\varepsilon)^{\frac{\alpha}{1-\alpha}}\varepsilon \,dx-\mathcal{L}^{n}(S_{2})\varepsilon^{\frac{1}{1-\alpha}}\nonumber
\\&\leq \frac{1}{1-\alpha}(c-\varepsilon)^{\frac{\alpha}{1-\alpha}}\mathcal{L}^{n}(S_{1})\varepsilon-\mathcal{L}^{n}(S_{1})\varepsilon^{\frac{1}{1-\alpha}}\nonumber
\\&\leq \bigg( \frac{(c-\varepsilon_0)^{\frac{\alpha}{1-\alpha}}}{1-\alpha}\varepsilon_0^{\frac{-\alpha}{1-\alpha}}-1\bigg)\mathcal{L}^{n}(S_{1})\varepsilon^{\frac{1}{1-\alpha}}, \label{BoundOfF(Rho1)}
\end{align} for any $0<\varepsilon<\varepsilon_0$, where   $0<\varepsilon_{0}<c$ is small enough so that  
 \begin{align*}  \frac{(c-\varepsilon_0)^{\frac{\alpha}{1-\alpha}}}{1-\alpha}\varepsilon_0^{\frac{-\alpha}{1-\alpha}}-1 <0. 
\end{align*} 
Together with \eqref{BoundOfIntOfPhiRho1} and \eqref{BoundOfF(Rho1)}, we have
\begin{align*}
    G(\varrho_{1})-G(\varrho_{0}) &< (1-\alpha) \bigg( \frac{(c-\varepsilon _0)^{\frac{\alpha}{1-\alpha}}}{1-\alpha}\varepsilon_0 ^{\frac{-\alpha}{1-\alpha}}-1\bigg)\mathcal{L}^{n}(S_{1})\varepsilon ^{\frac{1}{1-\alpha}} -\alpha\mathcal{L}^{n}(S_{1})(\varphi_{1}(x_{0})+\delta)\varepsilon , 
\end{align*}  which is negative if $0<\varepsilon<\varepsilon_0$ satisfies
$$
\varepsilon^{\frac{-\alpha}{1-\alpha}}<\frac{(1-\alpha)\big(1-\frac{(c-\varepsilon_0)^{\frac{\alpha}{1-\alpha}}}{1-\alpha}\varepsilon_0 ^{\frac{-\alpha}{1-\alpha}}\big)  }{-\alpha \left(\varphi_{1}(x_0\right)+\delta)}.
$$ Hence, we obtain a contradiction to the fact that $\varrho_{0}$ minimizes $G$. Therefore, for $\mathcal{L}^{n}$-almost every $x$, if $\varphi_1(x) < \infty$, then $\rho _{0}(x)>0$.

We now claim that $\varrho_{0}$ is also a minimizer of a linearized functional of $G$ on a subset of $\mathcal{P}_{1}(\mathbb{R}^{n})$. To this end, using the fact that $\varrho_{0}$ is a minimizer of $G$, we have that for any $0<t<1$, there exists $0<t_{\rho}<t$ such that
\begin{align}\label{GateauxDerivative}  0&\leq \frac{G(\varrho_{0}+t(\varrho-\varrho_{0}))-G(\varrho_{0})}{t}\nonumber \\
    &=(\alpha-1)  \int_{\Rn}\frac{\left(\rho_{0}+t(\rho-\rho_{0})\right)^{\frac{1}{1-\alpha}}- \rho_{0} ^{\frac{1}{1-\alpha}}}{t}\,dx -\alpha\int_{\Rn}\varphi_{1} d(\varrho-\varrho_{0})\nonumber  \\
    &=-\int_{\Rn}(\rho_{0}+t_{\rho} (\rho-\rho_{0}))^{\frac{\alpha}{1-\alpha}}(\rho-\rho_{0})\,dx-\alpha\int_{\Rn}\varphi_{1} d(\varrho-\varrho_{0}),
    \end{align}    
where the last equality follows from the mean value theorem. Notice that 
    $$
(\rho_{0}+t_{\rho}(\rho-\rho_{0}))^{\frac{\alpha}{1-\alpha}}|\rho-\rho_{0}|\leq \big(\min\{\rho,\rho_{0}\}\big)^{\frac{\alpha}{1-\alpha}}|\rho-\rho_{0}|.
    $$
Assume that $\varrho\in \mathcal{S}$, where $\mathcal{S}$ is given by 
$$\mathcal{S}:= \left\{\varrho\in\mathcal{P}_{1}(\mathbb{R}^{n}):\big(\min\{\rho,\rho_{0}\}\big)^{\frac{\alpha}{1-\alpha}}\left |\rho-\rho_{0}\right| \in L^{1}(\mathbb{R}^{n})\right\}.$$
  The set $\mathcal{S}$ is nonempty as $\varrho_{0}\in\mathcal{S}$. Let $\phi:=\rho-\rho_{0}$. By the dominated convergence theorem,
\begin{equation}\label{DCT}
    \lim_{t\to 0^{+}}\int_{\Rn}(\rho_{0}+t\phi)^{\frac{\alpha}{1-\alpha}}\phi\,dx = \int_{\Rn}\lim_{t\to 0^{+}}(\rho_{0}+t\phi)^{\frac{\alpha}{1-\alpha}}\phi\,dx = \int_{\Rn} \rho_{0} ^{\frac{\alpha}{1-\alpha}}\phi\,dx.
    \end{equation}
    Combining (\ref{GateauxDerivative}) and (\ref{DCT}), we find that $\varrho_{0}$ is a minimizer on $\mathcal{S}$ of the functional
  \begin{align*}
     L:\varrho & \mapsto -\int_{\Rn}\rho\rho_{0}^{\frac{\alpha}{1-\alpha}}\,dx-\alpha\int_{\Rn}\varphi_{1} \,d\varrho\\ 
    &=\int_{\mathbb{R}^{n}}\rho\left( -\rho_{0}^{\frac{\alpha}{1-\alpha}}-\alpha\varphi_{1}\right)dx-\alpha\int_{\mathbb{R}^{n}}\varphi_{1} \,d\varrho^{s}.
  \end{align*}

Next we prove that $\rho_{0}\,dx$ must be concentrated on the set $\{-\rho_{0}^{\frac{\alpha}{1-\alpha}}-\alpha\varphi_1 =c_{0}\}$, where $c_{0} \leq 0$ is a constant. Recall that $\mathrm{ess\,sup}f = \inf\{c:f(x)\leq c\text{ for $\mathcal{L}^{n}$-a.e. }x\}$ and $\mathrm{ess\,inf}f = -\mathrm{ess\,sup}(-f)$. Suppose, for the sake of contradiction, that 
    $$
   c_{1}:=\mathrm{ess\,inf}_{\{\rho_{0}>0\}}\Big\{-\rho_{0}^{\frac{\alpha}{1-\alpha}}-\alpha\varphi_{1}\Big\}<\mathrm{ess\,sup}_{\{\rho_{0}>0\}}\Big\{-\rho_{0}^{\frac{\alpha}{1-\alpha}}-\alpha\varphi_{1}\Big\}:= c_{2},$$
where $c_1$ and $c_2$ could be infinite. We can choose a constant $c_3\in (c_1, c_2)$ and two sets 
    $$
    S_{3}\subset \left\{x:c_{1}\leq -\rho_{0}^{\frac{\alpha}{1-\alpha}}(x)-\alpha\varphi_{1}(x)\leq c_3 \right\} \ \ \mathrm{and} \ \
    S_{4}\subset \left\{x:c_3< -\rho_{0}^{\frac{\alpha}{1-\alpha}}(x)-\alpha\varphi_{1}(x)\leq c_{2}\right\}
    $$
    such that $\inf_{S_{4}}\rho_{0} = c'>0$, $\inf_{S_{3}} \rho_{0}>0$, and 
    $$
    0< \mathcal{L}^{n}(S_{3})=\mathcal{L}^{n}(S_{4})<+\infty.
    $$

   Consider $\varrho = \rho\,dx + \varrho^{s}\in\mathcal{P}_{1}(\mathbb{R}^{n})$ such that $\varrho^{s} = \varrho_{0}^{s}$ and
    $$
    \rho=\begin{cases}\rho_{0}+\frac{c'}{2}&\text{on }S_{3},\\\rho_{0}-\frac{c'}{2} &\text{on }S_{4},\\\rho_{0} &\text{otherwise}.\end{cases}
    $$
 It follows that 
    \begin{align*}\int_{\mathbb{R}^{n}} \min\{\rho,\rho_{0}\}^{\frac{\alpha}{1-\alpha}}\left |\rho-\rho_{0}\right|dx&=\frac{c'}{2}\int_{S_{3}}\rho_{0}^{\frac{\alpha}{1-\alpha}}dx+\frac{c'}{2}\int_{S_{4}}\left( \rho_{0}-\frac{c'}{2}\right)^{\frac{\alpha}{1-\alpha}}dx
    \\&\leq \frac{c'}{2}\Big( \inf_{S_{3}}\rho_{0}\Big)^{\frac{\alpha}{1-\alpha}}\mathcal{L}^{n}(S_{3})+\frac{c'}{2}\left( \frac{c'}{2}\right)^{\frac{\alpha}{1-\alpha}}\mathcal{L}^{n}(S_{4})
    \\&<+\infty.\end{align*}
That is $\varrho\in \mathcal{S}$. On the other hand, 
    \begin{align*}L(\varrho)&= L(\varrho_{0})+\frac{c'}{2}\int_{S_{3}}(-\rho_{0}^{\frac{\alpha}{1-\alpha}}-\alpha\varphi_{1})\,dx-\frac{c'}{2}\int_{S_{4}}(-\rho_{0}^{\frac{\alpha}{1-\alpha}}-\alpha\varphi_{1})\,dx<L(\varrho_{0}),\end{align*}
     which is a contradiction to the fact that $\varrho_{0}$ is a minimizer of $L$ on $\mathcal{S}$. Hence, $\rho_{0}$ must be concentrated on the set $$\{-\rho_{0}^{\frac{\alpha}{1-\alpha}}-\alpha\varphi_1 =c_{0}\}$$ for some constant $c_{0}$. This means for $\mathcal{L}^{n}$-almost every $x\in \{\rho_{0}>0\}$, we have $-\rho_{0}(x)^{\frac{\alpha}{1-\alpha}}(x)-\alpha\varphi_{1}(x) = c_{0}$. As claimed above, $\mathcal{L}^{n}(\{\varphi_{1}<\infty\}\setminus \{\rho_{0}>0\}) = 0$, so by redefining $\rho_{0}$ on a Lebesgue null set, we may assume that, if $\varphi_{1}(x)<\infty$, then $\rho_{0}(x)>0$ and
     $$-\rho_{0}^{\frac{\alpha}{1-\alpha}}(x)-\alpha\varphi_{1}(x) = c_{0}.$$
     
     We now claim that  $c_{0}\leq 0$. To this end, take  the sequence $x_{n}$ such that $\varphi_{1}(x_{n})\to 0=\inf \varphi_{1}$, and then 
    $$
    c_{0}=\lim_{n\to\infty}(-\rho_{0}(x_{n})^{\frac{\alpha}{1-\alpha}}-\alpha\varphi_{1}(x_{n}) ) = - \lim _{n\to\infty}\rho_0\left(x_n\right)^{\frac{\alpha}{1-\alpha}}\leq 0.
    $$     
 From the fact that $\rho_{0}\,dx$ is concentrated on the set $\{-\rho_{0}^{\frac{\alpha}{1-\alpha}}-\alpha\varphi_{1} =c_{0}\}$, on the set $\{\varphi_{1}<\infty\}$, we necessarily have
  $$
    \rho_{0} = \Bigg( 1-\alpha\bigg( \varphi_{1}+\frac{c_{0}+1}{\alpha}\bigg)\Bigg)^{\frac{1-\alpha}{\alpha}}:=  \big( 1-\alpha \varphi_0 \big)^{\frac{1-\alpha}{\alpha}}.
    $$
As $c_0\leq 0$ and $
-\frac{1}{n}<\alpha<0$, we have $$\varphi_0 =\varphi_{1}+\frac{c_{0}+1}{\alpha}\geq \frac{1}{\alpha}.$$ In the set $\{\varphi_{1}=\infty\}$, we have $\rho_{0}=0$ almost everywhere with respect to $\rho_{0}\,dx$,  as otherwise $G(\varrho_{0}) = \infty$. Clearly, $\varphi_{0}$ satisfies \eqref{T(Rho0,Mu)}, due to \eqref{T(Rho0,Mu)2}.

Finally, we prove that $\operatorname{Supp}\left(\varrho_0^s\right) \subset\left\{\varphi_0=\frac{1}{\alpha}\right\}$. Assume $\varrho_{0}^{s}\neq 0$. If $c_{0}=0$ and $\mathrm{Supp}(\varrho_{0}^s)$ is not a subset of $\{\varphi_{1}=0\}=\{\varphi_{0}=\frac{1}{\alpha}\}$, then
$$\int \varphi_{1} \,d\varrho_0^{s} > 0 = \lim_{n\to\infty}\int \varphi_{1} \,d\delta_{x_{n}},$$ where $x_n$ is such that $\varphi_1(x_n)\to0$, a contradiction to the minimality of  $\varrho_{0}$ for $G$. If $c_{0}<0$, then $$L\bigg( \frac{\rho_{0}dx}{\int \rho_{0}dx}\bigg)< L(\varrho_{0}),$$ which is a contradiction to the minimality of $\varrho_{0}$ for $L$ on $\mathcal{S}$. The conclusion follows.
\end{proof}

The next result in this section concerns the existence of minimizers to problem (\ref{P}). To establish such a result, we need the following lemma.
\begin{lemma}\label{LowerBoundOfFa}
    Let $-\frac{1}{n}<\alpha<0$. Then, there exist constants $C$, $\alpha_{1}\in (0,-\alpha n)$, and $\alpha_{2}\in (-\alpha n,1)$, such that, for any $\varrho\in \mathcal{P}_{1}(\Rn)$, one has
    \begin{equation*}
        \mathcal{F}_{\alpha}(\varrho)\geq C-\left(\int_{\mathbb{R}^n}|x| \,d \varrho(x)\right)^{\alpha_1} -\left(\int_{\mathbb{R}^n}|x| \,d \varrho(x)\right)^{\alpha_2}.
    \end{equation*}
\end{lemma}

\begin{proof}
    Let $\varrho=\rho\,dx+\varrho^s\in \mathcal{P}_{1}(\Rn)$, and $g(\tau)=-\tau^{\frac{1}{1-\alpha}}$ for $\tau\in(0,+\infty)$.  The Legendre transform of $g$ is a function $g^*: (-\infty,0)\to \R$ given by 
    $$g^{*}(t) =\sup_{\tau\in (0, \infty) }
    \big(t \cdot \tau-g(\tau)\big)=-\alpha(1-\alpha)^{\frac{1}{\alpha}-1}(-t)^{\frac{1}{\alpha}},$$ for $t\in (-\infty, 0).$ Clearly, we have Young's inequality $$g^{*}(t)+g(\tau)\geq t\tau,\ \ \ \mathrm{for} \ \tau>0,t<0.$$
    Let $\beta:=-\alpha(1-\alpha)^{\frac{1}{\alpha}-1}>0$, $\alpha_{1}\in (0,-\alpha n)$, and $\alpha_{2}\in (-\alpha n,1)$. Then, \begin{align}\label{alpha1alpha2}\int_{\{|x|\leq 1\}} |x|^{\frac{\alpha_{1}}{\alpha}}\,dx<+\infty \ \ \mathrm{and}\ \ \int_{\{|x|\geq 1\}} |x|^{\frac{\alpha_{2}}{\alpha}}\,dx<+\infty.\end{align} We can apply the above Young's inequality to $\tau=\rho(x)$, $t=-|x|^{\alpha_{i}}$, $g(\tau)$ and $g^*(t)$, to obtain the following inequality:  
    \begin{equation}-\rho(x)\left |x\right|^{\alpha_{i}}\leq -\rho(x)^{\frac{1}{1-\alpha}}+\beta \left |x\right|^{\frac{\alpha_{i}}{\alpha}}.\label{Young'sIneq}\end{equation}  As $\alpha_1\in (0, -\alpha n)\subset (0,1)$, we can  apply Jensen’s inequality to get 
    \begin{align*}-\int_{\{|x|\leq 1\}}\rho^{\frac{1}{1-\alpha}}(x)\,dx &\geq -\beta \int_{\{|x|\leq 1\}} \left |x\right|^{\frac{\alpha_{1}}{\alpha}}dx-\int_{\{|x|\leq 1\}} \rho(x)\left |x\right|^{\alpha_{1}}dx
    \\
    &\geq C_{1}- \int_{\Rn}  \left |x\right|^{\alpha_{1}}d\varrho(x)\\
    &\geq C_{1}-\left( \int_{\Rn} |x|\,d\varrho(x)\right)^{\alpha_{1}}.\end{align*}
    Similarly, as $\alpha_2\in (0, 1)$, Jensen’s inequality implies 
    \begin{align*}-\int_{\{|x|\geq 1\}}\rho^{\frac{1}{1-\alpha}}(x)\,dx&\geq -\beta \int_{\{|x|\geq 1\}} \left |x\right|^{\frac{\alpha_{2}}{\alpha}}dx-\int_{\{|x|\geq 1\}} \rho(x)\left |x\right|^{\alpha_{2}}dx
    \\
    &\geq C_{2}-\int_{\Rn}  \left |x\right|^{\alpha_2}d\varrho(x)\\
    &\geq C_{2}-\left( \int_{\Rn}  \left |x\right|d\varrho(x)\right)^{\alpha_{2}}.\end{align*}
    It then follows that
    $$\mathcal{F}_{\alpha}(\varrho)\geq C_{1}+C_{2} -\left(\int_{\mathbb{R}^n}|x| \,d \varrho(x)\right)^{\alpha_1} -\left(\int_{\mathbb{R}^n}|x| \,d \varrho(x)\right)^{\alpha_2}.$$ This completes the proof. 
\end{proof}

 We need the following proposition as well. 
\begin{proposition}\label{Existence}
     Let $-\frac{1}{n}<\alpha<0$ and $\mu\in\mathcal{P}_1(\Rn)$. Suppose that the barycenter of $\mu$ is $o$, and $\mu$ is not supported in any hyperplane. Then there exists a solution to the problem (\ref{P}). 
\end{proposition}
\begin{proof} Recall that, if  $\mu$ has its barycenter at the origin $o$, we can see that $\mathcal{T}(\varrho, \mu)$, and  hence $$
(1-\alpha)\mathcal{F}_{\alpha}(\varrho)-\alpha\mathcal{T}(\varrho,\mu)$$ are invariant regarding the translation of the measure $\varrho$.  Without loss of generality, we can take  a  minimizing sequence $\varrho_{k}$ of (\ref{P}), such that, their  barycenters are at $o$. By Lemma \ref{LowerBoundOfFa}, each $\varrho_{k}$  satisfies that, for some fixed $\alpha_{1} \in (0, -\alpha n)$ and $\alpha_{2}\in (-\alpha n,1)$,  
    \begin{equation}\label{LowerBoundOfFAlongRk}
        \mathcal{F}_{\alpha}(\varrho_k)\geq C-\left(\int_{\mathbb{R}^n}|x| \,d \varrho_k(x)\right)^{\alpha_1} -\left(\int_{\mathbb{R}^n}|x| \,d \varrho_k(x)\right)^{\alpha_2}.
    \end{equation}
Due to Proposition \ref{[San16,Prop3.2]}, there exists a constant $\tau\geq 0$ such that $\mathcal{T}(\varrho,\mu)\geq \tau\int_{\Rn}\left |x\right|d\varrho(x)$ for any $\varrho\in \mathcal{P}_{1}(\mathbb{R}^{n})$ whose barycenter is at $o$.  Hence 
    \begin{align*}(1-\alpha )\left(C-\left( \int_{\Rn}\left |x\right|d\varrho_{k}(x)\right)^{\alpha_{1}}-\left( \int_{\Rn}\left |x\right|d\varrho_{k}(x)\right)^{\alpha_{2}} \right)-\alpha\tau\int_{\Rn}\left |x\right|d\varrho_{k}(x) \\ \leq (1-\alpha)\mathcal{F}_{\alpha}(\varrho_{k})-\alpha\mathcal{T}(\varrho_{k},\mu).\end{align*} Note that the sequence $(1-\alpha)\mathcal{F}_{\alpha}(\varrho_{k})-\alpha\mathcal{T}(\varrho_{k},\mu)$ is bounded from above (as $\varrho_k$ is a minimizing sequence to problem (\ref{P})). Hence we see that the sequence $$\int_{\Rn} |x|\,d\varrho_k(x)$$ must be bounded from above as well.  This gives the tightness of the sequence $\varrho_{k}$. Indeed, let $M := \sup_{k}\int_{\Rn} |x|\,d\varrho_{k}(x)<+\infty$. Then, for any  $\lambda >0$, we have
    $$\int_{\{|x|\geq \lambda\}}\,d\varrho_{k}(x)\leq \int_{\{|x|\geq \lambda\}} \frac{|x|}{\lambda}\,d\varrho_{k}(x) \leq  \frac{M}{\lambda}.$$
The tightness of the sequence $\varrho_{k}$ follows by letting $\lambda\to +\infty$. By Prokhorov's theorem (see \cite[Theorem 2.1.11]{FG21}), the sequence $\varrho_{k}$ admits a subsequence,   still denoted by $\varrho_k$ for convenience, that converges weakly to some $\varrho_{0}$, i.e.,
    $$\int_{\Rn}\phi \,d\varrho_{k}\to\int_{\Rn}\phi \,d\varrho_{0},$$  for any bounded continuous function $\phi$. 
    
We prove that $\varrho_{0}$ has finite first moment. By \eqref{LowerBoundOfFAlongRk} and the boundedness of $\int_{\mathbb{R}^n}|x| d \varrho_k(x)$, the sequence $\mathcal{F}_\alpha(\varrho_k)$ is bounded from below. Hence the sequence $\mathcal{T}(\varrho_k, \mu)$ is bounded from above. It follows from  Proposition \ref{[San16,Prop3.1]} that
\begin{equation}\mathcal{T}(\varrho_{0},\mu)\leq \liminf_{k\to\infty}\mathcal{T}(\varrho_{k},\mu)<\infty.\label{LSC-Of-T}\end{equation}
By Theorem \ref{[Vil08, Theorem 5.10]}, there exists a function $\varphi_{1}$ such that
$$\mathcal{T}(\varrho_{0},\mu)=\int_{\mathbb{R}^{n}}\varphi_{1}(x)\,d\varrho_{0}(x)+\int_{\mathbb{R}^{n}}\varphi_{1}^{*}(x)\,d\mu(x)<\infty.$$
Consequently, $\int_{\mathbb{R}^{n}}\varphi_{1}\,d\varrho_{0}<\infty$  and $\varphi_{1}$ is coercive, as otherwise $\varphi_{1}^{*}=\infty$ in a half space. Therefore, there exist positive numbers $R$ and $c_{1}$ such that $\varphi_{1}(x)\geq c_{1}|x|$ whenever $|x|>R$. Thus
\begin{align*}\int_{\mathbb{R}^{n}}|x|\,d\varrho_{0}(x) &= \int_{\{|x|\leq R\}}|x|\,d\varrho_{0}(x)+\int_{\{|x|>R\}}|x|\,d\varrho_{0}(x)\\&\leq \int_{\{|x|\leq R\}}|x|\,d\varrho_{0}(x)+\frac{1}{c_{1}}\int_{\{|x|> R\}}\varphi_{1}(x)\,d\varrho_{0}(x)<\infty.\end{align*}

By \eqref{LSC-Of-T}, to conclude the existence of a minimizer for problem (\ref{P}), it suffices to prove that
    \begin{equation}\label{LimInfOfF(RhoK)}
        \liminf _{k\to\infty}\mathcal{F}_{\alpha}(\varrho_{k})\geq \mathcal{F}_{\alpha}(\varrho_{0}).
    \end{equation}
Let $g(x):=\mathbf{1}_{\{|x|\leq 1\}} |x|^{ \alpha_1 } +\mathbf{1}_{\{ |x|> 1\}} |x|^{ \alpha_2}  $, and  $\beta:=-\alpha(1-\alpha)^{\frac{1}{\alpha}-1}>0$. Note that
    \begin{align}\label{SplitF}
        \mathcal{F}_\alpha(\varrho)&=\int_{\mathbb{R}^n}\left(-\rho(x)^{\frac{1}{1-\alpha}}+\beta g^{\frac{1}{\alpha}}(x) +g(x) \rho(x)\right) d x +\int_{\mathbb{R}^{n}}g(x)\,d\varrho^{s}(x)\nonumber
        \\ & \quad -\beta \int_{\mathbb{R}^n} g^{\frac{1}{\alpha}}(x)\,d x-\int_{\mathbb{R}^n} g(x)\,d\varrho(x).
    \end{align}  By \eqref{alpha1alpha2}, the term $\int_{\mathbb{R}^{\mathrm{n}}} g^{\frac{1}{\alpha}}(x)\,d x$ is finite and independent of $\varrho$. We now prove that 
\begin{align}
   \lim_{k\to \infty}\int_{\mathbb{R}^{n}}g(x)\,d\varrho_{k}(x) = \int_{\mathbb{R}^{n}}g(x)\,d\varrho_{0}(x). \label{limit-g-rho-k} 
\end{align} 
For any $M_1>1$, 
    \begin{align}\label{Continuity-Of-Int-g}&\left |\int_{\mathbb{R}^{n}}g(x) \,d\varrho_{k}(x) - \int_{\mathbb{R}^{n}}g(x)\,d\varrho_{0}(x)\right| \nonumber 
\\ & \leq \  \left |\int_{\mathbb{R}^{n}}\min\{g,M_1 \}\,d(\varrho_{k}-\varrho_0)\right|+\int_{\{|x|^{\alpha_{2}}>M_1 \}}|x|^{\alpha_{2}}\,d(\varrho_{k}+\varrho_0)\nonumber
\\ & \leq\  \left |\int_{\mathbb{R}^{n}}\min\{g,M_1 \}\,d(\varrho_{k}-\varrho_0)\right|+\int_{\{|x|^{\alpha_{2}}>M_1 \}}\frac{|x|}{M_1 ^{(1-\alpha_{2})/\alpha_{2}}}\,d(\varrho_{k}+\varrho_0)\nonumber
\\& \leq\  \left |\int_{\mathbb{R}^{n}}\min\{g,M_1\}\,d(\varrho_{k}-\varrho_0)\right|+\frac{1}{M_1^{(1-\alpha_{2})/\alpha_{2}}}\left( \int_{\mathbb{R}^{n}}|x| \,d\varrho_{k}(x)+\int_{\mathbb{R}^{n}}|x| \,d\varrho_0(x) \right)\end{align}
    Note that the sequence $\int_{\mathbb{R}^n}|x| \,d \varrho_k$ is bounded above by $M$. For each fixed $M_1>1$, after taking the upper limit as $k\rightarrow \infty$ in \eqref{Continuity-Of-Int-g},  we get $$
\limsup_{k\to\infty}\left|\int_{\mathbb{R}^n} g(x) \,d \varrho_k(x)-\int_{\mathbb{R}^n} g(x) \,d \varrho_0(x)\right| \leq \frac{1}{M_1^{(1-\alpha_{2})/\alpha_{2}}}\left( M+\int_{\mathbb{R}^{n}}|x|\,d\varrho_0(x) \right), 
$$  where  we have used the weak convergence of $\varrho_k$ to $\varrho_0.$ Now let $M_1 \to \infty$, we arrive at $$
\lim_{k\to\infty}\left|\int_{\mathbb{R}^n} g(x) \,d \varrho_k(x)-\int_{\mathbb{R}^n} g(x) \,d \varrho_0(x)\right| =0.
$$ That is, \eqref{limit-g-rho-k} holds.

    Finally, we prove that 
$$
\mathcal{G}_{\alpha}(\varrho):=\int_{\mathbb{R}^n}\left(-\rho(x)^{\frac{1}{1-\alpha}}+\beta g^{\frac{1}{\alpha}}(x)+g(x) \rho(x)\right) d x +\int_{\mathbb{R}^{n}}g(x)\,d\varrho^{s}(x)
$$
is lower semi-continuous. For any compact set $K$, let
$$
\mathcal{G}_{\alpha,K}(\varrho):= \int_{K}\left(-\rho(x)^{\frac{1}{1-\alpha}}+\beta g^{\frac{1}{\alpha}}(x)+g(x) \rho(x)\right) d x +\int_{K}g(x)\,d\varrho^{s}(x).
$$
By \eqref{Young'sIneq}, the function $-\rho(x)^{\frac{1}{1-\alpha}}+\beta g^{\frac{1}{\alpha}}(x)+g(x) \rho(x)$ is nonnegative. By the monotone convergence theorem, we have
    $$
\mathcal{G}_{\alpha}(\varrho) = \sup_{K\subset \mathbb{R}^{n},K \text{ is compact}} \mathcal{G}_{\alpha,K}(\varrho).
$$
It follows from \cite[Remark 3 on Page 413]{BV88} that   $\mathcal{G}_{\alpha,K}$ is lower semi-continuous in $\mathcal{P}_{1}(\mathbb{R}^{n})$. Therefore, $\mathcal{G}_{\alpha}$ is also lower semi-continuous in $\mathcal{P}_{1}(\mathbb{R}^{n})$ and thus, 
\begin{align*}
        \liminf _{k\to\infty}\mathcal{G}_{\alpha}(\varrho_{k})\geq \mathcal{G}_{\alpha}(\varrho_{0}).
    \end{align*} Together with \eqref{SplitF} and \eqref{limit-g-rho-k},   \eqref{LimInfOfF(RhoK)} is proved. As the sequence $\varrho_k$ is the limiting sequence of problem (\ref{P}) and $\varrho_0\in \mathcal{P}_1(\R^n)$, $\varrho_0$ is a solution to   problem (\ref{P}).   \end{proof}

    Now we prove that the solutions to our Minkowski problems, if exist, must have essentially continuous base functions. The idea of the proof originates from \cite[Theorem 4.3]{San16}.

\begin{proposition}\label{EssentialContinuity}
    Let $-\frac{1}{n}<\alpha<0$ and $\mu\in\mathcal{P}_1(\Rn)$ be such that  the barycenter of $\mu$ is $o$ and $\mu$ is not supported in any hyperplane. Suppose that $$\varrho_0 =\rho_{0}\,dx+\varrho_{0}^{s}= (1-\alpha\varphi_0)^{\frac{1}{\alpha}-1}\,dx + \varrho_{0}^{s},$$ with $\varphi_{0}\geq \frac{1}{\alpha}$, is the solution to (\ref{P}). Then, $\varphi_{0}$ is essentially continuous. In particular, the spherical surface area measure of $\varrho_{0}$ is identically $0$.
\end{proposition}

\begin{proof}
    Let $\Omega=\dom\varphi_{0}$. To prove that $\varphi_{0}$ is essentially continuous, i.e., continuous $\mathcal{H}^{n-1}$-a.e. on $\partial \Omega$, it is sufficient to show that $\rho_{0}=0$ $\mathcal{H}^{n-1}$-a.e. on $\partial \Omega$. Indeed, let $x_{0}\in \partial \Omega$. By virtue of lower semi-continuity of $\varphi_{0}$, we have
    $$\limsup_{x\to x_{0}}\rho_{0}(x) \leq \rho_{0}(x_{0}).$$
If $\rho_{0}(x_{0}) = 0$, then obviously $$\liminf_{x\to x_{0}}\rho_{0}(x)\geq 0 = \rho_{0}(x_{0}).$$ Hence $\rho_{0}$ and $\varphi$ are continuous at $x_{0}$.

For any $t\in \R$, define the sublevel set of $\varphi_0$ by $$[\varphi_0 ]_t=\{x\in \Rn: \varphi_0 (x)\leq   t\}.$$ 
Suppose that, for any $t\in \R$, the interior of $[\varphi_0]_t$ is empty. As $\varphi_0$ is lower semi-continuous, $[\varphi_0]_t$ is closed and  $\Omega=\cup_{t\in \mathbb{Z}}  [\varphi_0]_t$. Thus, $\Omega$ must have an empty interior by Baire category theorem.  Because $\Omega$ is convex, it should be contained in some hyperplane. Hence $\{\rho > 0\}$ is contained in some hyperplane, i.e. $\rho(x)=0$ for $\mathcal{L}^{n}$-a.e. $x$, which is a contradiction to the fact that $\varrho_{0}\neq \varrho_{0}^{s}$.

Without loss of generality, we can then assume that there exists $t_0$ such that $[\varphi_0]_{t_0}$ has a nonempty interior. Hence, for any $t>t_0$,  $\{\varphi<t\} $  must have non-empty interiors as well. Moreover, $\{\varphi<t\}$ must be  a convex open  set. 

Assume the opposite to our claim, namely,  there exist a constant $a>0$ and a set $S\subset \partial \Omega$ of positive $\mathcal{H}^{n-1}$- measure, such that, $\rho_{0}\geq a>0$   on $S$. Shrinking  $S$, if necessary, and  using the convexity of $\Omega,$ we see that, locally around some point $x_0\in S,$ the set $S$ can be  represented as the graph of a convex function. Without loss of generality, we can assume $x_0=o$ and  
 $$S = \{(x',h(x')):x'\in U\}, $$ for a convex function $h:\R^{n-1}\rightarrow [0,\infty)$, where $U$ is a neighborhood of $o\in \R^{n-1}$ and   $h(o)=0$.  
 Let $S_{\varepsilon}:=\{(x',h(x')+\delta):x'\in U,0< \delta\leq\varepsilon\}$. Shrinking $U$ if necessary, we can assume $S_{\varepsilon_{0}}\subset \text{int }\Omega$ for some small $\varepsilon_{0}$. There is   $S'\subset S+\varepsilon_{0}e_{n}$, which takes  the form $$S'=\{(x',h(x')+\varepsilon_{0}):x'\in V\}$$ with $V$  a relative open subset of $U$ such that $\overline{V}\subset U$. Hence,  $\varphi_0$ is bounded from above on $S'$.  Because $\{\varphi_{0}<t\}$ is open for $t$ large and $\varphi_{0}$ is finite on $S$, there exists $t'$ such that $\{\varphi_0 <t'\}\cap (S+\varepsilon_{0}e_{n})$ is a relative open subset of $S+\varepsilon_{0}e_{n}$. 
 
 By shrinking $S$ again, if necessary, we may assume that $U=V$ and $$S'=S+\varepsilon_{0}e_{n}\subset \{\varphi_{0}<t'\}.$$ Then, $\varphi_0$ is bounded from above on $S+\varepsilon_{0}e_{n}$. Due to the convexity of $\varphi_{0}$, the boundedness of $\varphi_{0}$ on $S$ and $S+\varepsilon_{0}e_{n}$ yield the boundedness of $\varphi_{0}$ on $S_{\varepsilon_{0}}$. Thus, there exists a constant $c_{0}>0$ such that  $\rho_{0} \geq c_{0}$ on $S_{\varepsilon}$ for any $0<\varepsilon<\varepsilon_{0}$.
  Therefore,
    \begin{equation} \int_{S_{\varepsilon}}\rho_{0}^{\frac{1}{1-\alpha}}\,dx\geq c_{0}^{\frac{1}{1-\alpha}} \int_{S_{\varepsilon}}\,dx = c_{0}^{\frac{1}{1-\alpha}} \int_{U}\int_{h(x')}^{h(x')+\varepsilon}\,dx_{n}\,dx' = c_{0}^{\frac{1}{1-\alpha}} \mathcal{L}^{n-1}(U)\varepsilon:=c_{1}\varepsilon.\label{LocalBoundOfRho}\end{equation} To find a contradiction, let $T^{\varepsilon}(x):=(x',x_{n}-\varepsilon)$ for $x=(x', x_n)$ and 
    $$
    \varrho_{\varepsilon}:= \varrho_{0}\mres{S_{\varepsilon}^{c}} + \frac{1}{2}\Big(\rho_{0}\,dx\mres{S_{\varepsilon}}\Big)+\frac{1}{2}\Big(T^{\varepsilon}_{\sharp}(\rho_{0}\,dx\mres{S_{\varepsilon}})\Big) +\varrho^s_{0}\mres{S_{\varepsilon}},
    $$ where $ S_{\varepsilon}^{c} =\Rn\setminus  S_{\varepsilon}$. Note that the measure $T^{\varepsilon}_{\sharp}(\rho_{0}\,dx\mres{S_{\varepsilon}})$ has its density function again equal to $\rho_0\big((T^{\varepsilon})^{-1}(\cdot)\big)$ on the set $T^{\varepsilon}(S_{\varepsilon})$, where $(T^{\varepsilon})^{-1}$ denotes the preimage of the map $T^{\varepsilon}$. Therefore, put $c'= \big( 2^{\frac{-\alpha}{1-\alpha}}-1\big),$ we have\begin{align*}\mathcal{F_{\alpha}}(\varrho_{\varepsilon})&= -\int_{S_{\varepsilon}^{c}}\rho_{0}^{\frac{1}{1-\alpha}}(x)\,dx-\! \int_{S_{\varepsilon}}\! 2^{-\frac{1}{1-\alpha}}\rho_{0}^{\frac{1}{1-\alpha}}(x)\,dx-\! \int_{T^{\varepsilon}(S_{\varepsilon})} \!\!  2^{-\frac{1}{1-\alpha}}\rho_0\big((T^{\varepsilon})^{-1}(x)\big)^{\frac{1}{1-\alpha}}\,dx\\
    &=-\int_{S_{\varepsilon}^{c}}\rho_{0}^{\frac{1}{1-\alpha}}(x)\,dx-\int_{S_{\varepsilon}}2^{\frac{-\alpha}{1-\alpha}}\rho_{0}^{\frac{1}{1-\alpha}}(x)\,dx\\
    &=-\int_{\Omega}\rho_{0}^{\frac{1}{1-\alpha}}(x)\,dx-\big( 2^{\frac{-\alpha}{1-\alpha}}-1\big) \int_{S_{\varepsilon}} \rho_{0}^{\frac{1}{1-\alpha}}(x)\,dx \\ 
    &=\mathcal{F_{\alpha}}(\varrho_{0})-c'\int_{S_{\varepsilon}} \rho_{0}^{\frac{1}{1-\alpha}}(x)\,dx \numberthis\label{EstimateF}.\end{align*}
    
     Next, we find an upper bound for $\mathcal{T}\left(\varrho_{\varepsilon}, \mu\right)$. Let $\psi$ be a fixed convex, positive function, such that $\int_{\Rn}\psi\,d\mu<\infty$ and $\psi^{*}$ is finite on $\Rn$. Such a function $\psi$ exists due to the fact that $\mu\in \mathcal{P}_{1}(\mathbb{R}^{n})$. Let $\delta>0$ be a positive number, and  $\varphi_{ \delta}:=(\varphi_{ 0}^{*}+\delta\psi)^{*}$. For $x\in S_{\varepsilon}$, it follows from \eqref{legendre-tran} that 
    \begin{align*}\varphi_{ \delta}\big(T^{\varepsilon}(x)\big) &=\sup_{z\in\Rn} \Big\{\langle z, T^{\varepsilon}(x)\rangle -\varphi_{ \delta}^*(z)\Big\} \\&=\sup_{z\in\Rn} \Big\{\langle z, x-\varepsilon e_n \rangle -\varphi_{ 0}^*(z)-\delta \psi(z)\Big\}\\&\leq \sup_{z\in\Rn} \Big\{\langle z, x \rangle -\varphi_{ 0}^*(z) \Big\} +\delta  \sup_{z\in\Rn} \Big\{\Big\langle z,  \frac{-\varepsilon e_n}{\delta} \Big\rangle  -\psi(z)\Big\}  \\
    &= \varphi_{ 0}(x)+\delta\psi^{*}\Big( \frac{-\varepsilon e_{n}}{\delta}\Big),\end{align*}
    where $\{e_{i}\}_{1\leq i\leq n}$ is the standard basis of $\mathbb{R}^{n}$. As the Legendre transform is order reversing, we have
     $\varphi_{ \delta}\leq \varphi_{ 0}$ in  $\Rn.$ This further implies  
    \begin{align*}\mathcal{T}(\varrho_{\varepsilon},\mu)&\leq \int_{\Rn}\varphi_{ \delta}\,d\varrho_{\varepsilon}+\int_{\Rn}(\varphi_{ 0} ^{*}+\delta\psi)\,d\mu\\
    &\leq\int_{S_{\varepsilon}^{c}}\!\varphi_{ 0} \,d\varrho_{0}+\frac{1}{2}\! \int_{S_{\varepsilon}}\!\varphi_{ 0} \rho_0\,dx +\frac{1}{2}\!\int_{T^{\varepsilon}(S_{\varepsilon})}\varphi_{ \delta} \rho_0\big((T^{\varepsilon})^{-1}(x)\big) \,dx \\ &\ \  +\int_{S_{\varepsilon}}\!\,d\varrho^s_{0}+  \int_{\Rn}\!\varphi^{*}_{0} \,d\mu+\delta\int_{\Rn}\! \psi \,d\mu\\
    &=\!\int_{\Rn}\! \varphi_{ 0} \,d\varrho_{0}\!+\!\int_{\Rn}\!\varphi_{ 0} ^{*}\,d\mu \!-\! \frac{1}{2}\!\int_{S^{\varepsilon}}\! \varphi_0 \,d\rho_{0}\!+\!\frac{1}{2}\!\int_{S_{\varepsilon}}\!\Big(\!\varphi_{\delta}\circ T^{\varepsilon}\!\Big)(x)\rho_{0}(x)\,dx \!+\!\delta\!\int_{\Rn}\! \psi \,d\mu\\
    &= \mathcal{T}(\rho_{0},\mu) +\delta\!\int_{\Rn}\! \psi \,d\mu - \frac{1}{2}\!\int_{S^{\varepsilon}}\! \varphi_0 \,d\rho_{0}+\frac{1}{2}\!\int_{S_{\varepsilon}}\!\Big(\!\varphi_{\delta}\circ T^{\varepsilon}\!\Big)(x)\rho_{0}(x)\,dx \\
    &\leq \mathcal{T}(\rho_{0},\mu)+\delta\int_{\Rn}\psi \,d\mu+\frac{\delta}{2} \psi^{*}\left( \frac{-\varepsilon e_{n}}{\delta}\right)\int_{S_{\varepsilon}}\rho_{0}(x)\,dx .\numberthis\label{EstimateT} 
    \end{align*} Combining (\ref{LocalBoundOfRho}), (\ref{EstimateF}), and (\ref{EstimateT})   with  the fact that $\varrho_{0}$ solves \eqref{P},   we get
    \begin{align*}
         (1\!-\!\alpha)\mathcal{F}_{\alpha}(\varrho_{0})\!-\! \alpha\mathcal{T}(\varrho_{0},\mu) 
        &\leq  (1\!-\!\alpha)\mathcal{F}_{\alpha}(\varrho_{\varepsilon})\!-\! \alpha\mathcal{T}(\varrho_{\varepsilon},\mu)\\
        & \leq\  (1\!-\!\alpha)\mathcal{F}_{\alpha}(\varrho_{0})-c' (1-\alpha)\int_{S_{\varepsilon}} \rho_0^{\frac{1}{1-\alpha}}(x) d x\\&\ \ -\alpha\bigg(\!\mathcal{T}(\rho_{0},\mu)\!+\!\delta\!\int_{\Rn}\! \psi \,d\mu \!+\! \frac{\delta}{2} \psi^{*}\Big(\! \frac{-\varepsilon e_{n}}{\delta}\!\Big)\!\int_{S_{\varepsilon}}\! \rho_{0}(x)\,dx\bigg)\\
        &  \leq\  (1-\alpha)\mathcal{F}_{\alpha}(\varrho_{0})-c'c_{1} (1-\alpha)\varepsilon\\&\ \ -\alpha\bigg(\!\mathcal{T}(\rho_{0},\mu)\!+\!\delta\!\int_{\Rn}\! \psi \,d\mu \!+\! \frac{\delta}{2} \psi^{*}\Big(\! \frac{-\varepsilon e_{n}}{\delta}\!\Big)\!\int_{S_{\varepsilon}}\! \rho_{0}(x)\,dx\bigg). 
    \end{align*}
    This estimate entails
    \begin{align}c'c_{1}  (1-\alpha)&\leq -\frac{\alpha\delta}{\varepsilon}\int_{\Rn}\psi \,d\mu-\frac{\alpha \delta}{2\varepsilon} \psi^{*}\Big( \frac{-\varepsilon e_{n}}{\delta}\Big)\int_{S_{\varepsilon}}\rho_{0}(x)\,dx \nonumber \\ &=- \alpha\beta \int_{\Rn}\psi \,d\mu-\frac{\alpha \beta}{2 } \psi^{*}\Big( \frac{- e_{n}}{\beta}\Big)\int_{S_{\varepsilon}}\rho_{0}(x)\,dx  \label{LastEstimate},\end{align} where $\delta$ is chosen to be $\delta=\beta\varepsilon$ for some positive constant $\beta$ so that  $$-\alpha \beta \int_{\Rn}\psi \,d\mu< c'c_1 (1-\alpha).$$  Let $\varepsilon\to 0$ in \eqref{LastEstimate}, then by monotone convergence theorem, 
    $$\int_{S_\varepsilon} \rho_0 (x) d x\to 0,$$
and we arrives at a contradiction. Thus, $\varphi_{0}$ is essentially continuous.
\end{proof}
 Our main result in this section is summarized in the following theorem. 
\begin{theorem}\label{sol-measure-measure}
Let $-\frac{1}{n}<\alpha<0$ and $\mu\in \mathcal{P}_{1}(\mathbb{R}^{n})$. Assume that  the barycenter of $\mu$ is $o$ and $\mu$ is not supported in any hyperplane. Then, there exists an $\alpha$-concave measure $$\bar\varrho = (1-\alpha\varphi_{0})^{\frac{1}{\alpha}}\,dx + \varrho^{s}_{0},$$ such that $\mu$ is the Euclidean surface area measure of $\bar{\varrho}$ and the spherical surface area measure of $\bar\varrho$ is the zero measure.
\end{theorem}

    \begin{proof} By Propositions \ref{PropertyOfMinimizer} and   \ref{Existence}, there exists a solution $$\varrho_{0} = \big(1-\alpha\varphi_{0}\big)^{\frac{1}{\alpha}-1}+\varrho_{0}^{s}$$ to the minimization problem \eqref{P} such that, $\varphi_{0}\geq \frac{1}{\alpha}$,   $\mathrm{Supp}(\varrho_{0}^{s})\subset \left\{\varphi_{0}=\frac{1}{\alpha}\right\}$, and 
        $$\mathcal{T}(\varrho_{0},\mu)=\int_{\mathbb{R}^{n}}\varphi_{0}\, \,d\varrho_{0}+\int_{\mathbb{R}^{n}}\varphi_{0}^*\, \,d\mu.$$
Due to Theorem \ref{[Vil08, Theorem 5.10]}, there exists a measure $\pi$ on $\mathbb{R}^{n}\times \mathbb{R}^{n}$, whose marginals are $\varrho_{0}$ and $\mu$, respectively, such that 
$$\mathrm{Supp}(\pi)\subset \mathrm{Graph}(\partial \varphi). $$
Therefore, $\mu$ is the Euclidean surface area measure of $\bar{\varrho}$. On the other hand, Proposition \ref{EssentialContinuity} asserts that the spherical surface area measure of $\bar\varrho$ vanishes. This completes the proof.
\end{proof}

Note that, if $\inf\varphi>\frac{1}{\alpha},$ Theorem \ref{sol-measure-measure} provides a solution to Problem \ref{problem-E-FMP}. However, it is not clear when the solution $\varphi$ given in Theorem \ref{sol-measure-measure} does  satisfy $\inf \varphi >\frac{1}{\alpha}$, and we leave this investigation for future studies.

\vskip 2mm \noindent  {\bf Acknowledgement.}  
The research of XL has been supported by the Science and Technology Research Program of Chongqing Municipal Education Commission (No. KJQN202300557) and the Research Foundation of Chongqing Normal University (No. 20XLB012). The
research of DY has been supported by a NSERC grant, Canada.


\begin{thebibliography}{99}
\bibitem{Ale38} Aleksandrov, A. D.: \textit{Zur Theorie der gemischten Volumina von konvexen K\" orpern, III: Die Erweiterung zweier Lehrs\"atze Minkowskis \"uber die konvexen Polyeder auf beliebige konvexe Fl\"achen} (in Russian). Mat. Sbornik N. S. \textbf{3}, 27-46 (1938).
\bibitem{AAGCJV19} Alonso-Gutiérrez, D., Artstein-Avidan, S., González Merino, B., Jiménez, C. H., Villa, R.: \textit{Rogers-Shephard and local Loomis–Whitney type inequalities.} Math. Ann. \textbf{374}, 1719-1771 (2019).
\bibitem{AMJV16} Alonso-Gutiérrez, D., González Merino, B., Jiménez, C. H., Villa, R.: \textit{Rogers–Shephard inequality for log-concave functions.} J. Funct. Anal. \textbf{271}, 3269-3299 (2016).
\bibitem{AMJV17}  Alonso-Gutiérrez, D., Merino, B.G., Jiménez, C. H., Villa, R.: \textit{John’s Ellipsoid and the Integral Ratio of a Log-Concave Function.} J. Geom. Anal. \textbf{28}, 1182-1201 (2018). 

\bibitem{AKM04} Artstein-Avidan, S., Klartag, B., Milman, V.: \textit{The Santaló point of a function, and a functional form of the Santaló inequality.} Mathematika. \textbf{51}, 33-48 (2004).
\bibitem{AKSW12} Artstein-Avidan, S., Klartag, B., Schütt, C., Werner, E.: \textit{Functional affine-isoperimetry and an inverse logarithmic Sobolev inequality.} J. Funct. Anal. \textbf{262}, 4181-4204 (2012).

\bibitem{A72} Avriel, M.: {\em $r$-convex functions.}  Math. Program. {\bf 2}(1), 309-323
(1972).


\bibitem{BBCY19} Bianchi, G., B\"or\"oczky, K. J., Colesanti, A., Yang, D.: \textit{The $L_{p}$-Minkowski problem for $-n<p<1$}. Adv. Math.  \textbf{341}, 493-535 (2019).
\bibitem{B74} Borell, C.: \textit{Convex measures on locally convex spaces.} Ark. Mat. \textbf{12}(1-2), 239-252 (1974).
\bibitem{B75} Borell, C.: \textit{Convex set functions in $d$-space}. Period. Math. Hungar. \textbf{6}(2), 111-136 (1975).

\bibitem{BLYZ13} B\"or\"oczky, K. J., Lutwak, E., Yang, D., Zhang, G.: \textit{The logarithmic Minkowski problem.} J. Amer. Math. Soc. \textbf{26}, 831-852 (2013).

\bibitem{BLYZZ20}
B\"or\"oczky, K. J., Lutwak, E., Yang, D., Zhang, G., Zhao, Y.: \textit{The Gauss Image Problem.} Comm. Pure Appl. Math. \textbf{73}, 1406-1452 (2020).

\bibitem{BV88} Bouchitté, G., Valadier, M.: Integral representation of convex functionals on a space of measures. J. Funct. Anal. \textbf{80}, 398-420 (1988).
\bibitem{BL76} Brascamp, H. J., Lieb, E. J.: \textit{On extensions of the Brunn-Minkowski and Pr\'ekopa-Leindler theorems, including inequalities for log concave functions, and with an application
to the diffusion equation.} J. Funct. Anal. \textbf{22}(4), 366-389 (1976).

\bibitem{Bre91} Brenier, Y.: \textit{Polar factorization and monotone rearrangement of vector-valued functions.} Comm. Pure Appl. Math.  \textbf{44}, 375-417 (1991).


\bibitem{CFGLSW16}  Caglar, U., Fradelizi, M., Guédon, O., Lehec, J., Schütt, C., Werner, E. M.: \textit{Functional Versions of $L_{p}$-Affine Surface Area and Entropy Inequalities.} Int. Math. Res. Not.  \textbf{4}, 1223-1250 (2016).

\bibitem{CW14} Caglar, U., Werner, E. M.: \textit{Divergence for \textit{s}-concave and log concave functions.} Adv. Math. \textbf{257}, 219-247 (2014). 

\bibitem{CW15}  Caglar, U., Werner, E. M.: \textit{Mixed f-divergence and inequalities for log-concave functions.} Proc. Lond. Math. Soc. \textbf{110}, 271-290 (2015). 




\bibitem{CLZ19} Chen, S., Li, Q., Zhu, G.: \textit{The logarithmic Minkowski problem for non-symmetric measures.} Trans. Amer. Math. Soc. \textbf{371}, 2623-2641 (2019).
\bibitem{CF13}  Colesanti, A., Fragal\`a, I.: \textit{The first variation of the total mass of log-concave functions and related inequalities.} Adv. Math.  \textbf{244}, 708-749 (2013). 

\bibitem{EK15}Cordero-Erausquin, D., Klartag, B.: \textit{Moment measures.} J. Funct. Anal. \textbf{268}, 3834-3866 (2015). 

\bibitem{Dur19} Durrett, R.: \textit{Probability: Theory and Examples.} Cambridge University Press, Cambridge, (2019).

\bibitem{FR25} Falah, T., Rotem, L.: \textit{On the functional Minkowski problem.} \hyperlink{http://arxiv.org/abs/2502.16929}{http://arxiv.org/abs/2502.16929}, (2025).

\bibitem{FXY20+} Fang, N., Xing,  S., Ye, D.: {\it Geometry of log-concave functions: the $L_p$ Asplund sum and the $L_{p}$ Minkowski problem},  Calc. Var. Partial Differential Equations {\bf 61}, 2 (2022).

\bibitem{FYZ24}  Fang, N., Ye, D., Zhang, Z.: \textit{The Riesz $\alpha$-energy of log-concave functions and related Minkowski problem.} \hyperlink{http://arxiv.org/abs/2408.16141}{http://arxiv.org/abs/2408.16141}, (2024).

\bibitem{FYZZ25}  Fang, N., Ye, D., Zhang, Z., Zhao, Y.: \textit{Dual Orlicz curvature measures for log-concave functions and their Minkowski problems.} Calc. Var. Partial Differential Equations  \textbf{64}, 44 (2025). 


\bibitem{FG21}  Figalli, A., Glaudo, F.: \textit{An Invitation to Optimal Transport, Wasserstein Distances, and Gradient Flows.} EMS Press, Berlin, (2021).
\bibitem{Fol99}Folland, G. B.: \textit{Real Analysis: Modern Techniques and Their Applications.} John Wiley \& Sons, Inc., New York,  (1999).

\bibitem{FM06} Fradelizi, M., Meyer, M.: \textit{Some functional forms of Blaschke-Santaló-inequality.} Math. Z. \textbf{256}, 379-395 (2007).
\bibitem{FM08} Fradelizi, M., Meyer, M.: \textit{Some functional inverse Santaló inequalities.} Adv. Math. \textbf{218}, 1430-1452 (2008).


\bibitem{GM96} Gangbo, W., McCann, R. J.: The geometry of optimal transportation. Acta. Math. \textbf{177}, 113-161 (1996). 

\bibitem{GHWXY19} Gardner, R. J., Hug, D., Weil, W., Xing, S., Ye, D.: \textit{General volumes in the Orlicz-Brunn–Minkowski theory and a related Minkowski problem I.} Calc. Var. Partial Differential Equations \textbf{58}, 12 (2018).




\bibitem{HLXZ24}  Huang, Y., Liu, J., Xi, D., Zhao, Y.: \textit{Dual curvature measures for log-concave functions.} J. Differential Geom. \textbf{128}, 815-860 (2024). 
\bibitem{HLYZ16}  Huang, Y., Lutwak, E., Yang, D., Zhang, G.: \textit{Geometric measures in the dual Brunn-Minkowski theory and their associated Minkowski problems.} Acta. Math. \textbf{216}, 325-388 (2016). 

\bibitem{HLYZ18} Huang, Y., Lutwak, E., Yang, D., Zhang, G.: \textit{The $L_{p}$-Aleksandrov problem for $L_{p}$-integral curvature.} J. Differential Geom. \textbf{110}, 1-29 (2018).

\bibitem{HXZ21} Huang, Y., Xi, D., Zhao, Y.: \textit{The Minkowski problem in Gaussian probability space.} Adv. Math.  \textbf{385}, 107769 (2021).

\bibitem{HK21} Huynh, K., Santambrogio, F.: \textit{$q$-moment measures and applications: a new approach via optimal transport}. Journal of Convex Analysis. \textbf{28} (4), 1033-1052 (2021).

\bibitem{IN22}Ivanov, G., Naszódi, M.: \textit{Functional John ellipsoids.} J. Funct. Anal. \textbf{282}, 109441 (2022).

\bibitem{KM05}Klartag, B., Milman, V. D.: \textit{Geometry of Log-concave Functions and Measures.} Geom. Dedicata. \textbf{112}, 169-182 (2005).

\bibitem{KS84} Knott, M., Smith, C. S.: \textit{On the optimal mapping of distributions.} J. Optim. Theory Appl. \textbf{43}, 39-49 (1984). 

\bibitem{Leh09}  Lehec, J.: \textit{Partitions and functional Santaló inequalities.} Arch. Math. (Basel) \textbf{92}, 89-94 (2009).

\bibitem{LiSW19}  Li, B., Schütt, C., Werner, E. M.: \textit{Floating functions.} Israel J. Math. \textbf{231}, 181-210 (2019).
\bibitem{LiSW19-2}  Li, B., Schütt, C., Werner, E. M.: The Löwner Function of a Log-Concave Function. J. Geom. Anal. \textbf{31}, 423-456 (2021). 

\bibitem{Lut93} Lutwak, E.: \textit{The Brunn-Minkowski-Firey theory. I. Mixed volumes and the Minkowski problem.} J. Differential Geom. \textbf{38}, 131-150 (1993). 
\bibitem{LXYZ23} Lutwak, E., Xi, D., Yang, D., Zhang, G.: \textit{Chord measures in integral geometry and their Minkowski problems. }Comm. Pure Appl. Math.  \textbf{77}, 3277-3330 (2024). 
\bibitem{Mc01} McCann, R. J.: \textit{Polar factorization of maps on Riemannian manifolds.}  Geom. Funct. Anal. \textbf{11}, 589-608 (2001).


\bibitem{MR13}  Milman, V., Rotem, L.:\textit{ $\alpha$-concave functions and a functional extension of mixed volumes.} Electron. Res. Announc. Math. Sci. \textbf{20}, 1-11 (2013).

\bibitem{Min1897} Minkowski, H.: \textit{Allgemeine Lehrsätze über die konvexen Polyeder.} In: Minkowski, H. (ed.) Ausgewählte Arbeiten zur Zahlentheorie und zur Geometrie: Mit D. Hilberts Gedächtnisrede auf H. Minkowski, Göttingen 1909.  121-139. Springer, Vienna (1989). 

\bibitem{Min1903}  Minkowski, H.: 
\textit{Volumen und Oberfl\"{a}che}, Math. Ann. \textbf{57}, 447-495  (1903).



\bibitem{Roc70}  Rockafellar, R. T.: \textit{Convex Analysis.} Princeton University Press, Princeton,  (1970).
\bibitem{Rot13}  Rotem, L.: \textit{Support functions and mean width for $\alpha$-concave functions.} Adv. Math.  \textbf{243}, 168-186 (2013). 

\bibitem{Rot22} Rotem, L.: \textit{Surface area measures of log-concave functions.} J. Anal. Math. \textbf{147}, 373-400 (2022).
\bibitem{Rot23} Rotem, L.: \textit{The Anisotropic Total Variation and Surface Area Measures.} In: Eldan, R., Klartag, B., Litvak, A., and Milman, E. (eds.) Geometric Aspects of Functional Analysis: Israel Seminar (GAFA) 2020-2022,  Springer, Cham, 297-312 (2023).


\bibitem{RX23} Roysdon, M., Xing, S.: \textit{On the framework of Lp summations for functions.} J. Funct. Anal. \textbf{285}, 110150 (2023).

\bibitem{San15}  Santambrogio, F.: \textit{Optimal Transport for Applied Mathematicians: Calculus of Variations, PDEs, and Modeling.} Birkhäuser/Springer, Cham, (2015).
\bibitem{San16}  Santambrogio, F.: \textit{Dealing with moment measures via entropy and optimal transport.} J. Funct. Anal. \textbf{271}, 418-436 (2016). 
\bibitem{Sch13}Schneider, R.: {\it Convex Bodies: The Brunn-Minkowski Theory}, Cambridge University Press,
Cambridge, (2014).
\bibitem{Uli24} Ulivelli, J.: \textit{First variation of functional Wulff shapes}, \hyperlink{http://arxiv.org/abs/2312.11172}{http://arxiv.org/abs/2312.11172}, (2024).

\bibitem{Vil03}Villani, C.: \textit{Topics in Optimal Transportation.} American Mathematical Society, Providence, RI, (2003).
\bibitem{Vil08}Villani, C.: \textit{Optimal Transport: Old and New.} Springer-Verlag, Berlin, (2009).


\bibitem{Zha17} Zhao, Y.: \textit{The dual Minkowski problem for negative indices.} Calc. Var. Partial Differential Equations \textbf{56}, 18 (2017). 

\bibitem{Zhu15} Zhu, G.: \textit{The $L_{p}$-Minkowski problem for polytopes for $0<p<1$.} J. Funct. Anal. \textbf{269}, 1070-1094 (2015). 


\end{thebibliography}
\end{document}